\documentclass[11pt,a4paper,twoside]{amsart}
\usepackage{amssymb,amsmath,amsthm,graphicx,color}
\theoremstyle{plain}
\newtheorem{theorem}{Theorem}[section]

\newtheorem{definition}[theorem]{Definition}
\newtheorem{lemma}[theorem]{Lemma}
\newtheorem{corollary}[theorem]{Corollary}
\newtheorem{proposition}[theorem]{Proposition}

\theoremstyle{remark}
\newtheorem{remark}[theorem]{Remark}

\usepackage{hyperref}
\allowdisplaybreaks[4]
\numberwithin{equation}{section}

\newcommand{\C}{\mathbb{C}}
\newcommand{\R}{\mathbb{R}}
\newcommand{\Z}{\mathbb{Z}}

\newcommand{\F}{\mathcal{F}}

\renewcommand{\Re}{\operatorname{Re}}
\newcommand{\I}{\infty}
\newcommand{\abs}[1]{\left\lvert #1\right\rvert}
\newcommand{\Jbr}[1]{\left\langle #1 \right\rangle}
\newcommand{\norm}[1]{\left\lVert #1\right\rVert}

\newcommand{\hMn}[4]{\left\lVert #1 \right\rVert_{\hat{M}^{#2}_{{#3},{#4}}}}
\newcommand{\hM}[3]{\hat{M}^{#1}_{{#2},{#3}}}
\newcommand{\hLebn}[2]{\left\lVert #1 \right\rVert_{\hat{L}^{#2}}}
\newcommand{\dSobn}[2]{\left\lVert #1 \right\rVert_{\dot{H}^{#2}}}

\newcommand{\IN}{\quad\text{in }}
\newcommand{\wIN}{\quad\text{weakly in }}

\def\({\left(}
\def\){\right)}
\def\<{\left\langle}
\def\>{\right\rangle}
\def\le{\leqslant}
\def\ge{\geqslant}

\def\d{{\partial}}

\newcommand{\eps}{\varepsilon}

\DeclareMathOperator{\supp}{supp}

\DeclareMathOperator{\dist}{dist}
\newcommand{\rre}{\mathbb{R}}
\newcommand{\pt}{\partial}

\begin{document}
\title[On mass-subcritical generalized KdV equation]
{Existence of a minimal non-scattering\\ 
solution to the mass-subcritical\\
generalized Korteweg-de Vries
equation}
\author[S.Masaki and J.Segata]
{Satoshi Masaki  
and Jun-ichi Segata}
%\author[S.Masaki and J.Segata]
%{Satoshi Masaki${}^{\dagger}$ 
%and Jun-ichi Segata${}^{\ddagger}$}
\address{Laboratory of Mathematics\\
Institute of Engineering\\
Hiroshima University\\
Higashihiroshima Hirhosima, 739-8527, Japan}
\email{masaki@amath.hiroshima-u.ac.jp}

\address{Mathematical Institute, Tohoku University\\
6-3, Aoba, Aramaki, Aoba-ku, Sendai 980-8578, Japan}
\email{segata@m.tohoku.ac.jp}



\maketitle
\vskip-15mm
\begin{abstract}
In this article, we prove existence of a 
non-scattering solution, which is minimal in some sense, to the mass-subcritical
generalized Korteweg-de Vries (gKdV) equation
in the scale critical $\hat{L}^r$ space where 
$\hat{L}^r=\{f\in{{\mathcal S}}'(\rre)| \norm{f}_{\hat{L}^r}=\|\hat{f}\|_{L^{r'}}<\infty\}$. 
We construct this solution by a concentration compactness argument.
Then, key ingredients are a linear profile decomposition 
result adopted to $\hat{L}^r$-framework and
 approximation of solutions 
to the gKdV equation which involves rapid linear oscillation
by means of solutions to the nonlinear Schr\"{o}dinger equation.
\end{abstract}

%\tableofcontents

\section{Introduction}

In this article, we consider 
generalized Korteweg-de Vries (gKdV) equation
\begin{equation}\tag{gKdV}\label{gKdV}
\left\{
\begin{aligned}
&\pt_t u+\pt_x^3 u=\mu\pt_x(|u|^{2\alpha}u),
&& t,x\in\rre,\\
&u(0,x)=u_{0}(x) \in \hat{L}^\alpha(\R),
&&  x\in\rre
\end{aligned}
\right.
\end{equation}
where $u:\rre\times\rre\to\rre$ is an unknown function, 
$u_{0}:\rre\to\rre$ is a given data, 
and $\mu=\pm1$ and $\alpha>0$ are constants.
The space $\hat{L}^r$ is defined for $1\le r \le \infty$ by 
\begin{eqnarray*}
	\hat{L}^r=\hat{L}^r(\rre):=\{f\in{{\mathcal S}}'(\rre)|
	\norm{f}_{\hat{L}^r} =\|\hat{f}\|_{L^{r'}}<\infty\},
\end{eqnarray*}
where $\hat{f}=\F f$ stands for Fourier transform of $f$ with respect to space variable
and $r'=(1-1/r)^{-1}$ denotes the H\"older conjugate of $r$
with conventions $1'=\I$ and $\I'=1$. 
We call that (\ref{gKdV}) is defocusing if 
$\mu=+1$ and focusing if $\mu=-1$. 
Our aim here is to study time global behavior of solutions to \eqref{gKdV}
with focusing nonlinearities in the \emph{mass-subcritical} range $\alpha<2$.
More specifically, we investigate existence of a threshold solution which lies 
on the boundary of small scattering solutions around zero and other solutions.

The class of equations \eqref{gKdV} arises in several fields of physics. 
Equation \eqref{gKdV} is a generalization
of the Korteweg-de Vries equation which models long waves 
propagating in a channel \cite{KV}.  
Equation (\ref{gKdV}) with $\alpha=1$ is also known as the 
modified Korteweg-de Vries equation which 
describes a time evolution for the  curvature of 
certain types of helical space curves \cite{L}. 

The equation (\ref{gKdV}) has the following scale property; 
if $u(t,x)$ is a solution to (\ref{gKdV}), then
\[
	u_{\lambda}(t,x):=\lambda^{\frac{1}{\alpha}} u(\lambda^3 t, \lambda x)
\]
is also a solution to (\ref{gKdV}) with a initial data $u_{\lambda}(0,x)=
\lambda^{\frac{1}{\alpha}} u_{0}(\lambda x)$ for any $\lambda>0$.
When $\alpha=2$, \eqref{gKdV} is called mass-critical
because the above scale leaves the mass invariant.

The small data global existence results of \eqref{gKdV} in scale critical spaces
have been studied by several authors. 
Kenig-Pone-Vega \cite{KPV2} proved the small data global well-posedness 
and scattering of (\ref{gKdV}) 
in the scale critical space $\dot{H}^{s_{\alpha}}$ 
for $\alpha\ge2$, where $s_{\alpha}:=1/2-1/\alpha$ 
is a scale critical exponent. 
Since the scale critical exponent $s_{\alpha}$ is
negative in the mass-subcritical case $\alpha<2$, 
well-posedness of (\ref{gKdV}) in $\dot{H}^{s_{\alpha}}$ 
becomes rather a difficult problem. 
Tao \cite{T} proved 
global well-posedness for small data for (\ref{gKdV}) 
with the quartic nonlinearity $\mu\pt_{x}(u^{4})$ 
in $\dot{H}^{s_{3/2}}$.  
Later on, Koch-Marzuola \cite{KM} simplified Tao's 
proof and extended his result to a Besov space 
$\dot{B}^{s_{3/2}}_{\infty,2}$. 
As for the $\hat{L}^r$-framework, 
Gr\"{u}nrock and his collaborator
proved well-posedness for various nonlinear 
dispersive equations, see \cite{G1,G2,GV}. 

On the other hand, the asymptotic behavior in time of 
solution to (\ref{gKdV}) is studied for the small initial data 
in weighted Sobolev spaces \cite{St,SSS,PV,CW,HN1,HN2,HN3}. 
It is known that $\alpha=1$ is a critical exponent for scattering 
problem of (\ref{gKdV}). 
More precisely, if $\alpha>1$, then solution to (\ref{gKdV}) 
converges to solution to Airy equation $\pt_{t}v+\pt_{x}^{3}v=0$ (see \cite{HN1}) 
and if $0<\alpha\le1$, then solution to (\ref{gKdV}) does not 
converge to solution to Airy equation (see \cite{R,HN4}). 
%(see \cite{R} for the case $0<\alpha\le1/2$ 
%and \cite{HN4} for the case $\alpha=1$). 
Furthermore,  
for the case $\alpha=1$, Hayashi-Naumkin \cite{HN2,HN3} proved a existence 
of modified scattering states for (\ref{gKdV}). 
Note that for the case $\alpha=1$,  (\ref{gKdV}) is completely 
integrable and the inverse scattering method is available. 
By using the inverse scattering method, Deift-Zhou \cite{DZ} 
obtained more precise asymptotic behavior in time of solution to 
(\ref{gKdV}) with $\alpha=1$. 
%As far as we know, 
%there is no result on the scattering problem for large initial data 
%except \cite{DZ}.

The well-posedness of \eqref{gKdV} and  small data scattering  in $\hat{L}^\alpha$
is established by the authors as long as
$8/5 < \alpha < 10/3$ by introducing a generalized version of Stichartz's estimates
adopted to the $\hat{L}^r$-framework, see \cite{MS}.
The mass
\[
	M[u] = \frac12 \norm{u}_{L^2}^2
\]
and the energy
\[
	E[u] = \frac12 \norm{\d_x u}_{L^2}^2 + \frac{\mu}{2\alpha+2} \norm{u}_{L^{2\alpha+2}}^{2\alpha+2}
\]
are well-known conserved quantities for \eqref{gKdV}.
However, neither makes sense in general for $\hat{L}^\alpha$-solutions.
Thus, global existence is nontrivial for large data
even in the mass-subcritical range $\alpha <2$.

As a step next to small data scattering, in this article, we consider existence of 
a threshold solution
which lies on the boundary of small scattering solutions around zero and other solutions,
via concentration compactness argument.
Let us make our setup more precise.
We say an $\hat{L}^{\alpha}$-solution $u(t)$ scatters forward in time
(resp.\ backward in time)
if maximal existence interval of $u(t)$ is not bounded from above (resp. from below)
and if $e^{t\d_x^3} u(t)$ converges in $\hat{L}^{\alpha}$ as $t\to\I$ (resp. 
$t\to-\I$).
We define a \emph{forward scattering set} $\mathcal{S}_+$ as follows
\[
	\mathcal{S}_{+} :=
 	\left\{ u_0 \in \hat{L}^{\alpha} \left|
 	\begin{aligned}
	&\text{a solution }u(t) \text{ to }\eqref{gKdV}\text{ with }u_{|t=0} = u_0 \\
	&\text{scatters forward in time}
	\end{aligned}
	\right.\right\}.
\]
A \emph{backward scattering set} $\mathcal{S}_{-}$ is defined in a similar way.
We now introduce a quantity
\begin{equation}\label{eq:wided}
	\widetilde{d} := \inf_{u_0 \in \hat{L}^\alpha \setminus \mathcal{S}_+} \norm{u_0}_{\hat{L}^\alpha} .
\end{equation}
The question we address in this article is that
existence of a special solution
which belongs to $\hat{L}^\alpha \setminus \mathcal{S}_+$ at each time
and
attains $\widetilde{d}$ in a suitable sense.
By small data scattering result in \cite{MS}, we know that
% a ball of $\hat{L}^\alpha$ centered at the origin is contained in $\mathcal{S}_+ \cap \mathcal{S}_-$.
$\widetilde{d}$ is bounded by a positive constant from below.
Remark that there are several choice on notion of minimality of non-scattering solutions
since $\norm{u(t)}_{\hat{L}^\alpha} $ is not a conserved quantity.
The above $\widetilde{d}$ is a number that gives a sharp scattering criterion; if $\norm{u_0}_{\hat{L}^\alpha} 
< \widetilde{d}$ then a corresponding solution scatters for positive time direction.
However, we actually work with a weaker formulation by some technical reason
(see \eqref{eq:d_+}, below).

The above problem has a connection with stability of solitons.
In the focusing case (i.e., $\mu=-1$), 
\eqref{gKdV} admits a soliton solution
\[
	Q_c(t,x) = c^{\frac1\alpha} Q(c(x-c^2t)),
\]
where $Q(x)$ is a (unique) positive even 
solution of $- Q^{\prime\prime} + Q = Q^{2\alpha+1}$ and
$c>0$ is a parameter describing amplitude and propagating speed of soliton.
Let us remind ourselves that we consider the mass-subcritical problem.
It is well known that $Q$ is orbitally stable if $\alpha <2$ \cite{Be,W2} 
and unstable if $\alpha \ge 2$ (see \cite{BSS} for $\alpha>2$ and \cite{MM1} for $\alpha=2$).
When the soliton solutions are unstable, for example in the mass-critical case $\alpha =2$,
it is conjectured that the above $\tilde{d}$ coincide with $L^2$-norm (since $\alpha =2$) of 
$Q_c$. 
So far, it is known that 
if $\alpha=2$ then $Q$
lies on the boundary of sets of {\it global} solutions 
and non-global solutions in $H^1$, see 
Weinstein \cite{W1} for the sharp global existence result 
and Martel-Merle \cite{MM2} for the existence of 
a finite time blow up solution. 
%the above is true for $H^1$-solutions by means of sharp %Gagliardo-Nirenberg inequality (see \cite{W2}).

On the other hand, in mass-subcritical case, solitons are stable (in $H^1$) and so
they are not thresholds any longer.
Indeed, it follows from \cite[Theorem 1.10]{MS} that $\widetilde{d} \le c_\alpha \norm{Q}_{\hat{L}^\alpha}$,
where 
\begin{equation}\label{eq:c_a}
c_\alpha =\(\frac{(\alpha+1) \norm{Q'}_{L^2}^2}{\norm{Q}_{L^{2\alpha+2}}^{2\alpha+2}}\)^{\frac{1}{2\alpha}} <1
\end{equation}
is a constant such that $E[c_\alpha Q] =0$.

Recently, there are much progress on analysis of global behavior 
of dispersive equations by so-called concentration compactness/rigidity argument,
after a pioneering work by Kenig and Merle \cite{KMe}.
The existence of a critical element is one of the main step of the argument.
As for generalized KdV equation \eqref{gKdV}, the mass-critical case is 
most extensively studied in this direction. 
Killip-Kwon-Shao-Visan \cite{KKSV} constructed 
a minimal blow-up solution to the 
mass critical KdV equation in $L^{2}$ under the assumption 
on the space time 
bounds for the one dimensional mass-critical Schr\"{o}dinger (NLS) equation. 
Subsequently, Dodson \cite{D1} proved the global well-posedness for 
the one dimensional, defocusing, mass-critical NLS in $L^{2}$. 
As by product of his result, 
the assumption imposed in \cite{KKSV} was removed for the defocusing case. 
Furthermore, 
Dodson \cite{D2} 
has shown the global well-posedness for the defocusing mass-critical KdV 
equation for any initial data in $L^{2}$. 
For the focusing mass-critical KdV 
equation, 
Martel-Merle-Rapha\"{e}l \cite{MMR1,MMR2,MMR3} and 
Martel-Merle-Nakanishi-Rapha\"{e}l \cite{MMNR} classified the dynamics of 
solution into three cases (blow-up, soliton, away from soliton) 
in the small neighborhood of $Q$.
As for the mass-subcritical nonlinear Schr\"dinger equation,
the first author treated a minimization problem similar to \eqref{eq:wided}
in a framework of weighted space and showed existence of
a threshold solution which is smaller than ground state solutions (see \cite{M1,M2}).
%
%

% In this article, we establish
A main contribution of the this article is to extend
the concentration compactness argument to $\hat{L}^\alpha$-framework.
We then come across two difficulties because 
of the fact that the $\hat{L}^{\alpha}$-norm 
is invariant under the following four group actions;
\begin{enumerate}
\item Translation in physical space: $(T(a) f)(x) = f(x-a)$, $\quad a\in \R$,
\item Translation in Fourier space: $(P(\xi) f)(x) = e^{-ix\xi} f(x)$, $\quad \xi\in\R$,
\item Dilation: $(D(h)f)(x)=(D_\alpha(h)f)(x)=h^{1/\alpha} f(h x)$, $\quad h\in 2^\Z$,
\item Airy flow: $(A(t) f)(x) = e^{-t\d_x^3}f(x)$, $\quad t\in \R$.
\end{enumerate}
They are one parameter groups of linear isometries in $\hat{L}^{\alpha}$.
In this article, we call a bijective linear isometry from a Banach space $X$ to $X$ itself
a \emph{deformation on $X$}.
Further,
we refer to a deformation 
of the form $\times \phi(x)$ as a \emph{phase-like deformation}, and a deformation
of the form $\phi((1/i)\d_x) = \F^{-1} \phi(\xi) \F$ as a \emph{multiplier-like deformation},
where $\phi(x):\R \to \C$ is some function with $|\phi| =1 $.
With these terminologies,
$ T(a) = e^{-a \d_x}$ and $A(t)$ are multiplier-like deformations on $\hat{L}^\alpha$ and
$P(\xi)$ is a phase-like deformation on $\hat{L}^\alpha$.
% Note that, as long as $\alpha\neq2$,
% the critical Lebesgue space $L^{\alpha}$, 
% the critical Sobolev space $\dot{H}^{s_\alpha}$, and
% the critical weighted Lebesgue space $L^2(\R, |x|^{-2s_\alpha}dx)$ do not possess
% at least one of $T(a)$, $P(\xi)$, and $A(t)$ deformations.
% Inclusion relations between these spaces are summarized in \cite[Appendix B]{MS}.

% P の対称性の話あたり? 
%
% From view point of the symmetry, $\hat{L}^{\alpha}$ is close to $L^2$ space,
% which possesses all above symmetries when $\alpha=2$.
% Hence, our case can be compared with mass-critical case in \cite{KKSV}.
%
% Technically, we use many ideas and arguments which are used in the
% mass-critical problem, in fact.

\vskip2mm

The first difficulty lies in a linear profile decomposition,
which  is roughly speaking a decomposition of 
a bounded sequence of functions into a sum of characteristic profiles and a remainder by
finding weak limit(s) of the sequence modulo deformations.
Intuitively, this decomposition is done by a recursive use of a suitable concentration compactness
result.
Then, to ensure smallness of remainder as the number of detected profiles increases,
a decoupling equality, so-called Pythagorean decomposition, plays a crucial role.

Let us now be more precise on the Pythagorean decomposition.
Let $\{ f_n \}$ be a bounded sequence of $\hat{L}^\alpha$.
Since $\hat{L}^\alpha$ is reflexive as long as $1< \alpha <\I$,
by extracting subsequence,  $f_n$ converges 
to some function $f \in \hat{L}^\alpha$ in weak $\hat{L}^\alpha$ sense.
Now we suppose that $f\neq 0$.
Then, the Pythagorean decomposition is a decoupling equality of the form
\begin{equation}\label{eq:decouple}
	\|\hat{f_n}\|_{L^{\alpha'}(\R)}^{\alpha'} = \|\hat{f}\|_{L^{\alpha'}(\R)}^{\alpha'} + 
	\|\hat{f}- \hat{f_n}\|_{L^{\alpha'}(\R)}^{\alpha'}  + o(1)
\end{equation}
as $n\to\I$, 
It is well-known that the above decoupling holds for $\alpha = 2$ and
may fail for $\alpha\neq 2$. 
% as the following example shows;
% let $f_n = f  + T(n) g$ with $f,g \in \hat{L}^\alpha$
Remark that the Brezis-Lieb lemma tells us 
that a sufficient condition for the decoupling (for $\alpha \neq2$)
is that $\hat{f_n}$ converges to $\hat{f}$ almost everywhere.
However, in our case, due to multiplier-like deformations $T$ and $A$, which are phase-like in the Fourier side,
Fourier transform of considering sequence does not necessarily converge almost everywhere.
Thus, we may not expect that \eqref{eq:decouple} holds for $\hat{L}^\alpha$-norm\footnote{
Actually, when $\alpha'=4$, $f_n=f+T(n)g$ with $f,g \in \hat{L}^{4/3}$ 
is a counter example to the above decoupling.}.
This respect is rather a serious problem for linear profile decomposition,
because a decoupling like \eqref{eq:decouple} is a key for obtaining smallness of 
remainder term as the number of detected profiles increases, as mentioned above.

To overcome this difficulty, we shall show a decoupling \emph{inequality} with respect to
a weaker norm, a \emph{generalized Morrey} norm, defined as follows:
\begin{definition}
For $4/3 < \alpha < 2 $ and for $\sigma\in (\alpha',\frac{6\alpha}{3\alpha-2})$, we introduce a 
\emph{generalized Morrey norm} $\norm{\cdot}_{\hat{M}^{\alpha}_{2,\sigma}}$ by
\[
	\norm{f}_{\hat{M}^\alpha_{2,\sigma}} =
	\norm{2^{j(\frac1\alpha-\frac12)} \|\hat{f}\|_{L^{2}(\tau_k^j)} }_{\ell^{\sigma}_{j,k}}
	=\norm{|\tau_{k}^{j}|^{\frac12-\frac1\alpha}
	\|\hat{f}\|_{L^{2}(\tau_k^j)} }_{\ell^{\sigma}_{j,k}},
\]
where $\tau_k^j = [k2^{-j}, (k+1)2^{-j})$.
Further, %for fixed $\sigma \in (\alpha',\frac{3\alpha(5\alpha-8)}{2(3\alpha-4)})$,
we introduce 
\begin{equation}\label{def:ell}
	\ell (u) = \ell_\sigma(u) := \inf_{\xi\in \R}\hMn{P(\xi)u}{\alpha}{2}{\sigma},
\end{equation}
for $u\in \hat{L}^\alpha$.
\end{definition}
Details on generalized Morrey space are summarized in Section 2.
Here, we only note that the embedding $\hat{L}^\alpha \hookrightarrow \hat{M}^{\alpha}_{2,\sigma}$ holds, that $\ell(f) \sim \norm{f}_{\hat{M}^\alpha_{2,\sigma}}$,
and so that $\ell(f)$ is a quasi-norm and makes sense for all $f \in \hat{L}^\alpha$.
It is obvious by definition that $T(a)$ and $A(t)$ are deformations on $\hat{M}^\alpha_{2,\sigma}$ for any $a,t \in \R$.
Similarly, $D(h)$ is a deformation on $\hat{M}^\alpha_{2,\sigma}$
if $h$ is a dyadic number.
We introduce $\ell(\cdot)$
because $\hat{M}^{\alpha}_{2,\sigma}$ norm is not invariant (but bounded from
above and below)
under $P(\xi)$ action.
The heart of matter is that local (in the Fourier side) $L^2$ norm decouples
even under presence of multiplier-like deformations $T$ and $A$.
Hence, summing up the local $L^2$ decoupling
with respect to intervals, we recover a decoupling \emph{inequality} for $\ell(\cdot)$.
This is one of the main ideas of this article.

Because our decoupling inequality is established only for $\ell(\cdot)$,
a natural choice of the meaning of ``minimality'' of the solution is not with respect to $\norm{\cdot}_{\hat{L}^\alpha}$ any longer but to $\ell (\cdot)$.
% and 
Thus, we consider the minimization problem for
\begin{equation}\label{eq:d_+}
	d_{+}  = d_{+}(\sigma, M) :=
	\inf \{ \ell(u_0) \ |\ u_0 \in B_M \setminus \mathcal{S}_{+} \},
\end{equation}
where $M>0$ is a parameter and
$B_M:=\{ f \in \hat{L}^{\alpha} | \norm{f}_{\hat{L}^\alpha} \le M \}$ is a ball.
We consider minimization problem in a ball in $\hat{L}^\alpha$ because well-posedness
of \eqref{gKdV} is not known in the generalized Morrey space $M^{\alpha}_{2,\sigma}$.
As a result, our threshold solution may depend on $M$. % unfortunately.

Here, it is worth mentioning that
the generalized Morrey space naturally appear in the context of refinement 
of Stirchartz's estimate.
% see \cite{KPV2} for Strichartz' estimate for Airy equation.
%This kind of \cite{B,MV,CK,BV}.
The refinement, which goes back to Bourgain \cite{B1} (see also \cite{B2,B3,MVV1,MVV2,KPV3}),
have been used for linear profile decomposition in $L^2$-framework.
See \cite{MV,CK,BV} for decomposition associated with Schr\"odinger equation and
see \cite{Shao} for that with Airy equation.
We show a similar refinement for 
a Stein-Tomas type inequality 
which is a version of Strichartz's estimate adopted to $\hat{L}^\alpha$-framework,
\[
	\norm{ |\d_x|^{\frac{1}{3\alpha}} e^{-t\d_x^3} f }_{ L^{3\alpha}_{t,x} (\R\times \R) } 
	\le C \norm{f}_{\hat{M}^{\alpha}_{2,\sigma}}.
\]
For the details on this estimate, see Theorem \ref{thm:rST}.

The second difficulty comes from a linking
between generalized KdV equation and nonlinear Schr\"odinger
equation caused by the presence of $P$-deformation.
More precisely, if an initial data is of the form $u_0(x) = \Re [P(\xi) \phi(x)]$ 
then a corresponding solution to \eqref{gKdV} can be approximated in terms of a solution to nonlinear Schr\"odinger equation
\begin{equation}
\left\{
\begin{aligned}
& i\pt_tv-\pt_x^2v=-\mu |v|^{2\alpha} v,
&& t,x\in\rre,\\
&v(0,x)=v_{0}(x),
&& x\in\rre.
\end{aligned}
\right.
\label{NLS}\tag{NLS}
\end{equation}
 in the limit $|\xi|\to\I$.
This interesting phenomena is known in \cite{KKSV,T2}
(see also \cite{BoyChe, CCT, Sch}).

As for linear Airly equation, the linking with
linear Schr\"odinger equation can be explained
by an elemental identity 
\begin{equation}\label{eq:GtA}
	 A(t) P(\xi) = e^{-it\xi^3} P(\xi) T(-3\xi^2 t) e^{3i\xi t\d_x^2} A(t).
\end{equation}
The identity infers that the presence of $P$ on the initial data produces
 Schr\"odinger group $e^{3i\xi t\d_x^2}$.
% A remarkable point that it involves the Schr\"odinger group $e^{3i\xi t\d_x^2}$.
Furthermore, in fact, the Schr\"odinger evolution takes a main part in the limit $|\xi|\to\I$
because the speed of Schr\"odinger evolution becomes much faster than that of Airy evolution.
The above identity is a kind of Galilean transform, and can be compared with the one 
for Schr\"odinger equations;
\begin{equation}\label{eq:Gt}
	e^{it\d_x^2} P(\xi) = e^{- it|\xi|^2} P(\xi) T(-2t\xi) e^{it\d_x^2}.
\end{equation}
Roughly speaking, as 
a nonlinear evolution generated by a class of nonlinear Schr\"odinger equation, such as \eqref{NLS}, inherits
the Galilean transform \eqref{eq:Gt},
the effect on the nonlinear problem \eqref{gKdV} which is caused by 
the presence of $P$ in initial data is similar to 
that on the Airy equation described as in \eqref{eq:GtA}.

% That is, the solution of \eqref{gKdV}
% with data $u_0(x) = P(\xi) \phi(x)$ is approximated, in the limit $|\xi|\to\I$, by a solution to

Because of the above linking, existence of a threshold solution is shown under the 
assumption
\begin{equation}\label{asmp:gKdV_NLS}
	d_{+} < 2^{1-\frac1\sigma}\(\frac{3 \sqrt{\pi} \Gamma(\alpha+2)}{2 \Gamma(\alpha+\frac{3}2)}  \)^{\frac1{2\alpha}} d_{\mathrm{NLS}},
\end{equation}
where $d_+$ is the number given in \eqref{eq:d_+}, $\sigma$ is a parameter
chosen to define $\ell(\cdot)$, $\Gamma(x)$ is the Gamma function,
and 
\begin{equation}\label{eq:d_NLS}
	d_{\mathrm{NLS}} = d_{\mathrm{NLS}}(\sigma, M) :=
	\inf \{ \ell(u_0)\ |\  u_0 \in B_M \setminus \mathcal{S}_{\mathrm{NLS}}\} 
\end{equation}
with
\[
	\mathcal{S}_{\mathrm{NLS}} :=
 	\left\{ v_0 \in \hat{L}^{\alpha} \left|
 	\begin{aligned}
	&\text{a solution }v(t) \text{ to }\eqref{NLS}\text{ with }u_{|t=0} = u_0 \\
	&\text{scatters forward and backward in time}
	\end{aligned}
	\right.\right\}.
\]
Here, the notion of scattering of
$\hat{L}^{\alpha}$-solution $v(t)$ to \eqref{NLS} forward in time (resp. backward in time)
is defined as validity of the following two;
(i) maximal existence interval of $v(t)$ is not bounded from above (resp. from below);
(ii) $e^{it\d_x^2} v(t)$ converges in $\hat{L}^{\alpha}$ as $t\to\I$ (resp. $t\to-\I$).
It was pointed out in \cite{KKSV, T2} that, in the mass-critical case $\alpha =2$,
the problem of a threshold solution for 
 \eqref{gKdV} relates to the same problem for \eqref{NLS}.
Although we are working in the mass-subcritical case,
the same linking appears because it is due to the presence of the $P$-deformation.
When $\alpha = 2$, the assumption \eqref{asmp:gKdV_NLS} essentially coincides with 
those in \cite{KKSV, T2}.

The justification of the Schr\"odinger approximation is done essentially in
the same way as in \cite{KKSV}.
A key idea for dealing with nonlinearities of fractional order is to use
a Fourier series expansion
\[
	|\cos \theta|^{2\alpha} \sin \theta = \sum_{k=1}^\I C_k \sin (k\theta).
\]
The constant in assumption \eqref{asmp:gKdV_NLS} given in terms of the first coefficient $C_1$ of the expansion.
For this approximation, we also establish local well-posedness of
\eqref{NLS} in a scale critical $\hat{L}^\alpha$ space, which seems already a new result.

\subsection{Main Results}
% Our main result is existence of a \emph{minimal non-scattering solution} for \eqref{gKdV},
% a minimizer of $d_{+,\mathrm{gKdV}}$, under the assumption that
% \begin{equation}\label{asmp:gKdV_NLS}
% 	d_{+,\mathrm{gKdV}} < 2^{1-\frac1\sigma}\(\frac{3 \sqrt{\pi} \Gamma(\alpha+2)}{2 \Gamma(\alpha+\frac{3}2)}  \)^{\frac1{2\alpha}} d_{\mathrm{NLS}},
% \end{equation}
% where $\Gamma(x)$ is the gamma function. 
% The main purpose of this article is to show the following theorem.
In what follows, we consider the focusing case $\mu=-1$ only.
However, the focusing assumption is used only for $d_+(M)<\I$.
Our analysis work also in the defocusing case $\mu=+1$ if we assume $d_+(M)<\I$.

\vskip2mm

\begin{theorem}\label{thm:minimal}
Let $3/2+\sqrt{7/60}<\alpha <2$ and $\sigma \in (\alpha',\frac{3\alpha(5\alpha-8)}{2(3\alpha-4)})$.
Let $M > 0$ so that $B_M \cap S_+^c \neq \emptyset$. 
If the assumption \eqref{asmp:gKdV_NLS} is true
then
there exists a special solution $u_c(t)$ to \eqref{gKdV} with
maximal interval $I_{\mathrm{max}}(u_c) \ni 0$ such that 

\vskip1mm
\noindent
(i) $u_c(0)\not\in \mathcal{S}_{+}$;

\vskip1mm
\noindent
(ii) $u_c$ attains $d_{+}$ in such a sense that
one of the following two properties holds;
\begin{enumerate}
\item[(a)] $u_c(0)\in B_M$ and
$\ell(u_c(0)) = d_{+}$;
\item[(b)] $u_c(0)\in \mathcal{S}_-$ and scatters backward in time to $\displaystyle u_{c,-}$
satisfying $u_{c,-} \in B_M$ and
$\ell(u_{c,-}) = d_{+} $.
% , where $\displaystyle u_{c,-} = \lim_{t\to-\I} e^{t\d_x^3}u_c(t)$.
\end{enumerate}
\end{theorem}

In this article we call $u_{c}$ constructed in Theorem \ref{thm:minimal} by {\it minimal non-scattering solution}.  

\begin{remark}
As mentioned above, $d_+(M)$ gives a scattering criterion;
if $u_0 \in \hat{L}^\alpha$ satisfies $\norm{u_0}_{\hat{L}^\alpha} \le M$ and 
$\ell(u_0)<d_+$ then $u_0 \in \mathcal{S}_+$.
By definition of $d_+$, this is sharp in such a sense that $d_+$ cannot be replaced by 
a larger number.
It is not clear whether we can replace $\ell(u_0)<d_+$ by $\ell(u_0)\le d_+$.
\end{remark}

The assumption 
$B_M \cap S_+^c \neq \emptyset$ is fulfilled for $M \ge c_\alpha \norm{Q}_{\hat{L}^\alpha}$
because $c_\alpha Q \not\in S_+$ by means of \cite[Theorem 1.10]{MS}.
By the same reason, we have the following:
\begin{theorem}\label{thm:notsolitons}
Let $3/2+\sqrt{7/60}<\alpha <2$ and $\sigma \in (\alpha',\frac{3\alpha(5\alpha-8)}{2(3\alpha-4)})$.
Let $M > 0$ so that $B_M \cap S_+^c \neq \emptyset$. 
Then, $d_+ \le c_\alpha \ell(Q)$, where $c_\alpha$ is the constant given in \eqref{eq:c_a}.
\end{theorem}

The second result is existence of minimal non-scattering solution without the assumption \eqref{asmp:gKdV_NLS}.
For fixed $8/5 < \widetilde{\alpha}<\alpha$ and $0<\widetilde{s}<2\alpha+1$,
define $\widetilde{B}_M = \{ f \in \hat{L}^{\alpha}\ |\ \hLebn{f}{\widetilde{\alpha}} + \dSobn{f}{\widetilde{s}} \le M \}$.
It turns out that,
as for a minimizing problem for 
\[
	{d}_{+}^{\prime} = {d}_{+}^{\prime} (\sigma, M) :=
	\inf \{ \ell(u_0) \ |\ u_0 \in \widetilde{B}_M \cap \mathcal{S}_{+}^c \},
\]
a minimizer exists \emph{without} the assumption \eqref{asmp:gKdV_NLS}.
\begin{theorem}\label{thm:minimal2}
Let $3/2+\sqrt{7/60}<\alpha<2$ and $\sigma \in (\alpha',\frac{3\alpha(5\alpha-8)}{2(3\alpha-4)})$.
Let $M > 0$ so that $\widetilde{B}_M \cap S_+^c \neq \emptyset$. 
Then, there exists a special solution $\widetilde{u}_c(t)$ to \eqref{gKdV} 
which attains $\widetilde{d}_{+,\mathrm{gKdV}}$ in a similar way 
to Theorem \ref{thm:minimal}.
\end{theorem}

% The following notation will be used throughout this 
% article: $|D_x|^s=(-\pt_x^2)^{s/2}$ and 
% $\langle D_x\rangle^s=(I-\pt_x^2)^{s/2}$ denote 
% the Riesz and Bessel potentials of order $-s$, 
% respectively. For $1\le p,q\le\infty$ and $I\subset\rre$, 
% let us define the space-time norms 
% \begin{eqnarray*}
% \|f\|_{L_x^pL_t^q(I)}=
% \|\|f(\cdot,x)\|_{L_t^q(I)}\|_{L_x^p(\rre)}.
% \end{eqnarray*}

Now let us introduce several consequential results which follow
from the arguments which we establish to prove our main results.
We begin with two scattering results.
The first one is as follows;
\begin{theorem}\label{thm:sds2}
Let $5/3 \le \alpha< 20/9$. For any $M>0$ there exists $\delta=\delta(M)>0$ such that
if $u_0 \in \hat{L}^\alpha$ satisfies $\norm{u_0}_{\hat{L}^\alpha} \le M$ and
\[
	\norm{|\d_x|^{\frac{1}{3\alpha}} e^{-t\d_x^3} u_0 }_{L^{3\alpha}_{t,x}(\R\times \R)} \le \delta
\]
then a corresponding solution $u(t)$ to \eqref{gKdV} exists globally and
scatters for both time direction. 
\end{theorem}
The above theorem is a variant of small data scattering, and a consequence of 
a stability type estimate which is so-called long time stability.
Notice that it contains the case that the data is not small in the $\hat{L}^\alpha$ topology.
\begin{remark}
The proof of  \cite[Theorem 1.7]{MS} shows that there exists a constant $\delta'$ 
independent of $\norm{u_0}_{\hat{L}^\alpha}$ such that if 
$$\norm{|\d_x|^{\frac{1}{3\alpha}}e^{-t\d_x^3} u_0 }_{L^{3\alpha}_{t,x}(\R\times \R)}
 + \norm{e^{-t\d_x^3} u_0 }_{L^{\frac{5\alpha}2}_{x}(\R, L^{5\alpha}_t(\R))} \le \delta'$$
then the solution scatters for both time directions.
In Theorem \ref{thm:sds2}, smallness assumption on 
the second term of the left hand side is removed, however the constant $\delta$ may
depends on $\norm{u_0}_{\hat{L}^\alpha}$.
% $\norm{e^{-t\d_x^3} u_0 }_{L^{\frac{5\alpha}2}_{x}(\R, L^{5\alpha}_t(\R))}$
\end{remark}
The second scattering result is the following.
\begin{theorem}[Scattering due to irrelevant deformations]\label{thm:BSS}
Let $5/3 \le \alpha <2$.
Let $\{u_{0,n}\}_n \subset \hat{L}^\alpha$ be a bounded sequence.
Let $u_n(t)$ be a solution to \eqref{gKdV} with $u_n(0)=u_{0,n}$.
If a set
\[
\left\{ \phi \in \hat{L}^{\alpha}\ \Biggl|\  
	\begin{aligned}
	&\phi= \lim_{k\to\I} 
	(D(h_k) A(s_k) T(y_k) P(\xi_k))^{-1} u_{0, n_k} \text{ weakly in }\hat{L}^{\alpha}, \\
	&\exists (h_k,\xi_k, s_k,y_k) \in 2^\Z \times \R\times \R \times \R, \,\exists \text{subsequence }n_k
	\end{aligned}
	\right\}
\]
is equal to $\{0\}$ then there exists $N_0$ such that
$u_n(t)$ is global and scatters for both time direction as long as $n\ge N_0$.
\end{theorem}
This theorem is a consequence of Theorem \ref{thm:sds2}
and a concentration compactness argument.
An example of sequence $\{u_{0,n}\}_n$ that satisfies the assumption of Theorem \ref{thm:BSS}
is $u_{0,n} = e^{in \d_x^4} f$, $f \in \hat{L}^{\alpha}$.
% Here, we give an example of a sequence which satisfies the assumption of 
% Theorem \ref{thm:BSS}.
As a corollary, we also see that $\mathcal{S}_+\cap \mathcal{S}_-$
is unbounded in $\hat{L}^\alpha$ topology. 
\begin{corollary}
For any $f \in \hat{L}^\alpha$, there exists $T>0$ such that
$e^{it \d_x^4} f \in \mathcal{S}_+ \cap \mathcal{S}_-$ for $|t| \ge T$. 
In particular, $\mathcal{S}_+ \cap \mathcal{S}_-$ is an unbounded subset of $\hat{L}^\alpha$.
\end{corollary}
Unboundedness of each $\mathcal{S}_+$ and $\mathcal{S}_-$ are
seen by considering an orbit $\{ A(t) f\ |\ t\in \R \}$ of $f \in \hat{L}^\alpha$.
However, this argument does not yield that of the intersection of the both.

Finally, we state well-posedness results of nonlinear Schr\"odinger equation \eqref{NLS} in $\hat{L}^\alpha$ and $\hat{M}^\alpha_{2,\sigma}$.
Although the analysis of \eqref{NLS} is not an original purpose of the article,
this is necessary for our analysis because there is a linking
between \eqref{gKdV} and \eqref{NLS} due to the presence of $P$-deformation.

\begin{theorem}[Local well-posedness of \eqref{NLS} in $\hat{L}^\alpha$]\label{thm:LWP_N1}
The equation \eqref{NLS} is locally well-posed in
$\hat{L}^\alpha$ if $4/3<\alpha < 4$.
\end{theorem}

\begin{theorem}[Local well-posedness of \eqref{NLS} in $\hat{M}^\alpha_{2,\sigma}$]\label{thm:LWP_N2}
The equation \eqref{NLS} is locally well-posed in
$\hat{M}^\alpha_{2,\sigma}$ if $4/3<\alpha < 2$ and $\alpha'<\sigma\le
\frac{6\alpha}{3\alpha-2}$.
\end{theorem}

The rest of the article is organized as follows.
Main theorems are proven in Section \ref{sec:proof} after
preliminaries on notations and basic facts (Section \ref{sec:setup})
and stability estimate (Section \ref{sec:stability}).
For the proof, we rely on two important ingredient,
linear profile decomposition (Theorem \ref{thm:pd_r}) and
NLS approximation (Theorem \ref{thm:ENK}).
We prove Theorem \ref{thm:pd_r} in Sections \ref{sec:pd} and \ref{sec:cc}.
Finally, we turn to the proof of Theorem \ref{thm:ENK}
in Sections \ref{sec:NLS} and \ref{sec:ENK}.
On the other hand,
consequential results are shown when we are ready;
Theorem \ref{thm:sds2} is proven in Section \ref{sec:stability},
Theorem \ref{thm:BSS} is in Section \ref{sec:cc},
and Theorems \ref{thm:LWP_N1} and \ref{thm:LWP_N2} are in 
Section \ref{sec:NLS}.

%%%%%%%%%%%%%%%%%%%%%%%%%%%%%%%%%%%%%%%%%%%%%%%%%%%%
%
%     Chapter 2
%
%%%%%%%%%%%%%%%%%%%%%%%%%%%%%%%%%%%%%%%%%%%%%%%%%%%%
\section{Notations and basic facts}\label{sec:setup}

In this section, 
we introduce several notations and give lemmas 
which are needed to prove main results.

The following notation will be used throughout this 
paper: $|\pt_x|^s=(-\pt_x^2)^{s/2}$ denotes 
the Riesz potential of order $-s$. 
For $1\le p,q\le\infty$ and $I\subset\rre$, 
let us define a space-time norm
\begin{eqnarray*}
%\|f\|_{L_t^qL_x^p(I)}&=&
%\|\|f(t,\cdot)\|_{L_x^p(\rre)}\|_{L_t^q(I)},\\
\|f\|_{L_x^pL_t^q(I)}=
\|\|f(\cdot,x)\|_{L_t^q(I)}\|_{L_x^p(\rre)}.
\end{eqnarray*}

\subsection{Deformations}
Let us first collect elementary facts on the deformations which is used thorough out the article.
As in the introduction, we set
\begin{itemize}
\item $(T(y) f)(x) = f(x-y)$, $\quad y\in \R$,
\item $(P(\xi) f)(x) = e^{-ix\xi} f(x)$, $\quad\xi\in\R$,
\item $(D_p(h)f)(x)=h^{1/p} f(h x)$, $\quad h\in 2^\Z$, 
\item $(A(t) f)(x) = e^{-t\d_x^3}f(x)$, $\quad t\in \R$.
\end{itemize}
They are deformations on $\hat{L}^p$ for any $1 \le p \le \I$.
Denote $D(h) = D_\alpha (h)$, where $\alpha$ is the number in \eqref{gKdV}.
Let $S(t)=e^{it\d_x^2}$ be a Schr\"odinger group.
Notice that $S(t)$ is also a deformation on $\hat{L}^p$, $1 \le p \le \I$.
The inverses of $A(t)$, $S(t)$, $T(y)$, and $P(\xi)$ are
$A(-t)$, $S(-t)$, $T(-y)$, and $P(-\xi)$, respectively.
Further, $D_p(h)^{-1} = D_p(h^{-1})$.

We use a notation $\hat{X}:=\F X \F^{-1}$, or equivalently, $\F X= \hat{X} \F$,
for $X=A,S,T,P,D$. More specifically,
$\hat{A}(t)=e^{it\xi^3}$, 
$\hat{S}(t) = e^{-it\xi^2}$,
$\hat{T}(y)=P(y)$,
$\hat{P}(\xi) = T(-\xi)$, and
$\hat{D_\alpha}(h)  = D_{\alpha'}(h^{-1})$.
With this notation, the identity \eqref{eq:GtA} is easily obtained as follows.
\[
	\hat{P}(\xi_0)^{-1}
	\hat{A}(t) \hat{P}(\xi_0)
	= e^{it(\xi-\xi_0)^3}
	= e^{-i\xi_0^3t} \hat{T}(-3\xi_0^2 t)  \hat{S}(3\xi_0t) \hat{A}(t).
\]

Next, we collect commutations of the above deformations.
We have
\[
	[A(t),S(t)] = [A(t),T(y)] = [S(t),T(y)] = 0, \quad
	T(y)P(\xi) = e^{iy\xi}P(\xi) T(y).
\]
Commutation property for $D(h)$ is as follows:
\begin{align*}
	&A(t)D(h) =D(h)A(h^3t), &&
	S(t)D(h) =D(h)S(h^2t),\\
	&T(y)D(h)=D(h)T(hy), &&
	P(\xi) D(h) = D(h)P(h^{-1}\xi).
\end{align*}
Combining above relations, we have the following identity
\begin{multline}\label{eq:orthty_representation}
	(D(\widetilde{h})T(\widetilde{y})A(\widetilde{s})P(\widetilde{\xi}))^{-1} (D(h)T(y)A(s)P(\xi))
	\\= e^{i\gamma} D\( \frac{h}{\widetilde{h}} \) P\(\xi - \frac{\widetilde{h}}{h} \widetilde{\xi}\) 
	A\( s -\( \frac{h}{\widetilde{h}}\)^3\widetilde{s}\)\\ S\(3\( s -\( \frac{h}{\widetilde{h}}\)^3\widetilde{s}\)\xi\) T\(y- \frac{h}{\widetilde{h}} \widetilde{y}-3\( s -\( \frac{h}{\widetilde{h}}\)^3\widetilde{s}\)\xi^2 \),
\end{multline}
where $\gamma$ is a real number given by $h,y,s,\xi,\widetilde{h}, \widetilde{y}, \widetilde{s}, \widetilde{\xi}$.
This identity is useful for linear profile decomposition (see Remark \ref{rem:meaning_of_orthty}).

\subsection{Generalized Morrey space}
For $j\in \Z$, we set $\mathcal{D}_j:=\{ [k2^{-j},(k+1)2^{-j}) \ |\ k \in \Z \}$.
Let $\mathcal{D}:= \cup_{j\in\Z} \mathcal{D}_j$.
For a function $a:\mathcal{D} \to \C$, we denote
$\norm{a}_{\ell^r_{\mathcal{D}}} := (\sum_{I\in \mathcal{D}} |a(I)|^r )^{1/r} $ if $0<r<\I$ and
$\norm{a}_{\ell^\I_{\mathcal{D}}} := \sup_{I\in \mathcal{D}} |a(I)|  $.
\begin{definition}
For $1 \le q\le p \le \I$ and for $r\in [1,\I]$, we introduce a 
\emph{generalized Morrey norm} $\norm{\cdot}_{M^{p}_{q,r}}$ by
\[
	\norm{f}_{M^p_{q,r}} =
	\norm{|I|^{\frac1p-\frac1q} \norm{f}_{L^{q}(I)} }_{\ell^{r}_{\mathcal{D}}} .
\]
Here, the case $p=q$ and $r<\I$ is excluded.
For $1\le p\le q \le \I$ and for $r\in [1,\I]$,
we also introduce $\norm{f}_{\hat{M}^p_{q,r}}:= \| \hat{f}\|_{M^{p'}_{q',r}}$, 
i.e.,
\[
	\norm{f}_{\hat{M}^p_{q,r}} =
	\norm{|I|^{\frac1q-\frac1p} \norm{\hat{f}}_{L^{q'}(I)} }_{\ell^{r}_{\mathcal{D}}} .
\]
Banach spaces $M^p_{q,r}$ and $\hat{M}^p_{q,r}$ are defined as 
sets of tempered distributions of which above norms are finite, respectively.
\end{definition}

\begin{remark}\label{rem:gMorrey}

\vskip1mm
\noindent
(i) $M^{p}_{q,\I}$ is a usual Morrey space. $M^p_{p,\I}=L^p$ with equal norm.

\vskip1mm
\noindent
(ii) For any $1 \le q_1 \le q_2 \le p \le \I$ and $1\le r_1 \le r_2 \le \I$,
it holds that $M^{p}_{q_1,r_1} \hookrightarrow M^{p}_{q_2,r_2}$.

\vskip1mm
\noindent
(iii) For any $1 \le p \le q_2 \le q_1 \le \I$ and $1\le r_1 \le r_2 \le \I$,
it holds that $\hat{M}^{p}_{q_1,r_1} \hookrightarrow \hat{M}^{p}_{q_2,r_2}$

\vskip1mm
\noindent
(iv) $L^p \hookrightarrow M^{p}_{q,r}$ holds as long as 
$1\le q<p<r\le\I$.

\vskip1mm
\noindent
(v) $\hat{L}^p \hookrightarrow \hat{M}^{p}_{q,r}$ holds as long as 
$ 1\le  q'<p'<r \le\I$.
\end{remark}
For the last two assertions, see Proposition \ref{prop:gm_interpolation}.

\begin{lemma}\label{lem:scale}
Let $1\le p\le q \le \I$ and let $r\in (0,\I]$.
There exists a constant $C\ge1$ such that
\[
	C^{-1} \norm{f}_{\hat{M}^p_{q,r}}
	\le 
	\norm{D_p(h) A(s) T(y) P (\xi) f}_{\hat{M}^p_{q,r}}
	\le 
	C\norm{f}_{\hat{M}^p_{q,r}}
\]
for any $f \in \hat{M}^{p}_{q,r}$ and any $(h,\xi, s,y) \in 2^\Z \times \R\times \R\times \R$.
Further, if $\xi=0$ then the above inequality hold with $C=1$.
\end{lemma}

\vskip2mm
\noindent
\begin{proof}
%{\it Proof of Lemma \ref{lem:scale}.} 
We only consider $q>1$.
Notice that
\[
	|\F D_p(h) A(s) T(y) P (\xi) f|(x)
	= h^{-\frac1{p'}} |\F f| \( \frac{x}{h} + \xi \).
\]
Therefore, for any $\tau_k^j=[k/2^j,(k+1)/2^j) \in \mathcal{D}_j$ we have
\[
	|\tau_k^j|^{\frac1{p'}-\frac1{q'}} \norm{\F D_p(h) A(s) T(y) P (\xi) f}_{L^{q'}(\tau_k^j)}
	= |\widetilde{\tau}_k^j|^{\frac1{p'}-\frac1{q'}} \norm{\F f}_{L^{q'}(\widetilde{\tau}_k^j)},
\]
where
\[
	\widetilde{\tau}_k^j = \left[\frac{ k }{h2^j} +\xi , \frac{ k + 1}{h2^j} +\xi\right).
\]
Denote $h=2^{j_0}$.
We choose $k_0=k_0(j)$ so that $  k_0 \le 2^{j+j_0} \xi < k_0+1$.
Then,
\[
	\widetilde{\tau}_k^j = \left[\frac{ k + 2^{j+j_0} \xi }{2^{j+j_0}} , \frac{ k +2^{j+j_0} \xi+ 1}{2^{j+j_0}} \right)\subset \tau_{k+k_0}^{j+j_0} \cup \tau_{k+k_0+1}^{j+j_0} 
\]
and $|\widetilde{\tau}_k^j| = |\tau_{k+k_0}^{j+j_0}| = |\tau_{k+k_0+1}^{j+j_0}|$.
Thus,
\begin{align*}
	&|\widetilde{\tau}_k^j|^{\frac1{p'}-\frac1{q'}} \norm{\F f}_{L^{q'}(\widetilde{\tau}_k^j)}\\
	&{}\le |\widetilde{\tau}_k^j|^{\frac1{p'}-\frac1{q'}} \(\norm{\F f}_{L^{q'}({\tau}_{k+k_0}^{j+j_0})}^{q'} + \norm{\F f}_{L^{q'}({\tau}_{k+k_0+1}^{j+j_0})}^{q'} \)^{\frac1{q'}} \\
	&{}\le |\tau_{k+k_0}^{j+j_0}|^{\frac1{p'}-\frac1{q'}} \norm{\F f}_{L^{q'}({\tau}_{k+k_0}^{j+j_0})}
	+ |\tau_{k+k_0+1}^{j+j_0}|^{\frac1{p'}-\frac1{q'}} \norm{\F f}_{L^{q'}({\tau}_{k+k_0+1}^{j+j_0})}. 
\end{align*}
We take $\ell_k^r$ norm and then $\ell_j^r$ norm to obtain
 the second inequality with $C=2$.
It is obvious that if $\xi=0$ then $\widetilde{\tau}_k^j=\tau_k^{j+j_0}$ and
$|\tau_k^j|/h = |\tau_k^{j+j_0}|$ hold and so
we can take $C=1$.
The first inequality follows in the same way.
We repeat the same argument from
$|\F f| ( y ) = |\F D(h) A(s) T(y) P (\xi) f|(hy - h\xi ) $. 
\end{proof}
%$\qquad\qed$

\subsection{Generalized Strichartz's estimates}

In this subsection we give a generalized Strichartz's estimates for 
the Airy equation. To this end, we introduce several notations.

\begin{definition}
\vskip1mm
\noindent
(i) A pair $(s,r)\in \R \times [1,\I]$ 
is said to be \emph{acceptable} if
$1/r\in [0, 3/4)$ and 
\[
	s \in 
	\begin{cases}
	[-\frac{1}{2r} , \frac2r ] & 0\le \frac1r \le \frac12 ,\\
	(\frac2r - \frac54 , \frac52 - \frac3r) & \frac12 < \frac1r < \frac34.
	\end{cases}
\]
%We define a conjugate-acceptable pair. 

\vskip1mm
\noindent
(ii) A pair $(s,r) \in \R \times [1,\I]$ is said to be 
\emph{conjugate-acceptable} if
$(1-s,r')$ is acceptable, where $\frac1{r'}=1-\frac1r \in [0,1]$.
\end{definition}

For an interval $I \subset \R$ and an acceptable pair $(s,r)$, 
we define a function space $X(I;s,r)$ of space-time functions with the following norm
\[
	\norm{f}_{X(I;s,r)} = \norm{ |\d_x|^{s} f }_{L^{p(s,r)}_x(\R; L^{q(s,r)}_t(I))} ,
\]
where the exponents $p(s,r)$ and $q(s,r)$ are given by
\begin{equation}\label{def:pq}
	\frac2{p(s,r)} + \frac1{q(s,r) }=\frac1r, \quad
	-\frac1{p(s,r)} + \frac2{q(s,r) } = s,
\end{equation}
or equivalently, 
%$p(s,r)=\frac{5r}{2-sr}$ ($p(s,r)=\I$ if $s=2/r$)
%and $q(s,r)=\frac{5r}{1+2sr}$ ($q(s,r)=\I$ if $s=-1/2r$).
\[
	\begin{pmatrix}
	1/{p(s,r)}\\ 
	1/{q(s,r)}
	\end{pmatrix}
	=\begin{pmatrix}
	-1/5 & 2/5\\ 
	2/5 & 1/5
	\end{pmatrix}
	\begin{pmatrix}
	s\\ 
	1/r
	\end{pmatrix}.
\]
We refer $X(I;s,r)$ to as an \emph{$\hat{L}^r$-admissible space}.

For an interval $I \subset \R$ and a conjugate-acceptable pair 
$(s,r)$,
we define a function space $Y(I;s,r)$ by
\[
	\norm{f}_{Y(I;s,r)} = \norm{ |\d_x|^{s} f }_{L^{\tilde{p}(s,r)}_x(\R; L^{\tilde{q}(s,r)}_t(I))} ,
\]
where the exponents $\tilde{p}(s,r)$ and $\tilde{q}(s,r)$ are given by
\begin{equation}\label{def:tpq}
	\frac2{\tilde{p}(s,r)} + \frac1{\tilde{q}(s,r) }=2+\frac1r, \quad
	-\frac1{\tilde{p}(s,r)} + \frac2{\tilde{q}(s,r) } = s,
\end{equation}
or equivalently, 
% $p(s,r)=\frac{5r}{2+4r-sr}$ 
% and $q(s,r)=\frac{5r}{1+2r+2sr}$.
\[
	\begin{pmatrix}
	1/{\tilde{p}(s,r)}\\ 
	1/{\tilde{q}(s,r)}
	\end{pmatrix}
	=\begin{pmatrix}
	-1/5 & 2/5\\ 
	2/5 & 1/5
	\end{pmatrix}
	\begin{pmatrix}
	s\\ 
	2+1/r
	\end{pmatrix}
	=
	\begin{pmatrix}
	1/{p(s,r)}\\ 
	1/{q(s,r)}
	\end{pmatrix}
	+
	\begin{pmatrix}
	4/5\\ 
	2/5
	\end{pmatrix}
	.
\]

Let us define some specific $X(I;s,r)$ and $Y(I;s,r)$ type spaces
by choosing specific degrees $s=s(r)$.

\begin{definition}
Set $s(L)=s(L,\alpha):=1/(3\alpha)$, $s(Z)=s(Z,\alpha):= \frac52 - \frac3{\alpha} -\eps$,
and $s(K)=s(K,r):=\frac2{\alpha} -\frac54  + \eps$, where $\eps>0$ is a sufficiently small number.
Define $S(I) := X(I;0,\alpha)$, $L(I):= X(I;s(L),\alpha)$, $Z(I) := X(I;s(Z),\alpha)$, 
and $K(I) := X(I;s(K),\alpha)$.
Also define $N(I):= Y(I;s(L),\alpha)$.
We use the notation $(p(X),q(X)):=(p(s(X),\alpha),q(s(X),\alpha))$ for $X=S, L,K,Z$ and
$(\widetilde{p}(N),\widetilde{q}(N)) = 
(\widetilde{p}(s(L),\alpha),\widetilde{q}(s(L),\alpha))$.
\end{definition}
From the definition, we have $(p(S),q(S))=(\frac52\alpha,5\alpha)$ 
and $(p(L),q(L))=(3\alpha,3\alpha)$.
For details of choice of $s(Z)$ and $s(K)$, see Remark \ref{rem:restriction} below.

\begin{remark}\label{rem:alpha_for_stability}
The $S(I)$ norm is so-called \emph{scattering norm}.
This norm plays an important role on well-posedness theory.
For example, criterion for blowup and scattering are given in terms of the scattering norm 
(See \cite[Theorems 1.8 and 1.9]{MS}).
Notice that the pair $(0,\alpha)$ is admissible only if $\alpha>8/5$.
The $L(I)$ norm is a non-mixed space.
This norm appears in refinement of 
Stein-Tomas type inequality, see Theorem \ref{thm:rST}, below.
A pair $(s_L(\alpha),\alpha)$ is acceptable and conjugate-acceptable if $5/3 \le \alpha < 20/9$.
Remark that there exists an acceptable and conjugate-acceptable pair
under a weaker assumption $10/7<\alpha<10/3$ (see \cite[Remark 4.1]{MS}).
\end{remark}

We have the following generalized version of Strichartz's estimate.

\vskip2mm

\begin{proposition}[Generalized Strichartz's estimates]\label{prop:ho_inho}

${}$\\
(i) (homogeneous estimate) It holds for any acceptable pair $(s,r)$ and interval $I$ that
\begin{equation}\label{ho}
	\norm{e^{-t\d_x^3} f}_{X(I;s,r)} \le C \norm{f}_{\hat{L}^r},
\end{equation}
where the constant $C$ depends only on $s$ and $r$. 

\vskip1mm
\noindent
(ii) (inhomogeneous estimate) Let $4/3<r<4$. Let $(s_1,r)$ be an acceptable pair and $(s_2,r)$
be a conjugate-acceptable pair. Then, it holds for any $t_0 \in I \subset \R$ that
\begin{equation}\label{inho}
	\norm{ \int_{t_0}^t e^{-(t-t')\d_x^3 }\d_x F(t') dt'}_{L^\I_t(I;\hat{L}^r_x) \cap X(I,s_1,r)}
	\le C\norm{F}_{Y(I,s_2,r)},
\end{equation}
where the constant $C$ depends on $s_1,s_2$ and $r$.
\end{proposition}

\begin{proof}
%\noindent
%{\it Proof of Proposition \ref{prop:ho_inho}.} 
The inequality (\ref{ho}) is obtained by interpolating 
the notable Kato's smoothing effect, the Kenig-Ruiz estimate and 
the Stein-Tomas inequality. See \cite[Proposition 2.1]{MS} for the detail.
Moreover, the inhomogeneous estimate 
(\ref{inho}) follows from the combination of 
the homogeneous inequality (\ref{ho}) and the Christ-Kiselev 
lemma. See \cite[Proposition 2.5]{MS} for the detail. 
\end{proof}%$\qed$

\vskip2mm

To handle $X(I;s,r)$ and $Y(I;s,r)$ spaces,
the following lemma is useful.

\begin{lemma}\label{lem:gHolder}
Let $1< p_i,q_i <\I$ and $s_i \in \R$ for $i=1,2$.
Let $p,q,s$ be 
\[
	\frac1p = \frac{\theta}{p_1}+\frac{1-\theta}{p_2}, \quad
	\frac1q = \frac{\theta}{q_1}+\frac{1-\theta}{q_2}, \quad
	s = \theta s_1+(1-\theta)s_2
\] 
for some $\theta \in (0,1)$.
Then, there exists a positive constant $C$ such that 
the inequality
\[
	\norm{|\d_x|^s f}_{L^p_x L^q_t}
	\le C \norm{|\d_x|^{s_1} f}_{L^{p_1}_x L^{q_1}_t}^\theta \norm{|\d_x|^{s_2} f}_{L^{p_2}_x L^{q_2}_t}^{1-\theta}
\]
holds for any $f$ such that $|\d_x|^{s_1} f \in L^{p_1}_x L^{q_1}_t$ and $|\d_x|^{s_2} f \in L^{p_2}_x L^{q_2}_t$.
\end{lemma}

%\vskip2mm
%\noindent
%{\it Proof of Lemma \ref{lem:gHolder}.} 
\begin{proof}
See \cite[Lemma 3.3]{MS}. 
\end{proof}%$\qed$

\vskip2mm

To evaluate the nonlinear term, we need the following lemma.

\vskip2mm
\begin{lemma}\label{lem:nlest}
Suppose that $8/5<\alpha < 10/3$.
Let $(s,r)$ be a pair which is acceptable and conjugate-acceptable.
Then, the following two assertions hold:

\vskip1mm
\noindent
(i) If $u \in S(I) \cap X(I;s,r)$ then $|u|^{2\alpha}u\in Y(I;s,r)$.
Moreover, there exists  a positive constant $C$ such that the inequality
\begin{align*}
\||u|^{2\alpha}u\|_{Y(I;s, r)}  \le
C\norm{u}_{S(I)}^{2\alpha} \|u\|_{X(I;s,r)}
\end{align*}
holds for any $u \in S(I)  \cap X(I;s,r)$. 

\vskip1mm
\noindent
(ii) There exists  a positive constant $C$ such that the inequality 
\begin{align*}
 & \||u|^{2\alpha}u-|v|^{2\alpha}v\|_{Y(I;s, r)} \\
 &{} \le
 C(\|u\|_{X(I;s,r)} + \|v\|_{X(I;s,r)})
 (\norm{u}_{S(I)} + \norm{v}_{S(I)})^{2\alpha-1} \norm{u-v}_{S(I)} \\
&\quad +
 C(\norm{u}_{S(I)} + \norm{v}_{S(I)})^{2\alpha}
\|u-v\|_{X(I; s ,r)}
\end{align*}
holds for any $u,v \in S(I) \cap X(I;s,r)$. 
\end{lemma}

%\vskip2mm
%\noindent
%{\it Proof of Lemma \ref{lem:nlest}.} 
\begin{proof} 
See \cite[Proposition 3.4]{MS}. 
\end{proof}%$\qed$

%%%%%%%%%%%%%%%%%%%%%%%%%%%%%%%%%%%%%%%%%%%%%%%%%%%%
%
%     Chapter 3
%
%%%%%%%%%%%%%%%%%%%%%%%%%%%%%%%%%%%%%%%%%%%%%%%%%%%%

\section{Stability estimates}\label{sec:stability}

\subsection{Stability for gKdV}
We consider the 
generalized KdV equation 
with the perturbation:
\begin{eqnarray}
\left\{
\begin{array}{l}
\displaystyle{
\pt_t\tilde{u}+\pt_x^3\tilde{u}=
\mu\pt_x(|\tilde{u}|^{2\alpha}\tilde{u})+\d_x e,
\qquad t,x\in\rre,}\\
\displaystyle{\tilde{u}(\hat{t},x)=\tilde{u}_{0}(x),
\qquad\qquad\qquad\qquad \ \  x\in\rre,}
\end{array}
\right.
\label{gK}
\end{eqnarray}
where the perturbation $e$ is small in a suitable sense
and the initial data $\tilde{u}_{0}$ is close to $u_{0}$. 

The estimates in this section are restricted to $5/3 \le \alpha <20/9$
but one can easily extend the results for $8/5< \alpha < 10/3$
by modifying the definitions of $L(I)$ and $N(I)$ spaces.
See Remark \ref{rem:alpha_for_stability} for the meaning of the above restriction on $\alpha$.

\vskip2mm

\begin{lemma}[Short time stability for gKdV]\label{ssy}
Assume $5/3\le\alpha<20/9$ and $\hat{t}\in\rre$. 
Let $I$ be a time interval containing $\hat{t}$ 
and let $\tilde{u}$ be a solution to (\ref{gK}) 
on $I\times\rre$ for some function $e$. 
Then, there exists $\varepsilon_{0}>0$ such that if 
$\tilde{u}$ and $e$ satisfy 
\[
	\|\tilde{u}\|_{S(I)}+\|\tilde{u}\|_{L(I)}
	\le\varepsilon_{0},
\]
and
\[
\|e^{-(t-\hat{t})\pt_{x}^{3}}
(u(\hat{t})-\tilde{u}(\hat{t}))
\|_{S(I)}
+
\|e^{-(t-\hat{t})
\pt_{x}^{3}}(u(\hat{t})-\tilde{u}(\hat{t}))
\|_{L(I)}
+
\|e\|_{N(I)}
\le\varepsilon,
\]
and if $0<\varepsilon<\varepsilon_{0}$ hold, 
then there exists a unique solution 
$u \in S(I) \cap L(I)$ to (\ref{gKdV}) satisfying 
\begin{eqnarray}
\|u-\tilde{u}\|_{S(I)}+
\|u-\tilde{u}\|_{L(I)}
&\le&C\varepsilon,\label{t10}\\
\|
|u|^{2\alpha}u-|\tilde{u}|^{2\alpha}\tilde{u}
\|_{N(I)}&\le&C\varepsilon,\label{t11}
\end{eqnarray}
If further $u(\hat{t})-\tilde{u}(\hat{t}) \in \hat{L}^{\alpha}$ holds then
\begin{equation}\label{t12}
	\|u-\tilde{u}\|_{L_{t}^{\infty}(I;\hat{L}_{x}^{\alpha})}
	\le \|u(\hat{t})-\tilde{u}(\hat{t})\|_{\hat{L}_{x}^{\alpha}}
	+ C\eps.
\end{equation}
\end{lemma}

%\vskip2mm
%\noindent
%{\it Proof of Lemma \ref{ssy}.} 
\begin{proof}
By the local well-posedness theory, 
it suffices to show (\ref{t10}), (\ref{t11}), 
and (\ref{t12}) %and (\ref{t13}) 
as a priori estimates. 
Let $w:=u-\tilde{u}$. Then $w$ satisfies 
\begin{eqnarray*}
w(t)
&=&e^{-(t-\hat{t})\pt_{x}^{3}}
w(\hat{t})
+\mu\int_{\hat{t}}^{t}
e^{-(t-t')\pt_{x}^{3}}
\pt_x\{|\tilde{u}+w|^{2\alpha}(\tilde{u}+w)
-|\tilde{u}|^{2\alpha}\tilde{u}\}dt'\\
& &-\int_{\hat{t}}^{t}
e^{-(t-t')\pt_{x}^{3}}\d_x e(t')dt'.
\end{eqnarray*}
For $t\in I$, set
\begin{eqnarray*}
F(t)&:=&\|w\|_{S(0,t)}+
\|w\|_{L(0,t)},\\
G(t)&:=&
\|
|\tilde{u}+w|^{2\alpha}(\tilde{u}+w)
-|\tilde{u}|^{2\alpha}\tilde{u}
\|_{N(0,t)},
\end{eqnarray*}
where we use abbreviation such as $S(0,t)=S([0,t))$ to simplify notation. 
Then the assumptions on $u(\hat{t})$, 
$\tilde{u}(\hat{t})$ and $e$, and Proposition \ref{prop:ho_inho} \eqref{inho} 
lead us to
\begin{eqnarray*}
F(t)
&\le&\|e^{-(t-\hat{t})
\pt_{x}^{3}}(u(\hat{t})-\tilde{u}(\hat{t}))
\|_{S(0,t)}
+
\|e^{-(t-\hat{t})
\pt_{x}^{3}}(u(\hat{t})-\tilde{u}(\hat{t}))
\|_{L(0,t)}\\
& &+CG(t)+C\|e\|_{N(0,t)}\\
&\le&C\varepsilon+CG(t).
\end{eqnarray*}
Lemma \ref{lem:nlest} (ii) yields 
\begin{eqnarray}
G(t)
&\le&
C(\|\tilde{u}+w\|_{L(0,t)}
+\|\tilde{u}\|_{L(0,t)})\label{t15}\\
& &\qquad\qquad\times(\|\tilde{u}+w\|_{S(0,t)}
+\|\tilde{u}\|_{S(0,t)})^{2\alpha-1}\|w\|_{S(0,t)}
\nonumber\\
& &+C
(\|\tilde{u}+w\|_{S(0,t)}
+\|\tilde{u}\|_{S(0,t)})^{2\alpha}
\|w\|_{L(0,t)}
\nonumber\\
&\le&C(\varepsilon_{0}
+F(t))^{2\alpha}F(t).
\nonumber
\end{eqnarray}
Hence 
\[
F(t)
\le C\varepsilon+C\varepsilon^{2\alpha}_{0}
F(t)+CF(t)^{2\alpha+1}.
\]
Since $F(0)=0$, by the continuity argument, 
we have that if $C\varepsilon^{2\alpha}_{0}<1$, 
then 
$F(t)\le C\varepsilon$ for any 
$t\in I$.
Hence we have (\ref{t10}). Combining 
(\ref{t10}) and (\ref{t15}), we have (\ref{t11}). 

Now we suppose that
$w(\hat{t})=u(\hat{t})-\tilde{u}(\hat{t}) \in \hat{L}^{\alpha}$. 
Then, Proposition \ref{prop:ho_inho} \eqref{ho} and \eqref{inho} yield
\begin{eqnarray*}
\lefteqn{\|w\|_{L_{t}^{\infty}(I)
\hat{L}_{x}^{\alpha}}}\\
&\le&
\|w(\hat{t})\|_{\hat{L}_{x}^{\alpha}}
+C\|
|\tilde{u}+w|^{2\alpha}(\tilde{u}+w)
-|\tilde{u}|^{2\alpha}\tilde{u}
\|_{N(I)}+C\|e\|_{N(I)} \\
&\le& \|w(\hat{t})\|_{\hat{L}_{x}^{\alpha}}
+C\varepsilon,
\end{eqnarray*}
which is (\ref{t12}). 
This competes the proof of Lemma \ref{ssy}.
\end{proof}%$\qed$

%%%%%%%%%%%%%%%%%%%%%%%%%%%%%%%%%%%%%%%%%%%%%%
\vskip2mm

\begin{proposition}[Long time stability for gKdV]\label{lsy}
Assume $5/3\le\alpha<20/9$ and $\hat{t}\in\rre$. 
Let $I\subset \R$ be an interval containing $\hat{t}$. 
Let $\tilde{u}$ be a solution to (\ref{gK}) 
on $I\times\rre$ for some function $e$. 
Assume that $\tilde{u}$ satisfies
\begin{eqnarray*}
\|\tilde{u}\|_{S(I)}
+\|\tilde{u}\|_{L(I)}
&\le&M,
\end{eqnarray*}
for some $M>0$. 
Then there exists 
$\varepsilon_{1}=\varepsilon_{1}(M)>0$ 
such that if
\begin{eqnarray*}
\|e^{-t\d_x^3}(u(\hat{t})-\tilde{u}(\hat{t}))\|_{S(I)}+
\|e^{-t\d_x^3}(u(\hat{t})-\tilde{u}(\hat{t}))\|_{L(I)}+
\|e\|_{N(I)}
\le\varepsilon
\end{eqnarray*}
and $0<\varepsilon<\varepsilon_{1}$, 
then there exists a solution $u$ to (\ref{gKdV}) on 
$I\times\rre$ satisfies 
\begin{align}
\|u-\tilde{u}\|_{S(I)}+
\|u-\tilde{u}\|_{L(I)}
&{}\le C\varepsilon, \label{tt10}\\
\||u|^{2\alpha}u
-|\tilde{u}|^{2\alpha}\tilde{u}
\|_{N(I)}
&{}\le C\varepsilon,\label{tt11}
\end{align}
where the constant $C$ depends only on $M$. 
Further, if $u(\hat{t})-\widetilde{u}(\hat{t}) \in \hat{L}^{\alpha}$
for some $\hat{t}\in I$
then, it also holds that
\begin{equation}\label{tt12}
	\|u-\tilde{u}\|_{L_{t}^{\infty}(I;\hat{L}_{x}^{\alpha})}
	\le 
	\| u(\hat{t}) - \tilde{u}(\hat{t}) \|_{\hat{L}^{\alpha}}+
	C\varepsilon.
\end{equation}
\end{proposition}

%\vskip2mm
%\noindent
%{\it Proof of Proposition \ref{lsy}.} 
\begin{proof}
% Throughout the Proof of Proposition \ref{lsy}, 
% we omit subindex of $L_2$, $M_2,$, and $N_2$ for 
% simplicity.
The proof is the combination of Lemma \ref{ssy} 
and an iterative procedure. Without loss of 
generality, we may assume that $\hat{t}=0$ and 
$\inf I=0$. Now let
$\varepsilon_{0} %=\varepsilon_{0}(A,2A')
$ 
be 
the constant given in Lemma \ref{ssy}.
We first show the following 
claim: There 
exists a positive integer 
$N\le1+(2M/\varepsilon_{0})^{q(S)}$ 
such that 
\begin{eqnarray*}
I=\bigcup_{j=1}^{N}I_{j},\qquad I_{j}
=[t_{j-1},t_{j}]
\qquad\text{with}
\qquad
\|\tilde{u}\|_{S(I_{j})}+
\|\tilde{u}\|_{L(I_{j})}
\le\varepsilon_{0}
%\label{v1}
\end{eqnarray*}
for any $1\le j\le N$. 
Suppose $M>\varepsilon_{0}$, otherwise there is 
nothing to prove. Take $t_{1}\in I$ so that 
$t_{0}<t_{1}$ and $\|\tilde{u}\|_{S(I_{1})}+
\|\tilde{u}\|_{L(I_{1})}
=\varepsilon_{0}$. Similarly, as long as 
$\|\tilde{u}\|_{S((t_{j-1},\sup I))}+
\|\tilde{u}\|_{L((t_{j-1},\sup I))}
>\varepsilon_{0}$ we define $t_{j}\in I$ so that 
$t_{j-1}<t_{j}$ and $\|\tilde{u}\|_{S(I_{j})}+
\|\tilde{u}\|_{L(I_{j})}
=\varepsilon_{0}$.
Now we show that $N\le1+(2M/\varepsilon_{0})^{q(S)}$ 
by the contradiction argument. 
Suppose that $1+(2M/\varepsilon_{0})^{q(S)}<N\le\infty$. 
Let $N'$ be an integer defined by $N'=N$ if $N$ is finite 
and $N'$ any integer satisfying $1+(2M/\varepsilon_{0})^{q(S)}<N'$ if $N$ is infinite. 

For $1\le j\le N'$, set
\[
f_{j}(x):=
\|\tilde{u}(\cdot,x)\|_{L_{t}^{q(S)}(I_{j})},
\qquad
g_{j}(x):=
\||\pt_{x}|^{s(L)}\tilde{u}(\cdot,x)
\|_{L_{t}^{q(L)}(I_{j})}.
\]
Then 
\begin{eqnarray}
M&\ge&\|\tilde{u}\|_{S((0,t_{N'}))}=\biggl\|
\biggl(\|\tilde{u}(\cdot,x)
\|_{L_{t}^{q(S)}((0,t_{N'}))}^{q(S)}
\biggl)^{\frac{1}{q(S)}}\biggl\|_{L_{x}^{p(S)}}
\label{v3}\\
&=&\biggl\|
\biggl(\sum_{j=1}^{N'}|f_{j}(x)|^{q(S)}
\biggl)^{\frac{1}{q(S)}}\biggl\|_{L_{x}^{p(S)}}.
\nonumber
\end{eqnarray}
In a similar way, we have
\begin{eqnarray}
M\ge\biggl\|
\biggl(\sum_{j=1}^{N'}|g_{j}(x)|^{q(L)}
\biggl)^{\frac{1}{q(L)}}
\biggl\|_{L_{x}^{p(L)}}.
\label{v4}
\end{eqnarray}
Noting $p(S)<q(S)$ and $p(L)=q(L)$, by 
the H\"{o}lder inequality, (\ref{v3}) and (\ref{v4}), 
we obtain  
\begin{eqnarray*}
{\varepsilon_{0}N'}
&=&
\sum_{j=1}^{N'}
(\|\tilde{u}\|_{S(I_{j})}+\|\tilde{u}\|_{L(I_{j})})
\\
&\le&
(N')^{1-\frac{1}{p(S)}}
\biggl(\sum_{j=1}^{N'}\|\tilde{u}\|_{S(I_{j})}^{p(S)}
\biggl)^{\frac{1}{p(S)}}+
(N')^{1-\frac{1}{p(L)}}
\biggl(\sum_{j=1}^{N'}
\|\tilde{u}\|_{L(I_{j})}^{p(L)}
\biggl)^{\frac{1}{p(L)}}
\\
&=&
(N')^{1-\frac{1}{p(S)}}
\biggl\|
\biggl(\sum_{j=1}^{N'}|f_{j}(x)|^{p(S)}
\biggl)^{\frac{1}{p(S)}}
\biggl\|_{L_{x}^{p(S)}}\\
& &+
(N')^{1-\frac{1}{p(L)}}
\biggl\|
\biggl(\sum_{j=1}^{N'}|g_{j}(x)|^{p(L)}
\biggl)^{\frac{1}{p(L)}}
\biggl\|_{L_{x}^{p(L)}}
\\
&\le&
(N')^{1-\frac{1}{q(S)}}
\biggl\|
\biggl(\sum_{j=1}^{N'}|f_{j}(x)|^{q(S)}
\biggl)^{\frac{1}{q(S)}}
\biggl\|_{L_{x}^{p(S)}}\\
& &+
(N')^{1-\frac{1}{q(L)}}
\biggl\|
\biggl(\sum_{j=1}^{N'}|g_{j}(x)|^{q(L)}
\biggl)^{\frac{1}{q(L)}}
\biggl\|_{L_{x}^{p(L)}}
\\
&\le&((N')^{1-\frac{1}{q(S)}}
+(N')^{1-\frac{1}{q(L)}})M.
\end{eqnarray*}
Since $q(L)>q(S)$, we obtain
$N'\le1+(2M/\varepsilon_{0})^{q(L)}$. 
This contradicts the definition of $N'$, 
which proves the claim. 

From Lemma \ref{ssy}, we have that there 
exists a positive constant $C_{0}$ such that 
if a positive constant $\eta_{j}$ satisfies 
\begin{equation}\label{g1}
\|e^{-(t-t_{j-1})\pt_{x}^{3}}
w(t_{j-1})\|_{S(I_{j})}
+
\|e^{-(t-t_{j-1})\pt_{x}^{3}}
w(t_{j-1})\|_{L(I_{j})}
\le\eta_{j},
\end{equation}
and
\begin{eqnarray}
\eta_{j}\le\varepsilon_{0},\label{g11}
\end{eqnarray}
then we have
\begin{eqnarray}
\|w\|_{S(I_{j})}+\|w\|_{L(I_{j})}
&\le& C_{0}\eta_{j},\label{g15}\\
\||u|^{2\alpha}u
-|\tilde{u}|^{2\alpha}\tilde{u}
\|_{N(I_{j})}&\le&C_{0}\eta_{j}.
\label{g2}
\end{eqnarray}

On the other hand, since $w(t_{j-1})$ 
satisfies the integral equation 
\begin{eqnarray*}
\lefteqn{e^{-(t-t_{j-1})\pt_{x}^{3}}
w(t_{j-1})}\\
&=&e^{-t\pt_{x}^{3}}w(0)
+\mu\int_{0}^{t_{j-1}}
e^{-(t-t')\pt_{x}^{3}}
\pt_{x}(|u|^{2\alpha}u
-|\tilde{u}|^{2\alpha}\tilde{u})dt'\\
& &-\int_{0}^{t_{j-1}}
e^{-(t-t')\pt_{x}^{3}} \d_x e(t')dt',
\end{eqnarray*}
Proposition \ref{inho} and the assumption yield
\begin{eqnarray}
\lefteqn{\|e^{-(t-t_{j-1})\pt_{x}^{3}}
w(t_{j-1})\|_{S(I_{j})}+
\|e^{-(t-t_{j-1})\pt_{x}^{3}}
w(t_{j-1})\|_{L(I_{j})}}
\label{g5}\\
&\le&
\|e^{-t\pt_{x}^{3}}
w(0)\|_{S(I_{j})}
+
\|e^{-t\pt_{x}^{3}}
w(0)\|_{L(I_{j})}
\nonumber\\
& &+C
\||u|^{2\alpha}u
-|\tilde{u}|^{2\alpha}\tilde{u}\|_{N(0,t_{j-1})}
+\|e\|_{N(0,t_{j-1})}
\nonumber\\
&\le&C_{1}\varepsilon+C_{1}
\sum_{k=1}^{j-1}
\|
|u|^{2\alpha}u
-|\tilde{u}|^{2\alpha}\tilde{u}\|_{N(I_{j})}.
\nonumber
\end{eqnarray}
Let $\beta\ge C_{0}C_{1}+1$ and 
let
$\eta_{j}:=\beta^{j-1}C_{1}\varepsilon$
for $1\le j\le N$.
Then we easily see that 
$\eta_{1}<\eta_{2}<\cdots<\eta_{N}=\beta^{N-1}C_{1}\varepsilon.$
Here, we take  
$$
\varepsilon_{1} %=\varepsilon_{1}(A,A')
:=C_{1}^{-1}\beta^{-N+1}\varepsilon_{0}.
$$
Then, we easily see that $\eta_{j}$ satisfies (\ref{g11}) for all $1\le j \le N$
if $\varepsilon\le\varepsilon_{1}$. 
Let us show that if $\varepsilon\le\varepsilon_{1}$, 
then %(\ref{g0}) and 
 (\ref{g1}) also holds for 
$1\le j\le N$, by an induction argument on $j$,
which yields (\ref{g15}) and (\ref{g2}) %and (\ref{g25}) 
hold for $1\le j \le N$.

For $j=1$, %(\ref{g0}) and 
 (\ref{g1}) is 
fulfilled by the assumption. 
Assume for induction that %(\ref{g0}) and
 (\ref{g1}) hold 
for $1\le j\le J$ ($1\le J \le N-1$). 
Since (\ref{g2}) holds for $1\le j\le J$, from 
(\ref{g5}), we have 
\begin{eqnarray*}
\lefteqn{\|e^{-(t-t_{J})\pt_{x}^{3}}
(u(t_{J})-\tilde{u}(t_{J}))\|_{S(I_{J+1})}}\\
& &+
\|e^{-(t-t_{J})\pt_{x}^{3}}
(u(t_{J})-\tilde{u}(t_{J}))\|_{L(I_{J+1})}
\\
&\le&C_{1}\varepsilon
+C_{0}C_{1}\sum_{k=1}^{J}\eta_{k}
\le\eta_{1}+(\beta-1)\sum_{k=1}^{J}\eta_{k}\\
&=&\beta^{-J}\eta_{J+1}
+(\beta-1)\sum_{k=1}^{J}
\beta^{k-J-1}\eta_{J+1}
=\eta_{J+1}.
\end{eqnarray*}
Therefore 
(\ref{g1}) holds for $j=J+1$, and hence for all $1\le j \le N$ by induction.

From (\ref{g15}), we have
\[
\|w\|_{S(I)} + \|w\|_{L(I)}
\le \sum_{j=1}^{N} ( \|w\|_{S(I_{j})} + \|w\|_{L(I_{j})} )
\le C_{0}\sum_{j=1}^{N}\eta_{j}
\le\beta^{N}\varepsilon
\]
if $\eps \le \eps_1$.
This proves (\ref{tt10}).
In a similar way, \eqref{g2} implies \eqref{tt11}.
Finally if $w(0) \in \hat{L}^{\alpha}$, we use \eqref{tt11} to obtain
\begin{align*}
	\norm{w}_{L^\I(I,\hat{L}^{\alpha})}
	\le{}& \norm{w(0)}_{\hat{L}^{\alpha}}
	+ C \norm{|u|^{2\alpha}u -|\tilde{u}|^{2\alpha} \tilde{u} }_{N(I)}
	+ C \norm{e}_{N(I)} \\
	\le{}& \norm{w(0)}_{\hat{L}^{\alpha}} + C\eps
\end{align*}
This completes the proof of 
Proposition \ref{lsy}. 
\end{proof}%$\qed$

\subsection{A version of small data scattering}
As a simple consequence of Proposition \ref{lsy}, 
we have the following result, which is Theorem \ref{thm:sds2}.
\begin{corollary}\label{cor:sds}
Let $5/3 \le \alpha< 20/9$. For any $M>0$ there exists $\delta=\delta(M)>0$ such that
if $u_0 \in \hat{L}^\alpha$ satisfies $\norm{u_0}_{\hat{L}^\alpha} \le M$ and
\[
	\eps := \norm{e^{-t\d_x^3} u_0 }_{L(\R)} \le \delta
\]
then a corresponding solution $u(t)$ to \eqref{gKdV} exists globally and
scatters for both time directions. Further, it holds that
\[
	\norm{u}_{S(\R)} + \norm{u}_{L(\R)} \le M + CM^{2\alpha}\eps
\]
for some constant $C$.
\end{corollary}

%\vskip2mm
%\noindent
%{\it Proof of Corollary \ref{cor:sds}.} 
\begin{proof}
We just apply Proposition \ref{lsy} with $\widetilde{u}(t,x) = e^{-t\d_x^3} u_0$, $I=\R$, and $\hat{t}=0$.
Remark that
\[
	\norm{\widetilde{u}}_{S(\R)} + \norm{\widetilde{u}}_{L(\R)} 
	\le C \norm{u_0}_{\hat{L}^\alpha} \le CM
\]
follows from \eqref{ho} and by assumption.
Further, $u(0)-\widetilde{u}(0) \equiv0$ and
\[
	\norm{e}_{N(\R)}
	= \norm{|\widetilde{u}|^{2\alpha} \widetilde{u} }_{N(\R)}
	\le C\norm{\widetilde{u}}_{S(\R)}^{2\alpha} \norm{\widetilde{u}}_{L(\R)}
	\le C M^{2\alpha} \eps \le \eps_1
\]
for sufficiently small $\delta=\delta(M)$, where $\eps_1$ is the constant given in
Proposition \ref{lsy}.
Therefore, the assumption of Proposition \ref{lsy} is satisfied.
\end{proof}%$\qed$

%%%%%%%%%%%%%%%%%%%%%%%%%%%%%%%%%%%%%%%%%%%%%%%%%%%%
%
%     Chapter 4
%
%%%%%%%%%%%%%%%%%%%%%%%%%%%%%%%%%%%%%%%%%%%%%%%%%%%%

\section{Proof of main theorems}\label{sec:proof}

\subsection{Two tools}
For the proof of Theorem \ref{thm:minimal}, we introduce the following two tools.

The first one is a linear profile decomposition for $\hat{L}^{\alpha}$-bounded sequences.
% To state it, we first introduce notation.
% We denote $A(t)=e^{-t\d_x^3}$ for $t\in \R$.
% For $y,\xi\in \R$, define translation operators $T(y)$ and $P(\xi)$ by
% \[
% 	(T(y) f)(x)=f(x-y),\quad
% 	(P(\xi)f)(x) = e^{-ix\xi}f(x).
% \]
% We also introduce a dilation operator $D(h)=D_\alpha(h)$ by
% \[
% 	(D(h)f)(x) = h^{\frac1{\alpha}} f(hx)
% \]
% for $h>0$.
% These are one-parameter groups of isometric isomorphisms on $\hat{L}^{\alpha}$.
Let us define a set of deformations as follows
\begin{equation}\label{eq:symG}
	G := \{ D(h) A(s) T(y) P(\xi) \ |\ \Gamma = (h, \xi, s, y) \in 2^\Z \times
\R\times \R \times \R \}.
\end{equation}
% It is obvious that a symmetry $\mathcal{G} \in G$ is an isometric isomorphism.
We often identify $\mathcal{G} \in G$ with
a corresponding parameter $\Gamma \in 2^\Z \times\R\times \R \times \R$
if there is no fear of confusion.
Let us now introduce a notion of orthogonality between two families of deformations.

\begin{definition}\label{def:orthty}
We say two families of deformations $\{\mathcal{G}_n\} \subset G$ and 
$\{ \widetilde{\mathcal{G}}_n \} \subset G$ are \emph{orthogonal} if
corresponding parameters $\Gamma_n ,\widetilde{\Gamma}_n 
\in 2^\Z \times \R\times \R \times \R$ satisfies 
\begin{multline}\label{eq:def_orthty}
	\lim_{n\to\I} \bigg(\abs{\log \frac{h_n}{\widetilde{h}_n}}+\abs{ \xi_n - \frac{\widetilde{h}_n}{h_n} \widetilde{\xi}_n } 
	+ \abs{ s_n -\( \frac{h_n}{\widetilde{h}_n}\)^3\widetilde{s}_n}(1+|\xi_n|) \\
	+ \abs{y_n- \frac{h_n}{\widetilde{h}_n} \widetilde{y}_n-3\( s_n -\( \frac{h_n}{\widetilde{h}_n}\)^3\widetilde{s}_n\)(\xi_n)^2 } \bigg)= +\I.
\end{multline}
\end{definition}

\begin{remark}\label{rem:meaning_of_orthty}
 It follows from  \eqref{eq:orthty_representation} that
\begin{multline*}%\label{eq:orthty_representation}
	(\widetilde{\mathcal{G}}_n)^{-1} \mathcal{G}_n
	= e^{i\gamma_n} D\( \frac{h_n}{\widetilde{h}_n} \) P\(\xi_n - \frac{\widetilde{h}_n}{h_n} \widetilde{\xi}_n\) \\
	A\( s_n -\( \frac{h_n}{\widetilde{h}_n}\)^3\widetilde{s}_n\) S\(3\( s_n -\( \frac{h_n}{\widetilde{h}_n}\)^3\widetilde{s}_n\)\xi_n\) \\ T\(y_n- \frac{h_n}{\widetilde{h}_n} \widetilde{y}_n-3\( s_n -\( \frac{h_n}{\widetilde{h}_n}\)^3\widetilde{s}_n\)\xi_n^2 \),
\end{multline*}
where $\Gamma_n$ and $\widetilde{\Gamma}_n$ are parameters associated
with $\mathcal{G}_n$ and $\widetilde{\mathcal{G}}_n$, respectively,
and $\gamma_n$ is a real constant given by $\Gamma_n$ and $\widetilde{\Gamma}_n$.
Intuitively,  the orthogonality given in Definition \ref{def:orthty} implies at least one of the deformations in the right hand side produces bad behavior.
\end{remark}

\begin{theorem}[Linear profile decomposition for ``real valued'' 
functions]\label{thm:pd_r}
Let $4/3<\alpha<2$ and $\alpha'<\sigma < \frac{6\alpha}{3\alpha-2}$.
Let $u=\{u_n\}_n$ be a sequence of real-valued functions in $B_M$.
Then, there exist $\psi^j \in B_M $, $r_n^j \in B_{(2j+1)M}$ and 
pairwise orthogonal families of deformations $\{\mathcal{G}^j_n\}_n \subset G$ 
($j=1,2,\dots$) parametrized by $\{ \Gamma_n^j = (h_n^j, \xi_n^j, s_n^j , y_n^j) \}_n$
such that, extracting a subsequence in $n$,
\begin{equation}\label{eq:pf:decomp}
	u_n = \sum_{j=1}^l \Re (\mathcal{G}^j_n \psi^j) + r_n^l
\end{equation}
for all $l\ge1$ and
\begin{equation}\label{eq:pf:smallness}
	\limsup_{n\to\I} \norm{ e^{-t\d_x^3} r_n^l }_{L(\R)} \to 0
\end{equation}
as $l\to\I$. For all $j \ge1$, 
\[
	\text{either} \quad \xi^j_n=0,\,\ \forall n\ge 0 \quad \text{or} \quad \xi_n^j \to \I \text{ as }n\to\I.
\]
Moreover, a decoupling inequality
\begin{equation}\label{eq:pf:Pythagorean}
	\limsup_{n\to\I} \ell(u_n) \ge \(\sum_{j=1}^J c_j^{1-\sigma}\ell(\psi^j)^\sigma \)^{1/\sigma}
	+ \limsup_{n\to\I} \ell( r_n^J ) 
\end{equation}
holds for all $J\ge1$, where
\[
	c_j =
	\begin{cases}
	1& \text{ if }\xi_n^j \equiv 0 ,\\
	2& \text{ if }\xi_n^j \to \I \text{ as }n\to\I.
	\end{cases}
\]
Furthermore, it holds that
\begin{equation}\label{eq:pf:bddness1}
	c_j\norm{\psi^j}_{\hat{L}^{\alpha}} \le \limsup_{n\to\I} \norm{u_n}_{\hat{L}^{\alpha}}
\end{equation}
for any $j$.
% \begin{equation}\label{eq:pf:bddness2}
% 	\limsup_{n\to\I} \norm{ r_n^j }_{\hat{M}^\alpha_{2,\sigma}}
% 	\le \limsup_{n\to\I} \norm{ u_n }_{\hat{M}^\alpha_{2,\sigma}}
% \end{equation}
% for any $j$.
\end{theorem}

The second tool to prove Theorem \ref{thm:minimal} 
is uniform boundedness of solutions with highly oscillating initial data.
The assumption \eqref{asmp:gKdV_NLS} is necessary for this boundedness.

%%%%%%%%%%%%%%%%%%%%%%%%%%%%%%%
\begin{theorem}\label{thm:ENK}
Let $12/7<\alpha<2$. 
Assume \eqref{asmp:gKdV_NLS}.
Let 
$\phi\in\hat{L}_{x}^{\alpha}(\rre)$ 
be a complex valued function such that
\[
	\ell(\phi) < 2^{1-\frac1\sigma} d_{+}.
\] 
Let $\{\xi_{n}\}_{n\ge1}\subset(0,\infty)$ 
with $\xi_{n}\to\infty$ and let 
$\{t_{n}\}_{n\ge1}\subset\rre$ be such that 
$-3t_{n}\xi_{n}$ converges to some $T_{0}
\subset[-\infty,\infty]$. Then for $n$ 
sufficiently large, a corresponding $\hat{L}^\alpha$-solution $u_{n}$ to (\ref{gKdV}) with the 
initial condition 
\begin{eqnarray}
u_{n}(t_{n},x)=A(t_{n})
Re[P(\xi_{n})\phi(x)]
\label{tttg}
\end{eqnarray}
exists globally in time.
Moreover, the solution $u_{n}$ satisfies a uniform
space-time bound
$\|u_{n}\|_{S(\rre)}+\|u_{n}\|_{L(\rre)}\le C$,
where $C$ is a positive constant depending only 
on $\phi$. 
\end{theorem}
We postpone the proof of Theorems \ref{thm:pd_r} and \ref{thm:ENK}
to Sections \ref{sec:pd} and \ref{sec:ENK}, respectively.

\subsection{Proof of Theorem \ref{thm:minimal}}

\subsection*{Step 1}
Take a minimizing sequence $\{u_n\}_n$ as follows;
\begin{eqnarray}
	u_n \in B_M \setminus \mathcal{S}_{+} , \quad
	\ell(u_n) \le d_{+} + \frac1n.\label{mins}
\end{eqnarray}
We apply the linear profile decomposition theorem (Theorem \ref{thm:pd_r}) 
to the sequence $\{u_n\}_n$.
Then, up to subsequence, we obtain a decomposition
\begin{eqnarray}
	u_n = \sum_{j=1}^l \Re (\mathcal{G}^j_{n} \psi^j) + r_n^l
	\label{pdec}
\end{eqnarray}
for $n,l\ge1$.
By extracting subsequence and changing notations if necessary, we may assume that
for each $j$ and $\{x_n^j\}_{n,j} = \{\log h_n^j\}_{n,j}, \{t_n^j\}_{n,j}$, $\{y_n^j\}_{n,j},
\{ 3 \xi_n^j t_n^j \}$,
either $x_n^j \equiv 0$, $x_n^j \to \I$ as $n\to\I$, or 
$x_n^j \to -\I$ as $n\to\I$ holds.
\subsection*{Step 2}
In this step and the next step,
we shall show that $\psi^j \equiv 0$ except for at most one $j$.

Suppose not. Then, by means of \eqref{eq:pf:Pythagorean}, we have
$c_j^{\frac1\sigma-1} \ell (\psi^j) < d_{+}$ for all $j$.
Let us define $V_n^j(t,x)$ as follows:
\begin{itemize}
\item When $\xi_n \equiv 0$, we let
$V_n^j(t) = D(h_n^j) T(y_n^j) \Psi^j ((h_n^j)^3t + t_n^{j})$, where
$\Psi^j (t)$ is a nonlinear profile associated with
$(\Re \psi^j,t_n^j)$, that is,
\begin{itemize}
\item if $t_n^j\equiv0$ then $\Psi^j (t)$ is a solution to \eqref{gKdV} with $\Psi^j(0) = \Re \psi^j$;
\item if $t_n^j \to \I$ as $n\to \I$ then $\Psi^j (t)$ is a solution to \eqref{gKdV}
that scatters forward in time to $e^{-t\d_x^3} \Re \psi^j$;
\item if $t_n^j \to -\I$ as $n\to \I$ then $\Psi^j (t)$ is a solution to \eqref{gKdV}
that scatters backward in time to $e^{-t\d_x^3} \Re \psi^j$;
\end{itemize}
\item When $\xi_n \to  \I$, we let
$V_n^j(t) = D(h_n^j) T(y_n^j) \Psi_n^j ((h_n^j)^3t + t_n^{j})$, where
$\Psi_n^j$ is a solution to \eqref{gKdV} with the initial condition
\[
	\Psi_n^j(t_n^{j}) = A(t_n^j) \Re (P(\xi_n^j)\psi^j ).
\]
\end{itemize}
Let us show the following two lemmas.
\begin{lemma}[uniform bound on the approximate solution]\label{lem:pf_bdd1}
There exists $M>0$ such that 
\[
	\norm{ V_n^j }_{K(\rre_{+})} +
	\norm{ V_n^j }_{Z(\rre_{+})}
	\le M
\]
holds for any $j,n \ge 1$. %and $n \ge N(J)$.
\end{lemma}

%\vskip2mm
%\noindent
%{\it Proof of Lemma \ref{lem:pf_bdd1}.} 
\begin{proof}
The case $\xi_n^j \to \I$ follows from Theorem \ref{thm:ENK}.
Hence, here we assume that $\xi_n^j \equiv 0$.
Note that $c_j=1$.
Since the deformations $D(h_n^j)$ and $T(y_n^j)$ leave the left hand side invariant,
it suffices to show that
\[
	\norm{  \Psi^j }_{K((t_n^j,\I))} +
	\norm{  \Psi^j }_{Z((t_n^j,\I))}
\]
is bounded uniformly in $n$.
Since $\ell (\psi^j) < d_{+}$ by assumption, 
$\Psi^j$ scatters forward in time.
Hence, if $t_n^j\equiv 0$ or if $t_n^j\to\I$ as $n\to\I$ then 
\[
	\norm{  \Psi^j }_{K((t_n^j,\I))} +
	\norm{ \Psi^j }_{Z((t_n^j,\I))}
	\le 
	\norm{  \Psi^j }_{K((0,\I))} +
	\norm{  \Psi^j }_{Z((0,\I))}
	<\I
\]
by scattering criterion.
If $t_n \to-\I$ as $n\to\I$ then $\Psi^j$ scatters for both time directions and so
\[
	\norm{ \Psi^j }_{K((t_n^j,\I))} +
	\norm{ \Psi^j }_{Z((t_n^j,\I))}
	\le 
	\norm{ \Psi^j }_{K(\R)} +
	\norm{ \Psi^j }_{Z(\R)}
	<\I.
\]
Hence, we obtain Lemma \ref{lem:pf_bdd1}. 
\end{proof}%$\qed$

\vskip2mm

Next lemma is concerned with the decoupling of 
the nonlinear profile. 

\begin{lemma}\label{lem:pf:orth}
For any $j \neq k$, we have
\[
	\lim_{n\to\I} \norm{ V_n^j V_n^k 
	}_{L^{\frac{p(S)}2}_x L^{\frac{q(S)}2}_t(\rre)}
	= 0,
\]
where $(p(S),q(S))=(\frac52\alpha,5\alpha)$.
\end{lemma}

\vskip2mm

To prove Lemma \ref{lem:pf:orth}, it suffices to show the following lemma. 

\begin{lemma}\label{331} Set 
\begin{equation*}
\tilde{V}_{n}^{j}:=\left\{
\begin{aligned}
&D(h_{n}^{j})T(y_{n}^{j})
\Psi^{j}((h_{n}^{j})^{3}t+t_{n}^{j},x)
\qquad(\text{if}\ \xi_{n}^{j}\equiv0),
\\
&D(h_{n}^{j})T(y_{n}^{j})
\Psi^{j}(-3(\xi_{n}^{j})[(h_{n}^{j})^{3}t+t_{n}^{j}],
x+3(\xi_{n}^{j})^{2}[(h_{n}^{j})^{3}t+t_{n}^{j}])\\
&\qquad\qquad\qquad\qquad\qquad\qquad\ \ 
\qquad(\text{if}\ \xi_{n}^{j}\to\infty\ \text{as}\ n\to\infty).
\end{aligned}
\right.
\end{equation*}
Then for any $j\neq k$, we have
\begin{eqnarray}
	\lim_{n\to\I} \norm{ \tilde{V}_n^j \tilde{V}_n^k }_{L^{\frac{p(S)}2}_x L^{\frac{q(S)}2}_t(\rre)}
	= 0.
	\label{st}
\end{eqnarray}
\end{lemma}

%\vskip2mm
%\noindent
%{\it Proof of Lemma \ref{331}.} 
\begin{proof} 
The proof is now standard. 
%We employ an argument by \cite[Lemma 2.6]{KKSV}. 
Let us first prove (\ref{st}) when $\xi_{n}^{j}\to\infty$ and $\xi_{n}^{k}\to\infty$ as $n\to\I$.
By density, it suffices to handle the case $\Psi^j(t,x) \equiv  \Psi^k(t,x) \equiv {\bf 1}_{\{|t|\le C ,\, |x|\le C\}}(t,x)$, where $C>0$.
Note that  
\begin{eqnarray*}
\lefteqn{D(h_{n}^{j})T(y_{n}^{j})
\Psi^{j}(-3(\xi_{n}^{j})[(h_{n}^{j})^{3}t+t_{n}^{j}],x+3(\xi_{n}^{j})^{2}[(h_{n}^{j})^{3}t+t_{n}^{j}])}\\
& &
=(h_{n}^{j})^{\frac{1}{\alpha}}
\Psi^{j}(-3(\xi_{n}^{j})[(h_{n}^{j})^{3}t+t_{n}^{j}],(h_{n}^{j})x-y_{n}^{j}
+3(\xi_{n}^{j})^{2}[(h_{n}^{j})^{3}t+t_{n}^{j}]).
\end{eqnarray*}

\vskip1mm
\noindent
{\bf Case 1: $\lim_{n\to\I}|\log\frac{h_{n}^{j}}{h_{n}^{k}}|=\infty$.} 
We may assume $\lim_{n\to\I}\frac{h_{n}^{k}}{h_{n}^{j}}=\infty$.
The H\"{o}lder inequality yields  
\begin{eqnarray*}
\norm{ \tilde{V}_n^j \tilde{V}_n^k }_{L^{\frac{p(S)}2}_x L^{\frac{q(S)}2}_t
(\rre)}^{\frac{5\alpha}{4}}
&\le&(h_{n}^{j})^{\frac54}(h_{n}^{k})^{\frac54}
\|{\bf 1}_{R_{n}^{j}}\|_{L_{x}^{\infty}L_{t}^{\infty}
}^{\frac{5\alpha}{4}}
\|{\bf 1}_{R_{n}^{k}}\|_{L^{\frac{5\alpha}4}_x L^{\frac{5\alpha}2}_t}^{\frac{5\alpha}{4}}\\
&\le&C(h_{n}^{j})^{\frac54}(h_{n}^{k})^{\frac54}
|I_{n}^{k}|^{\frac12}|\text{Area}R_{n}^{k}|^{\frac12},
\end{eqnarray*}
where 
\begin{eqnarray*}
I_{n}^{k}&=&\{x\in\rre|\ 
\|{\bf 1}_{R_{n}^{k}}(\cdot,x)\|_{L_{t}^{1}}\neq0\},\\
R_{n}^{k}&=&\{(t,x)\in\rre^{2}|\ 3(\xi_{n}^{k})|(h_{n}^{k})^{3}t+t_{n}^{k}|\le C, \\
& &\qquad\qquad\quad
|(h_{n}^{k})x-y_{n}^{k}+3(\xi_{n}^{k})^{2}[(h_{n}^{k})^{3}t+t_{n}^{k}]|\le C\}.
\end{eqnarray*}
We easily see that
\begin{eqnarray*}
|I_{n}^{k}|\le C\frac{\Jbr{\xi_{n}^{k}}}{h_{n}^{k}},\qquad
|\text{Area} R_{n}^{k}|\le\frac{C}{(h_{n}^{k})^{4}\xi_{n}^{k}}.
\end{eqnarray*}
Since $\xi_n^k\to\I$ as $n\to\I$, we have 
\begin{eqnarray*}
\norm{ \tilde{V}_n^j \tilde{V}_n^k 
}_{L^{\frac{p(S)}2}_x L^{\frac{q(S)}2}_t(\rre)}^{\frac{5\alpha}{4}}\le
 C\left(\frac{h_{n}^{j}}{h_{n}^{k}}\right)^{\frac54}\to0\qquad\text{as}\ n\to\infty.
\end{eqnarray*}

\vskip1mm
\noindent
{\bf Case 2: $|\log\frac{h_{n}^{j}}{h_{n}^{k}}|<\infty$ and 
$\lim_{n\to\I}|\xi_{n}^{j}-\frac{h_{n}^{k}}{h_{n}^{j}}\xi_{n}^{k}|=\infty$.} 
The H\"{o}lder inequality yields  
\begin{eqnarray}
\ \ \ \ \ \norm{ \tilde{V}_n^j \tilde{V}_n^k }_{L^{\frac{p(S)}2}_x L^{\frac{q(S)}2}_t(\rre)}^{\frac{5\alpha}{4}}
\le C(h_{n}^{j})^{\frac54}(h_{n}^{k})^{\frac54}
|I_{n}^{j,k}|^{\frac12}
|\text{Area}(R_{n}^{j}\cap R_{n}^{k})|^{\frac12},
\label{rr1}
\end{eqnarray}
where 
$I_{n}^{j,k}=\{x\in\rre|\ 
\|{\bf 1}_{R_{n}^{j}}{\bf 1}_{R_{n}^{k}}\|_{L_{t}^{1}}\neq0\}.$ 
Since $ R_{n}^{j}\subset\{(t,x)\in\rre^{2}|\ 
|(h_{n}^{j})x-y_{n}^{j}+3(\xi_{n}^{j})^{2}[(h_{n}^{j})^{3}t+t_{n}^{j}]|\le C\},$
changing variables $(t,x)\mapsto(v,w):$
\begin{eqnarray*}
v:&=&(h_{n}^{j})x-y_{n}^{j}+3(\xi_{n}^{j})^{2}[(h_{n}^{j})^{3}t+t_{n}^{j}],\\
w:&=&(h_{n}^{k})x-y_{n}^{k}+3(\xi_{n}^{k})^{2}[(h_{n}^{k})^{3}t+t_{n}^{k}],
\end{eqnarray*}
we have
\begin{eqnarray}
\text{Area}(R_{n}^{j}\cap R_{n}^{k})
&\le&\int_{|v|\le C}\int_{|w|\le C}
\left|\frac{\pt(t,x)}{\pt(v,w)}\right|dvdw
\label{rr12}\\
&\le&\frac{C}{h_{n}^{j}h_{n}^{k}|(\xi_{n}^{j})^{2}(h_{n}^{j})^{2}
-(\xi_{n}^{k})^{2}(h_{n}^{k})^{2}|
}.
\nonumber
\end{eqnarray}
Furthermore, we easily see 
\begin{eqnarray}
|I_{n}^{j,k}|\le C\min\left\{
\frac{\xi_{n}^{j}}{h_{n}^{j}},
\frac{\xi_{n}^{k}}{h_{n}^{k}}\right\}.
\label{rr13}
\end{eqnarray}
Plugging (\ref{rr12}) and (\ref{rr13}) into (\ref{rr1}), we have
\begin{eqnarray*}
\norm{ \tilde{V}_n^j \tilde{V}_n^k 
}_{L^{\frac{p(S)}2}_x L^{\frac{q(S)}2}_t(\rre)}^{5\alpha}
%\le
% C\frac{h_{n}^{j}h_{n}^{k}}{|\xi_{n}^{j}h_{n}^{j}-\xi_{n}^{k}h_{n}^{k}|^{2}}
\le \frac{C}{\left|
\xi_{n}^{j}-\frac{h_{n}^{k}}{h_{n}^{j}}\xi_{n}^{k}\right|^{2}}
\to0\qquad\text{as}\ n\to\infty.
\end{eqnarray*}
where we have used that $h_{n}^{j}\sim h_{n}^{k}$.

\vskip1mm
\noindent
{\bf Case 3: $|\log\frac{h_{n}^{j}}{h_{n}^{k}}|+|\xi_{n}^{j}-\frac{h_{n}^{k}}{h_{n}^{j}}\xi_{n}^{k}|<\infty$ and 
$\lim_{n\to\I}|t_{n}^{j}-(\frac{h_{n}^{j}}{h_{n}^{k}})^{3}t_{n}^{k}|(1+|\xi_{n}^{j}|)=\infty$.} 
Let $T_{n}^{j}=\{t\in\rre|\exists x\in\rre\ \text{such\ that}\ (t,x)\in R_{n}^{j}\}$. 
We see that for $n$ sufficiently large, $T_{n}^{j}\cap T_{n}^{k}=\phi$. Indeed, we have  
\begin{eqnarray*}
T_{n}^{j}\cap T_{n}^{k}=\phi
&\Leftrightarrow&
%\left|\frac{t_{n}^{j}}{(h_{n}^{j})^{3}}
%-\frac{t_{n}^{k}}{(h_{n}^{k})^{3}}\right|
%>C
%\left(\frac{1}{3\xi_{n}^{j}(h_{n}^{j})^{3}}
%+\frac{1}{3\xi_{n}^{k}(h_{n}^{k})^{3}}\right)\\
%&\Leftrightarrow&
\left|t_{n}^{j}-\left(\frac{h_{n}^{j}}{h_{n}^{k}}\right)^{3}t_{n}^{k}\right||\xi_{n}^{j}|>C 
\end{eqnarray*}
for some large positive constant. 
Since the right hand side of the above inequality goes to infinity as $n\to\infty$, we have 
that for $n$ sufficiently large, $T_{n}^{j}\cap T_{n}^{k}=\phi$. 
Therefore the H\"{o}lder inequality yields (\ref{st}). 

\vskip1mm
\noindent
{\bf Case 4: $|\log\frac{h_{n}^{j}}{h_{n}^{k}}|+|\xi_{n}^{j}-\frac{h_{n}^{k}}{h_{n}^{j}}\xi_{n}^{k}|
+|t_{n}^{j}-(\frac{h_{n}^{j}}{h_{n}^{k}})^{3}t_{n}^{k}|(1+|\xi_{n}^{j}|)<\infty$ and 
$\lim_{n\to\I}|y_{n}^{j}-\frac{h_{n}^{j}}{h_{n}^{k}}y_{n}^{k}
-3[t_{n}^{j}-(\frac{h_{n}^{j}}{h_{n}^{k}})^{3}t_{n}^{k}](\xi_{n}^{j})^{2}|=\infty$.} 
Let $X_{n}^{j}(t)=\{x\in\rre|(t,x)\in R_{n}^{j}\}$. 
We see that if $t\in T_{n}^{j}\cap T_{n}^{k}$ for $n$ sufficiently large, then 
$X_{n}^{j}(t)\cap X_{n}^{k}(t)=\phi$. Indeed, we have 
\begin{eqnarray*}
X_{n}^{j}(t)\cap X_{n}^{k}(t)=\phi 
&\Leftrightarrow&
\left|y_{n}^{j}-\frac{h_{n}^{j}}{h_{n}^{k}}y_{n}^{k}
-3%\frac{1}{h_{n}^{j}}
\left\{t_{n}^{j}
-\frac{(h_{n}^{j})^{3}}{(h_{n}^{k})^{3}}t_{n}^{k}
\right\}(\xi_{n}^{j})^{2}     
\right|>C
\end{eqnarray*}
for some large positive constant. 
Since the right hand side of the above inequality goes to infinity as $n\to\infty$, we have 
that for $n$ sufficiently large, $X_{n}^{j}(t)\cap X_{n}^{k}(t)=\phi$. 
Therefore the H\"{o}lder inequality yields (\ref{st}). 

By arguments similar to those in cases $1,3$, and $4$, we obtain (\ref{st})
when $\xi_{n}^{j}\equiv0$ and $\xi_{n}^{k}\to\infty$ as $n\to\I$ . 
Similarly, arguing as in cases $1$ and $2$, we obtain (\ref{st}) 
when $\xi_{n}^{j}\equiv\xi_{n}^{k}\equiv0$ as $n\to\I$. 
This completes the proof of Lemma \ref{331}. 
\end{proof}%$\qed$

\vskip2mm

\begin{lemma}\label{lem:pf:bdd}
Let $F(z) = |z|^{2\alpha} z$.
For any $\mathcal{J} \subset \Z_+$,
\[
	\norm{ F\( \sum_{j \in \mathcal{J}} V_n^j  \) -\sum_{j\in \mathcal{J}} F(V_n^j)  }_{L^{\frac{p(S)}{2\alpha+1}}_x L^{\frac{q(S)}{2\alpha+1}}_t (\rre_{+})}
	=o(1)
\]
as $n\to\I$. Similarly,
\[
	\norm{F\( \sum_{j \in \mathcal{J}} V_n^j  \) -\sum_{j\in \mathcal{J}} F(V_n^j)}_{N (\rre_{+})}
	=o(1)
\]
as $n\to\I$.
\end{lemma}

\begin{proof}
%\vskip2mm
%\noindent
%{\it Proof of Lemma \ref{lem:pf:bdd}.} 
The former estimate is a consequence of Lemma \ref{lem:pf:orth}.
Indeed, we see that
\begin{multline*}
	\abs{F\( \sum_{j \in \mathcal{J}} V_n^j  \) -\sum_{j\in \mathcal{J}} F(V_n^j) }
	\le C \abs{\sum_{j_1,j_2 \in \mathcal{J}, j_1 \neq j_2} V_n^{j_1} V_n^{j_2}} \\
	\abs{\sum_{j_3,j_4 \in \mathcal{J}, j_3 \neq j_4} V_n^{j_3} V_n^{j_4}}
	\abs{\sum_{j_5,j_6 \in \mathcal{J}, j_5 \neq j_6} V_n^{j_5} V_n^{j_6}}^{\frac{2\alpha-3}2}.
\end{multline*}
Therefore, the H\"older inequality and Lemma \ref{lem:pf:orth} give us the desired estimate.
Take a conjugate-acceptable pair $(s(N'),\alpha)$ so that $0 < s(N') - s(N) \ll 1$,
and set $N'(I):=Y(I;s(N'),\alpha)$.
By means of the interpolation estimate (Lemma \ref{lem:gHolder}), the latter estimate follows if we 
show that
\[
	\norm{\(F\( \sum_{j \in \mathcal{J}} V_n^j  \) -\sum_{j\in \mathcal{J}} F(V_n^j)  \)}_{N'(\rre_{+})}
\]
is bounded uniformly in $n$.
When $s(N')$ is chosen sufficiently close to $s(N)$, we have
\[
	\norm{F(u)}_{\rre_{+})} \le C \norm{u}_{L(\rre_{+}) \cap S(\rre_{+})}^{2\alpha+1}
\]
just as in the proof of Lemma \ref{lem:nlest}.
Therefore,
\begin{align*}
	&\norm{ \(F\( \sum_{j \in \mathcal{J}} V_n^j  \) -\sum_{j\in \mathcal{J}} F(V_n^j)  \)}_{N' (\rre_{+})} \\
	&{} \le \norm{ F\( \sum_{j \in \mathcal{J}} V_n^j  \) }_{N'(\rre_{+})}
	+ \sum_{j\in \mathcal{J}} \norm{ F(V_n^j) }_{N' (\rre_{+})} \\
	&{} \le C \norm{\sum_{j \in \mathcal{J}} V_n^j}_{L(\rre_{+}) \cap S(\rre_{+})}^{2\alpha+1}
	+ C \sum_{j \in \mathcal{J}}  \norm{V_n^j}_{L(\rre_{+}) \cap S(\rre_{+})}^{2\alpha+1} \\
	&{} \le C \sum_{j \in \mathcal{J}}  \norm{V_n^j}_{L (\rre_{+})\cap S(\rre_{+})}^{2\alpha+1}.
\end{align*}
The right hand side is bounded uniformly in $n$, thanks to Lemma \ref{lem:pf:bdd},
which completes the proof. 
\end{proof}%$\qed$

\subsection*{Step 3}
Here, we define an approximate solution
\begin{eqnarray}
	\tilde{u}_n^J (t,x)= \sum_{j=1}^J V_n^j(t,x) + e^{-t\d_x^3} r_n^J,
	\label{aps}
\end{eqnarray}
where $V_{n}^{j}$ is given in Step 2. 
To apply long time stability, we now check that $\tilde{u}_n^J$ satisfies the assumption.

\begin{proposition}[Asymptotic agreement at the initial time]\label{prop:pf:lps1} 
Let $\tilde{u}_{n}^{J}$ and $u_{n}$ be given by 
(\ref{aps}) and (\ref{mins}), respectively. 
Then it holds for any $J\ge 1$ that
\[
	\hLebn{\tilde{u}_n^J(0) - u_n}{\alpha} \to 0
\]
as $n\to\I$.
\end{proposition}

\begin{proof}
%\vskip2mm
%\noindent
%{\it Proof of Proposition \ref{prop:pf:lps1}.} 
This follows from
\[
	V_n^j(0) - \mathcal{G}^j_{n} \psi^j \to 0 \IN \hat{L}^{\alpha}
\]
for each $j$,
which is an immediate consequence of the way $V_n^j$ are constructed. 
\end{proof}%$\qed$

\vskip2mm

\begin{proposition}[Uniform bound on the approximate solution]\label{prop:pf:lps2}
There exists $M>0$ such that 
\[
	\norm{ \tilde{u}_n^J }_{K(\rre_{+})} +
	\norm{ \tilde{u}_n^J }_{Z(\rre_{+})}
	\le M
\]
holds for any $J \ge 1$ and $n \ge N(J)$.
\end{proposition}

\vskip2mm

Recall that each $V_n^j$ ($j\ge1$) is 
bounded in $K(0,\I) \cap  Z(0,\I) $ uniformly in $n$ 
(Lemma \ref{lem:pf_bdd1}).
Further, $e^{-t\d_x^3} r_n^J$ is also bounded uniformly in $J,n \ge 1$.
Hence, we shall show that there exists $J_0$ such that
\[
	 \norm{ \sum_{j = J_0+1}^{J_0 + k} V_n^j }_{K(\rre_{+})} +
	\norm{ \sum_{j = J_0+1}^{J_0 + k}  V_n^j }_{Z(\rre_{+})}
 \le C
\]
for any $k\ge 1$ and $n\ge N(k)$.
To this end, we need the following.

\begin{lemma}\label{lem:pf:small1}
For any $\eps>0$,
there exists $J_0=J_0(\eps)$ such that
\[
	 \norm{ \sum_{j = J_0+1}^{J_0 + k}e^{-t\d_x^3}V_n^j(0) }_{L(\rre_{+})} +
	\norm{ \sum_{j = J_0+1}^{J_0 + k} e^{-t\d_x^3} V_n^j(0) }_{S(\rre_{+})}
 \le \eps
\]
for any $k\ge 1$ and $n\ge N(k)$.
\end{lemma}

\begin{proof}
%\vskip2mm
%\noindent
%{\it Proof of Lemma \ref{lem:pf:small1}.} 
By Proposition \ref{prop:pf:lps1}, it suffices to prove the estimate for $e^{-t \d_x^3} \mathcal{G}^j_{n} \psi^j$ instead of $e^{-t\d_x^3} V_n^j(0)$.
By Theorem \ref{thm:rST} in Section 6, we see that
\begin{align*}
	\norm{ \sum_{j = J_0+1}^{J_0 + k} e^{-t\d_x^3}
	\mathcal{G}^j_{n} \psi^j}_{L(\rre_{+})}
	\le{}& C\hMn{ \sum_{j = J_0+1}^{J_0 + k} \mathcal{G}^j_{n} \psi^j }{\alpha}{2}{\sigma} \\
	\le{}& C \( \sum_{j = J_0+1}^{J_0 + k} \hMn{ \psi^j }{\alpha}{2}{\sigma}^\sigma \)^{1/\sigma} + o(1).
\end{align*}
Therefore, for any $\eps>0$, we can choose $J_0(\eps)$ so that 
\[
	\norm{ \sum_{j = J_0+1}^{J_0 + k}e^{-t\d_x^3}\mathcal{G}^j_{n}
	 \psi^j}_{L(\rre_{+})}
	\le \frac{\eps}2
\]
for any $k \ge 1$ and $n\ge n(k)$.

On the other hand,
\[
	\norm{ \sum_{j = J_0+1}^{J_0 + k} e^{-t\d_x^3}\mathcal{G}^j_{n}
	 \psi^j}_{S(\rre_{+})}^{p(S)}
	\le \sum_{j = J_0+1}^{J_0 + k} \norm{  e^{-t\d_x^3}\mathcal{G}^j_{n} \psi^j}_{S(\rre_{+})}^{p(S)}
	+o(1)
\]
as $n\to\I$. Since
\begin{align*}
	\norm{  e^{-t\d_x^3}\mathcal{G}^j_{n} \psi^j}_{S(\rre_{+})}
	&{}\le C \norm{  e^{-t\d_x^3}\mathcal{G}^j_{n} \psi^j}_{L(\rre_{+})}^{\theta}
	\norm{ e^{-t\d_x^3}\mathcal{G}^j_{n} \psi^j}_{Z(\rre_{+})}^{1-\theta}, \\
	&{}\le C \hMn{ \mathcal{G}^j_{n} \psi^j}{\alpha}{2}{\sigma}^{\theta}
	\hLebn{\psi^i}{\alpha}^{1-\theta}\\
	&{}\le C \hMn{ \mathcal{G}^j_{n} \psi^j}{\alpha}{2}{\sigma}^{\theta},
\end{align*}
where $\theta = -s(Z)/(s(L)-s(Z))$, one verifies that
\[
	\sum_{j = J_0+1}^{J_0 + k} \norm{  e^{-t\d_x^3}\mathcal{G}^j_{n} \psi^j}_{S(\rre_{+})}^{p(S)}
	\le C \sum_{j = J_0+1}^{\I} \hMn{ \mathcal{G}^j_{n} \psi^j}{\alpha}{2}{\sigma}^{\theta p(S)}.
\]
The right hand side is bounded since $ \theta p(S) > \sigma$.
Hence, we can choose $J_1(\eps)$ so that 
\[
	\sum_{j = J_1+1}^{J_1 + k} 
	\norm{  e^{-t\d_x^3}\mathcal{G}^j_{n} \psi^j}_{S(\rre_{+})}
	\le \frac{\eps}2
\]
for any $k \ge 1$ and $n\ge n(k)$. 
\end{proof}%$\qed$

\vskip2mm

\begin{remark}\label{rem:restriction}
Our assumption $\alpha > 3/2+ \sqrt{7/60}$ comes from the condition $\theta p(S) > \sigma$
in this lemma.
By letting $s(Z) \downarrow \frac2\alpha - \frac54$, we have $\theta p(S) \to \frac{3\alpha (5\alpha -8)}{2(3\alpha -4)}$.
This upper bound of $\sigma$,
which restrict us to the above range of $\alpha$ with lower bound $\sigma > \alpha '$,
is used only in this lemma, and all other arguments work with a weaker assumption $\sigma < 6\alpha /(3\alpha -2)$.
Here, we also remark on the choice of the space $Z(I)$.
For any fixed $\alpha > 3/2 + \sqrt{7/60}$, we are able to choose $\sigma$ so that
$\alpha' < \sigma < \frac{3\alpha (5\alpha -8)}{2(3\alpha -4)}$.
Then, we fix the space $Z(I)$ so that $\theta = -s(Z)/(s(L)-s(Z))$
satisfies $\theta p(S) > \sigma$.
\end{remark}

\vskip2mm
%\noindent
%{\it Proof of Proposition \ref{prop:pf:lps2}.} 
We now prove Proposition \ref{prop:pf:lps2}. 
\begin{proof}
The integral equation that $W_n^k:=\sum_{j=J_0+1}^{J_0+k} V_n^j$ satisfies is
\[
	W_n^k = e^{-t\d_x^3} W_n^k(0)+\mu\int_0^t  e^{-(t-s)\d_x^3} \d_x (|W_n^k|^{2\alpha} W_n^k +  E_n^k) ds
	,
\]
where
\[
	- E_n^k = |W_n^k|^{2\alpha} W_n^k - \sum_{j=J_0+1}^{J_0+k} |V_n^j|^{2\alpha} V_n^j .
\]
Therefore,
\begin{align*}
	\norm{W_n^k}_{L(\R_+) \cap S(\R_+)}
	\le{}& \norm{e^{-t\d_x^3} W_n^k(0)}_{L(\R_+) \cap S(\R_+)}
	+ C \norm{W_n^k}_{L(\R_+) \cap S(\R_+)}^{2\alpha+1} \\
	&{}+ C \norm{ E_n}_{N(\R_+)}.
\end{align*}
Fix $\eps>0$.
Thanks to Lemma \ref{lem:pf:small1}, one can choose $J_0$ so that
\[
	\norm{e^{-t\d_x^3} W_n^k(0)}_{L(\R_+) \cap S(\R_+)}
	\le \eps
\] 
for any $k\ge 1$ and $n\ge N(k)$.
Further, for this $J_0$, we have
\[
	\norm{ E_n}_{N(\R_+)} \le \eps
\]
for any $k \ge 1$ and $n\ge N(k)$. 
\end{proof}%$\qed$

\vskip2mm

\begin{proposition}[Approximate solution to the equation]\label{prop_123}
Let $\tilde{u}_n^J$ be defined by (\ref{aps}). Then 
$\tilde{u}_n^J$ is an approximate solution to \eqref{gKdV}
in such a sense that
\[
	\lim_{J\to\I} 
	\limsup_{n\to\I} \norm{|\d_x|^{-1}[(\d_t + \d_{xxx}) \tilde{u}_n^J -\mu \d_x (|\tilde{u}_n^J|^{2\alpha}\tilde{u}_n^J )]}_{N(\R_+)} = 0.
\]
\end{proposition}

%\vskip2mm
%\noindent
%{\it Proof of Proposition \ref{prop_123}.} 
\begin{proof} 
First note that the identity
\begin{align*}
	&(\d_t + \d_{xxx}) \tilde{u}_n^J - \mu\d_x (|\tilde{u}_n^J|^{2\alpha}\tilde{u}_n^J )\\
	&{}= \mu\sum_{j=1}^J \d_x (|V_n^j|^{2\alpha} V_n^j) - \mu\d_x (|\tilde{u}_n^J|^{2\alpha}\tilde{u}_n^J ) \\
	&{}= \mu\d_x \left\{ \sum_{j=1}^J (|V_n^j|^{2\alpha} V_n^j) -  \(\abs{\sum_{j=1}^J V_n^j}^{2\alpha} \sum_{j=1}^J V_n^j \) \right\} \\
	&{}\quad + \mu\d_x \left\{  \(\abs{\sum_{j=1}^J V_n^j}^{2\alpha} \sum_{j=1}^J V_n^j \) 
	-  (|\tilde{u}_n^J|^{2\alpha}\tilde{u}_n^J) \right\}\\
	&{}=: \d_x I_1+ \d_x I_2.
\end{align*}
Lemma \ref{lem:pf:bdd} implies $\lim_{n\to\I} \norm{I_1}_{N(\R_+)} = 0 $.
Therefore, we only have to handle $I_2$.
From Lemma \ref{lem:nlest} (ii) 
and Proposition \ref{prop:pf:lps2}, we have
\begin{eqnarray*}
	\lefteqn{\norm{I_2}_{N(\R_+)}}\\
	&\le&C(\|u_{n}\|_{L (\R_+) \cap S(\R_+)}
	+\|e^{-t\pt_{x}^{3}} r_n^J\|_{L (\R_+) \cap S(\R_+)})^{2\alpha}
	\norm{e^{-t\pt_{x}^{3}} r_n^J }_{L (\R_+) \cap S(\R_+)}\\
	&\le&C\norm{e^{-t\pt_{x}^{3}} r_n^J }_{L (\R_+) \cap S(\R_+)}
\end{eqnarray*}
for any $n \ge N(J)$. By \eqref{eq:pf:smallness},
\[
	\lim_{J\to\I} \limsup_{n\to\I} \norm{e^{-t\pt_{x}^{3}}r_n^J }_{L } =0
\]
and
\begin{align*}
	\limsup_{n\to\I} \norm{e^{-t\pt_{x}^{3}}r_n^J }_{S(\rre_{+})} 
	\le{}& C \limsup_{n\to\I} \norm{e^{-t\pt_{x}^{3}}r_n^J 
	}_{L(\rre_{+})}^{\theta} 
	\norm{e^{-t\pt_{x}^{3}}r_n^j}_{Z(\rre_{+})}^{1-\theta} \\
	\le{}& C M^{1-\theta}  \limsup_{n\to\I} \norm{e^{-t\pt_{x}^{3}}r_n^J 
	}_{L(\rre_{+})}^{\theta} \to 0
\end{align*}
as $J\to\I$. This yields $\lim_{n\to\I} \norm{I_2}_{N(\R_+)} = 0 $. 
Hence we have the desired estimate. 
\end{proof}%$\qed$

\vskip2mm

Now, we apply long time stability to see that $\norm{u_n}_{S(\R_+)} < \I$
for sufficiently large $n$.
This implies that $u_n \in \mathcal{S}_+$, which contradicts with the definition of 
$\{u_n\}_n$.

\subsection*{Step 4}
We now see that there exists $j_0$ such that
$c_j^{\frac1\sigma -1} \ell(\psi^{j_0}) = d_{+,\mathrm{gKdV}}$.
Then, one sees from the definition of $\{u_n\}_n$ and \eqref{eq:pf:Pythagorean} 
that $\psi^{j}\equiv 0$ for $j\neq j_0$.
For simplicity, we drop index $j_0$ and write
\[
	u_n = \mathcal{G}_{n} \psi + r_n,\qquad
	\tilde{u}_n (t) = V_n(t) + e^{-t\d_x^3} r_n
\]
in what follows.
Further, we have $\lim_{n\to\I}\hMn{r_n}{\alpha}{2}{\sigma}=0$ 
and so
\[
	\lim_{n\to\I} \norm{e^{-t\d_x^3} r_n }_{K \cap Z} = 0.
\]
When $|\xi_n| \to \I$, as in the previous step,
we see from assumption \eqref{asmp:gKdV_NLS}
and Theorem \ref{thm:ENK} that
$u_n \in S_{+}$ for large $n$, a contradiction.
Hence, $\xi_n\equiv 0$. Recall that
\[
	V_n =
	 D(h_n) T(y_n) \Psi ((h_n)^3t + t_n)
\]
where
$\Psi (t)$ is a nonlinear profile associated with $(\psi,t_n)$.
Let us now show that $u_{c}:=\Psi$ is the solution which has the desired property.
We have $\Psi (t_n) \not\in \mathcal{S}_{+}$, otherwise
$u_n \in \mathcal{S}_+$ for large $n$ by long time stability.

The case $t_n \to \I$ ($n\to\I$) is excluded since this implies $\Psi (t_n)\in \mathcal{S}_{+}$.
If $t_n \equiv 0$ then $\Psi(0) = \psi$ and so
 $\ell(\Psi(0))  = d_{+} $.
Finally, if $t_n \to -\I$ as $n\to\I$ then $\lim_{t\to-\I} e^{t\d_x^3} \Psi(t) = \psi$ and
 putting $u_{c,-}:=\lim_{t\to-\I} e^{t\d_x^3}\Psi(t)$, we have 
$\ell(u_{c,-}) = d_{+} $. This completes the proof of Theorem \ref{thm:minimal}. 
$\qed$

\subsection{Proof of Theorem \ref{thm:minimal2}}
The proof is essentially the same as for Theorem \ref{thm:minimal}.
We first take a minimizing sequence associated with $\tilde{d}_{+}$.
Then, we apply Theorem \ref{thm:pd_r}.
The difference is that uniform boundedness in $\hat{L}^{\tilde{\alpha}} \cap H^s$
gives us $h_n^j \equiv 1$ and $\xi_n^j \equiv 0$ (See Proposition \ref{prop:pd:addtionalbound}).
Thus, the assumption \eqref{asmp:gKdV_NLS} is not necessary  any longer
since it is necessary just to exclude the case $\xi_n^j \to \I$ via Theorem \ref{thm:ENK}.
% Hence, we do not use Theorem \ref{thm:ENK},
% and so the assumption \eqref{asmp:gKdV_NLS} is not necessary  any longer.
Recall that $\tilde{B}_M \subset B_{M}$.
Hence, the rest of the proof is the same. 
This completes the proof of Theorem \ref{thm:minimal2}. 

%%%%%%%%%%%%%%%%%%%%%%%%%%%%%%%%%%%%%%%%%%%%%%%%%%%%
%
%     Chapter 5
%
%%%%%%%%%%%%%%%%%%%%%%%%%%%%%%%%%%%%%%%%%%%%%%%%%%%%

\section{Linear profile decomposition}\label{sec:pd}

In this section and the next section, we prove the linear profile decomposition 
(Theorem \ref{thm:pd_r}).
To clarify the proof, we first prove a decomposition of 
sequence of \emph{complex-valued}  functions (Theorem \ref{thm:pd}, below).
The desired decomposition for real-valued functions then follows as a corollary.

As in \cite{CK}, the proof splits into two parts.
The first part, treated in this section as Theorem \ref{thm:pd1}, is the procedure
of finding profiles 
and obtaining pairwise orthogonality between profiles
and remainder term.
By employing a successive notion of smallness of remainder term, used in \cite{BV,CK,Ker}, 
this part can be shown in an abstract way.
Remark that, in fact, even a set of deformations need not to be specified for this decomposition procedure.
For given boundedness and a set of deformations, a corresponding notion of orthogonality
is selected and a corresponding decomposition is obtained.
The set of the deformations 
\begin{equation*}%\label{eq:symG}
	G := \{ D(h) A(s) T(y) P(\xi) \ |\ \Gamma = (h, \xi, s, y) \in 2^\Z \times
\R\times \R \times \R \}
\end{equation*}
used in Theorem \ref{thm:pd_r} % and Theorem \ref{thm:pd}
 comes not from the decomposition procedure
but from a concentration compactness, which is 
the second part of the proof and to be proven in the next section.

\begin{remark}
The above $G$ plays a role of \emph{a group of dislocations} in the sense of \cite{ST}.
Remark that, however, $G$ is not a group.
\end{remark}

As explained in the introduction, 
a decoupling equality \eqref{eq:decouple} fails by the presence of 
multiplier-like deformations $A$ and $T$,
and so the main point of our decomposition is to establish a decoupling
\emph{inequality} with respect to $\ell(\cdot)$ (Lemma \ref{lem:decouple}).

Let us now state the main result of this section.
For a bounded sequence $P=\{P_n\}_n \subset \hat{L}^{\alpha}$,
we introduce
a \emph{set of weak limits modulo deformations}
\[
	\mathcal{V}(P)
	:=\left\{ \phi \in \hat{L}^{\alpha}\ \Biggl|\  
	\begin{aligned}
	&\phi= \lim_{k\to\I} \mathcal{G}_{n_k}^{-1} P_{n_k} \text{ weakly in }\hat{L}^{\alpha}, \\
	&\exists \mathcal{G}_n \in G, \,\exists \text{subsequence }n_k
	\end{aligned}
	\right\}.
\]
and define
\[
	\eta (P):= \sup_{\phi \in \mathcal{V}(P)}  \ell(\phi).
\]
By definition, $\eta(P)=0$ implies that we may not find any weak limit from 
a sequence $\{P_n\}_n$ even modulo the orbit by deformations $G$.
Conversely, if $\eta(P)>0$ we can find a non-zero weak limit modulo $G$.
The main result of this section is decomposition with a smallness
of remainder with respect to $\eta$.
\begin{theorem}\label{thm:pd1}
Let $4/3<\alpha<2$ and $\alpha'<\sigma \le \frac{6\alpha}{3\alpha-2}$.
Let $u=\{u_n\}_n$ be a bounded sequence of $\C$-valued functions in $\hat{L}^{\alpha}$.
Then, there exist $\psi^j \in \mathcal{V}(u) $ and 
pairwise orthogonal families $\{\mathcal{G}_n^j\}_{n} \subset
G$ ($j=1,2,\dots$) such that
\begin{equation}\label{eq:pd:key_decomp}
	u_n = \sum_{j=1}^l \mathcal{G}^j_n \psi^j + r_n^l
\end{equation}
for all $l\ge1$ with
\begin{equation}\label{eq:pd:key_smallness}
	\eta (r^l) \to 0
\end{equation}
as $l\to\I$. Further, a decoupling inequality
\begin{equation}\label{eq:pd:key_Pythagorean}
	\limsup_{n\to\I} \ell(u_n)^\sigma \ge \sum_{j=1}^J \ell(\psi^j)^\sigma	
	+\limsup_{n\to\I} \ell(r_n^J)^\sigma
\end{equation}
holds for all $J\ge1$.
Further, it holds that
\begin{equation}\label{eq:pd:key_bddness1}
	\norm{\psi^j}_{\hat{L}^{\alpha}} \le \limsup_{n\to\I} \norm{u_n}_{\hat{L}^{\alpha}}
\end{equation}
and
% \begin{equation}\label{eq:pd:key_bddness2}
% 	\limsup_{n\to\I} \norm{ r_n^j }_{\hat{M}^\alpha_{2,\sigma}}
% 	\le \limsup_{n\to\I} \norm{ u_n }_{\hat{M}^\alpha_{2,\sigma}}
% \end{equation}
\begin{equation}\label{eq:pd:key_bddness2}
	\limsup_{n\to\I} \norm{ r_n^j }_{\hat{L}^\alpha}
	\le (j+1)\limsup_{n\to\I} \norm{ u_n }_{\hat{L}^\alpha}
\end{equation}for any $j$.
\end{theorem}

\begin{remark} 
(i) The important thing in decoupling inequality \eqref{eq:pd:key_Pythagorean}
is that the coefficients of the right hand is equal to one.
This is why we work not with $\norm{\cdot}_{\hat{M}^\alpha_{2,\sigma}}$ but with $\ell(\cdot)$.

\vskip1mm
\noindent
(ii) Unlikely to Theorem \ref{thm:pd_r},
the parameters $\xi_n^j$ can be negative in this theorem.
As a result, the notion of orthogonality is weaker.
See Remark \ref{rem:realorthty} for details.
\end{remark}

\subsection{A characterization of orthogonality}
% \subsection{On orthogonality of families of parameters}
To begin with, we give a characterization of orthogonality of two families of 
deformations given in Definition \ref{def:orthty}.
The orthogonality is realized 
as a condition which gives us the following two properties.
\vskip2mm

\begin{lemma}[Characterization of orthogonality]\label{lem:orthty}
Let $\mathcal{G}_n, \widetilde{\mathcal{G}}_n \in G$ be two families of deformations.
%  and let
% $\Gamma_n ,\widetilde{\Gamma}_n 
% \in 2^\Z \times \R\times \R \times \R$ be corresponding sequence of paramteres.
The following three statements are equivalent.

\vskip1mm
\noindent
(i) $\mathcal{G}_n$ and $\widetilde{\mathcal{G}}_n$ are orthogonal.

\vskip1mm
\noindent
(ii) It holds that
\[
	(\widetilde{\mathcal{G}}_n)^{-1} 
	\mathcal{G}_{n} \psi \rightharpoonup 0 \wIN
	\hat{L}^\alpha
\]
as $n\to\I$ for any $\psi \in \hat{L}^{\alpha}$.

\vskip1mm
\noindent
(iii) 
For any subsequence of $n_k$ there exists a sequence $u_k \in \hat{L}^\alpha$ such that, up to subsequence (of $k$),
\[
	(\mathcal{G}_{n_k})^{-1} u_k \rightharpoonup \psi \neq 0,\qquad
	(\widetilde{\mathcal{G}}_{n_k})^{-1} u_k\rightharpoonup 0
\]
weakly in $\hat{L}^\alpha$ as $k\to\I$ .
\end{lemma}

% By \eqref{eq:orthty_representation},

\begin{proof}
%\vskip2mm
%\noindent
%{\it Proof of Lemma \ref{lem:orthty}.} 
``(ii)$\Rightarrow$(iii)'' is immediate by taking $u_k = (\mathcal{G}_{n_k}) \psi$ for some $\psi\neq0$.

We prove ``(i)$\Rightarrow$(ii)''.
Remark that the stated weak convergence is equivalent to
\[
	\F(\widetilde{\mathcal{G}}_n)^{-1}
	\mathcal{G}_n \psi = (\hat{\widetilde{\mathcal{G}}}_{n})^{-1}
	\hat{\mathcal{G}}_{n} \F \psi \rightharpoonup 0 \wIN
	{L}^{\alpha'}
\]
as $n\to\I$. Set 
\begin{align*}
	h'_n &{} = \frac{h_n}{\widetilde{h}_n},&
	\xi'_n &{} = \xi_n - \frac{\widetilde{h}_n}{h_n} \widetilde{\xi}_n,\\
	s'_n &{} = s_n -\( \frac{h_n}{\widetilde{h}_n}\)^3\widetilde{s}_n,&
	y'_n &{} = y_n- \frac{h_n}{\widetilde{h}_n} \widetilde{y}_n-3\( s_n -\( \frac{h_n}{\widetilde{h}_n}\)^3\widetilde{s}_n\)\xi_n^2  
\end{align*}
By density argument and \eqref{eq:orthty_representation}, it suffices to show that
\[
	\int {\bf 1}_K(\xi) [\hat{D}(h'_n) \hat{P}(\xi'_n)  \hat{A}(s'_n) \hat{T}(y'_n) \hat{S}(3\xi_n s'_n) 
	{\bf 1}_L](\xi) d\xi \to0
\]
as $n\to\I$ for any compact intervals $K,L$.
When $|\log h'_n|\to\I$ as $n\to\I$, we have
\begin{multline*}
	\abs{\int {\bf 1}_K(\xi) [\hat{D}(h'_n) \hat{P}(\xi'_n)  \hat{A}(s'_n) \hat{T}(y'_n) \hat{S}(3\xi_n s'_n) 
	{\bf 1}_L](\xi) d\xi }\\
	\le \min \( (h'_n)^{-1+\frac1r} |K|, (h'_n)^{\frac1r} |L| \) \to 0
\end{multline*}
as $n\to\I$. 
When $\limsup_{n\to\I }|\log h'_n|  <\I$ and $|\xi'_n|\to\I $ as $n\to\I$,
$K$ and support of $\hat{D}(h'_n) \hat{P}(\xi'_n) {\bf 1}_L$
are disjoint for large $n$ and so we have the desired estimate.

We hence assume that $\limsup_{n\to\I } \(|\log h'_n| + |\xi'_n| \) < \I$.
Taking subsequence, we may suppose that $h'_n\to h' \in (0,\I)$ and $\xi'_n \to \xi' \in \R$
as $n\to\I$.
Then, for any $f\in \hat{L}^{\alpha}$, $D(h'_n)P(\xi'_n)f$ converges to $D(h')P(\xi')f$
strongly in $\hat{L}^{\alpha}$.
Since $D(h')P(\xi')$ is invertible, we need to show that
${A}(s'_n) T(y'_n) {S}(3\xi_n \tau'_n) \psi \rightharpoonup 0$ weakly in $\hat{L}^\alpha$
as $n\to\I$. 
To do so, it suffices to show 
\[
	\int [\hat{A}(s'_n) \hat{T}(y'_n) \hat{S}(3\xi_n s'_n) 
	{\bf 1}_L](\xi) d\xi \to0
\]
as $n\to\I$ for any compact interval $L$.
Put $\Phi_n(\xi) := s'_n\xi^3 -3 s'_n \xi_n \xi^2 - y'_n \xi$.
Then, the right hand side is written as
\[
	\int_{L} e^{i\Phi_n(\xi)}\, d\xi.
\]
We first consider the case $|\xi_n| \to \I$. 
By orthogonality assumption, we have
$|s'_n\xi_n| + |y'_n| \to \I$ as $n\to\I$. Hence,
we see by an elementary computation
that $\inf_{L} |\Phi_n'|\to\I$ and 
$\sup_{L} |\Phi''_n|/|\Phi_n'|^2 \to 0 $ as $n\to\I$.
Notice that there exists a special case such that
$s'_n \to 0 $, $|s'_n\xi_n| + |y'_n| \to \I$,
and $(6 \eta s'_n \xi_n + y'_n) \to 0$ as $n\to\I$ for some constant $\eta \in \R$.
In such a case, we may assume that $\eta \not\in L$ by density.
Now, the desired smallness follows by integration by parts;
\begin{align*}
	\abs{\int_{L} e^{i\Phi_n(\xi)}\, d\xi}
	&{}\le \abs{\left[ \frac{e^{i\Phi_n(\xi)}}{i\Phi'_n(\xi)}\right]_{\inf L}^{\sup L} 
	+ \int_L \frac{\Phi''_n(\xi)}{i(\Phi'_n(\xi))^2}e^{i\Phi_n(\xi)} \, d\xi} \\
	&{}\le {L}\( \frac{1}{\inf_{L} |\Phi_n'|} + \sup_{L} \frac{|\Phi''_n|}{|\Phi_n'|} \)
	\to 0
\end{align*}
as $n\to\I$.
The proof for the case $\sup_n |\xi_n| <\I$ is similar.
In this case, we have $|s'_n|+|y'_n|\to\I$ as $n\to\I$ by orthogonality condition.
Since
\[
	\Phi_n'(\xi) = s'_n(3\xi^2 -6 \xi_n \xi) - y'_n,\quad
	\Phi_n''(\xi) = 6s'_n(\xi - \xi_n ),
\]
by removing from $L$ the constant $\eta$ such that
$\Phi_n'(\eta)\to0$ as $n\to\I$ if exists, we obtain the smallness.

Let us proceed to the proof of ``(iii)$\Rightarrow$(i)''.
Assume for contradiction that $h'_n$, $\xi'_n$, $s'_n$, $3s'_n \xi_n $, and $y'_n$ are 
uniformly bounded. Then, there exists a subsequence $n_k$ such that these parameters converge as $k\to\I$.
Denote the limits by $h'$, $\xi'$, $s'$, $\tau'$, and $y'$, respectively.
By refining subsequence if necessary, we may suppose that $e^{i\gamma_{n_k}}$ also converges.
In this case, for any $f \in \hat{L}^{\alpha}$ we have
\begin{equation}\label{eq:pf_orthty1}
	(\widetilde{\mathcal{G}}_{n_k})^{-1}
	\mathcal{G}_{n_k} f \to e^{i\gamma} D(h') P(\xi')
	A(s') S(\tau') T(y') f
\end{equation}
as $k\to \I$ strongly in $\hat{L}^\alpha$. 
Now, suppose that there exists a subsequence of $k$, which we denote again by $k$, such that
$(\mathcal{G}_{n_k})^{-1} u_k$ and $({G}_{n_k})^{-1} u_k$ converge weakly in $\hat{L}^\alpha$ to $\psi$ and $0$, respectively,
as $k\to\I$.
Since $(\mathcal{G}_{n_k})^{-1} u_k$ converges
to $\psi$ weakly in $\hat{L}^\alpha$, we see from \eqref{eq:pf_orthty1} that
\[
	(\widetilde{\mathcal{G}}_{n_k})^{-1} u_k = ((\widetilde{\mathcal{G}}_{n_k})^{-1} \mathcal{G}_n) (\mathcal{G}_n)^{-1} u_n \rightharpoonup 
	e^{i\gamma} D(h') P(\xi')
	A(s') S(\tau') T(y') \psi
\]
weakly in $\hat{L}^\alpha$. On the other hand, $(\widetilde{\mathcal{G}}_{n_k})^{-1} u_k  \rightharpoonup 0$ weakly in $\hat{L}^\alpha$ by assumption.
Thanks to uniqueness of weak limit, we see that
\[
	e^{i\gamma} D(h') P(\xi')
	A(s') S(\tau') T(y') \psi = 0.
\]
Thus, we obtain $\psi \equiv 0$, a contradiction. 
\end{proof}%$\qed$

\subsection{Decoupling inequality}
We next prove a decoupling inequality for $\ell$.
The idea of the proof is to sum up the local (in the Fourier side) $L^2$ decoupling
with respect to intervals.
\begin{lemma}[Decoupling inequality]\label{lem:decouple} 
Let $4/3<\alpha<2$ and $\alpha'<\sigma \le \frac{6\alpha}{3\alpha-2}$. 
Let $\{u_n\}_n$ be a bounded sequence in $ \hat{L}^{\alpha} %\hat{M}^{\alpha}_{2,\sigma}
$.
Suppose that $\mathcal{G}_n^{-1} u_n $ converges to $\psi$ weakly in $\hat{L}^{\alpha}$
as $n\to\I$ with some $\{\mathcal{G}_n\}_n \subset G$.
Set $r_n := u_n - \mathcal{G}_n \psi$.
Then, for any $\gamma>1$ and $\xi_0 \in \R$, it holds that
\begin{equation}\label{eq:pd:key_decouple}
	\gamma \norm{P(\xi_0) u_n}_{\hat{M}^\alpha_{2,\sigma}}^{\sigma}
	\ge \norm{P(\xi_0) \mathcal{G}_{n} \psi}_{\hat{M}^\alpha_{2,\sigma}}^{\sigma}
	 + \norm{P(\xi_0) r_n}_{\hat{M}^\alpha_{2,\sigma}}^{\sigma}
	+o_\gamma(1)
\end{equation}
as $n\to\I$.
\end{lemma}

\begin{proof}
%\noindent
%{\it Proof of Lemma \ref{lem:decouple}.}
We only consider the case $\xi_0 =0$.
The other cases handled in the same way because
the presence of $P(\xi_0)$ causes merely a universal translation in the Fourier side.
It is also clear from the proof that the small error term 
can be taken independently of $\xi_0$.

Denote $\mathcal{D}:=\{ \tau^l_k:=[k/2^l,(k+1)/2^l) \ |\ k,l\in \Z \}$.
For each $\tau^l_k \in \mathcal{D}$, we have the decoupling in $L^2$;
\begin{align*}
	\norm{\F u_n}_{L^2(\tau^l_k)}^2 = \norm{\F \mathcal{G}_{n}\psi}_{L^2(\tau^l_k)}^2 + \norm{\F r_n}_{L^2(\tau^l_k)}^2 + 2\Re
	\Jbr{\F \mathcal{G}_{n}\psi, \F r_n}_{\tau^l_k}.
\end{align*}
Let $\gamma=\frac{m+1}{m}, m>0$. 
By an elementary inequality
\[
	(a-b)^{\frac{\sigma}{2}} \ge \(\frac{m}{m+1}\)^{\frac{\sigma-2}{2}} a^{\frac{\sigma}2} - m^{\frac{\sigma-2}{2}} b^{\frac{\sigma}{2}}
\]
for any $a\ge b\ge0$ and $m>0$ and by embedding
$\ell^2_{\mathcal{D}} \hookrightarrow \ell^{\sigma}_{\mathcal{D}}$, it follows that
\begin{align*}
	&	\sum_{k,l \in\Z} |\tau^l_k|^{\sigma(\frac12-\frac1\alpha)} \norm{\F u_n}_{L^2(\tau_k^l)}^\sigma \\
	={}& \sum_{k,l \in\Z} |\tau^l_k|^{\sigma(\frac12-\frac1\alpha)} \( \norm{\F \mathcal{G}_n\psi}_{L^2(\tau_k^l)}^2 + \norm{\F r_n}_{L^2(\tau_k^l)}^2 + 2\Re
	\Jbr{\F \mathcal{G}_n\psi, \F r_n}_{\tau^l_k} \)^{\frac{\sigma}2}  \\
	\ge{}& \sum_{k,l \in\Z} |\tau^l_k|^{\sigma(\frac12-\frac1\alpha)} \( \norm{\F \mathcal{G}_n\psi}_{L^2(\tau_k^l)}^2 + \norm{\F r_n}_{L^2(\tau_k^l)}^2 - 2
	\abs{\Jbr{\F \mathcal{G}_n\psi, \F r_n}_{\tau^l_k}} \)^{\frac{\sigma}{2}} \\
	\ge{}& \(\frac{m}{m+1}\)^{\frac{\sigma-2}{2}} \sum_{k,l \in\Z} |\tau^l_k|^{\sigma(\frac12-\frac1\alpha)} \( \norm{\F \mathcal{G}_n\psi}_{L^2(\tau_k^l)}^2 + \norm{\F r_n}_{L^2(\tau_k^l)}^2 \)^{\frac{\sigma}{2}}\\
	&{}- 2^{\frac{\sigma}{2}}m^{\frac{\sigma-2}{2}} \sum_{k,l \in\Z} |\tau^l_k|^{\sigma(\frac12-\frac1\alpha)} \abs{\Jbr{\F \mathcal{G}_n\psi, \F r_n}_{\tau^l_k}}^{\frac{\sigma}{2}} \\
	\ge{}& \(\frac{m}{m+1}\)^{\frac{\sigma-2}{2}}\( \sum_{k,l \in\Z} |\tau^l_k|^{\sigma(\frac12-\frac1\alpha)} \norm{\F \mathcal{G}_n\psi}_{L^2(\tau_k^l)}^\sigma + \sum_{k,l \in\Z} |\tau^l_k|^{\sigma(\frac12-\frac1\alpha)} \norm{\F r_n}_{L^2(\tau_k^l)}^{\sigma} \) \\
	&{}- 2^{\frac{\sigma}{2}}m^{\frac{\sigma-2}{2}} \sum_{k,l \in\Z} |\tau^l_k|^{\sigma(\frac12-\frac1\alpha)} \abs{\Jbr{\F \mathcal{G}_n\psi, \F r_n}_{\tau^l_k}}^{\frac{\sigma}{2}}
\end{align*}
To obtain \eqref{eq:pd:key_decouple}, it therefore suffices to show that
\begin{equation}\label{eq:pd:key_pf_remainder1}
	R_n:=\sum_{k,l \in\Z} |\tau^l_k|^{\sigma(\frac12-\frac1\alpha)} \abs{\Jbr{\F \mathcal{G}_n\psi, \F r_n}_{\tau^l_k}}^{\frac{\sigma}{2}} \to 0
\end{equation}
as $n\to\infty$. A computation shows that
\[
	 |I|^{\sigma(\frac12-\frac1\alpha)} \abs{\Jbr{\F \mathcal{G}_n\psi, \F r_n}_I}^{\frac{\sigma}{2}}
	 = |J_n|^{\sigma(\frac12-\frac1\alpha)} \abs{\Jbr{\F\psi, \F (\mathcal{G}_n)^{-1} r_n}_{J_n}}^{\frac{\sigma}{2}}
\]
for any $I \subset \R$, where $J_n=I/h_n + \xi_n$ with the parameters $h_n,\xi_n$ associated 
with $\mathcal{G}_n$. 
By changing notation if necessary, one sees that
\[
	R_n= \sum_{k,l \in\Z} |\widetilde{\tau}^l_k|^{\sigma(\frac12-\frac1\alpha)} \abs{\Jbr{\F\psi^1, \F (\mathcal{G}^1_n)^{-1} r_n^1}_{\widetilde{\tau}^l_k}}^{\frac{\sigma}{2}},
\]
where $\widetilde{\tau}^l_k = \tau_k^l + 2^{-l}\sigma_n$ with some $0\le \sigma_n<1$.
Fix $\eps>0$. Since
	$\norm{\psi}_{\hat{M}^\alpha_{2,\sigma}}\le C \norm{\psi}_{\hat{L}^\alpha}<\I,$
there exist $k_0(\eps)$ and $l_0(\eps)$ such that
$D:=\{ |l| \le l_0, \, |k|\le k_0\} \subset \Z^2$ satisfies
\[
	\sum_{(k,l) \in \Z^2 \setminus D} |{\tau}^l_k|^{\sigma(\frac12-\frac1\alpha)} \norm{\F \psi}_{L^2(\tau_k^l)}^\sigma
	\le \eps.
\]
It is obvious that $\widetilde{\tau}^l_k \subset \tau^{l}_k \cup \tau^{l}_{k+1}$  and $|\widetilde{\tau}^l_k|=|\tau^l_k|=|\tau^l_{k+1}|$ for each $l,k$.
Hence, denoting $D':= \{ |l| \le l_0, \,  |k|\le k_0+1\} \subset \Z^2$, 
we have
\begin{eqnarray*}
	\lefteqn{\sum_{(k,l) \in \Z^2 \setminus D'} 
	|\widetilde{\tau}^l_k|^{\sigma(\frac12-\frac1\alpha)} \norm{\F \psi^{1}}_{L^2(\widetilde{\tau}_k^l)}^\sigma}\\
	&\le& \sum_{(k,l) \in \Z^2 \setminus D'} |\widetilde{\tau}^l_k|^{\sigma(\frac12-\frac1\alpha)} \( \norm{\F \psi^{1}}_{L^2({\tau}_k^l)}^2 + \norm{\F \psi^{1}}_{L^2({\tau}_{k+1}^l)}^2\)^{\sigma/2}\\
	&\le& 2^{\frac\sigma2}\sum_{(k,l) \in \Z^2 \setminus D} |{\tau}^l_k|^{\sigma(\frac12-\frac1\alpha)} \norm{\F \psi^{1}}_{L^2(\tau_k^l)}^\sigma \le C \eps.
\end{eqnarray*}
Then, by Schwartz' inequality,
\begin{eqnarray*}
\lefteqn{\sum_{(k,l) \in\Z^2 \setminus D'} 
	|\widetilde{\tau}^l_k|^{\sigma(\frac12-\frac1\alpha)} 
	\abs{\Jbr{\F\psi^{1}, \F (\mathcal{G}_n)^{-1} r_n^{1}
	}_{\widetilde{\tau}^l_k}}^{\frac{\sigma}{2}} }\\
	&\le& \sum_{(k,l) \in\Z^2 \setminus D'} |\widetilde{\tau}^l_k|^{\sigma(\frac12-\frac1\alpha)} \norm{\F\psi^{1}}_{L^2(\widetilde{\tau}^l_k)}^{\frac{\sigma}{2}}  \norm{\F (\mathcal{G}_n)^{-1} r_n^{1}}_{L^2(\widetilde{\tau}^l_k)}^{\frac{\sigma}{2}} \\
	&\le& \(\sum_{(k,l) \in\Z^2 \setminus D'} |\widetilde{\tau}^l_k|^{\sigma(\frac12-\frac1\alpha)} \norm{\F\psi^{1}}_{L^2(\widetilde{\tau}^l_k)}^\sigma \)^{\frac12}\\
	& &\qquad\times
	\(\sum_{(k,l) \in\Z^2 \setminus D'} |\widetilde{\tau}^l_k|^{\sigma(\frac12-\frac1\alpha)}  \norm{\F (\mathcal{G}_n)^{-1} r_n^{1}}_{L^2(\widetilde{\tau}^l_k)}^\sigma\)^{\frac12} \\
	&\le& C \eps^{\frac12}
	\norm{(\mathcal{G}_n)^{-1} r_n^{1}}_{\hat{M}^\alpha_{2,\sigma}
	}^{\frac{\sigma}{2}} 
	\le  C \eps^{\frac12} \norm{(\mathcal{G}_n)^{-1} r_n^{1}
	}_{\hat{L}^\alpha}^{\frac{\sigma}{2}}\\
	&=&C \eps^{\frac12} \norm{r_n^{1}}_{\hat{L}^\alpha}^{\frac{\sigma}{2}}.
\end{eqnarray*}
Remark that
\[
	\limsup_{n\to \I} \norm{r_n}_{\hat{L}^\alpha}
	\le \limsup_{n\to\I}\norm{u_n}_{\hat{L}^\alpha} + \norm{\psi}_{\hat{L}^\alpha}
	\le 2 \limsup_{n\to\I}\norm{u_n}_{\hat{L}^\alpha} \le C.
\]
Hence, the proof of \eqref{eq:pd:key_pf_remainder1} is reduced to showing
\begin{equation}\label{eq:pd:key_pf_remainder2}
	\sum_{(k,l) \in D'} |\widetilde{\tau}^l_k|^{\sigma(\frac12-\frac1\alpha)} \abs{\Jbr{\F\psi^{1}, \F (\mathcal{G}_n)^{-1} r_n^{1}}_{\widetilde{\tau}^l_k}}^{\frac{\sigma}{2}} \to 0
\end{equation}
as $n\to\I$. For $l \in [-l_0,l_0]$, consider a function
\[
	f_n^l (x) := \abs{\Jbr{\F\psi^{1}, \F (\mathcal{G}_n)^{-1} r_n^{1}}_{[x,x+2^{-l}]}}
\]
with domain $x \in [-k_0/2^l, (k_0+1)/2^l]$. Then, there exists a  constant $C=C(k_0,l_0)=C(\eps)$ such that
\begin{eqnarray*}
\lefteqn{\sum_{(k,l) \in D'} |\widetilde{\tau}^l_k|^{\sigma(\frac12-\frac1\alpha)} \abs{\Jbr{\F\psi^{1}, \F (\mathcal{G}_n)^{-1} r_n^{1}}_{\widetilde{\tau}^l_k}}^{\frac{\sigma}{2}}}\\
	&\le& C(\eps) \max_{l\in[-l_0,l_0]} \(\sup_{x\in [-k_0/2^l, (k_0+1)/2^l]} f^l_n(x)\).
\end{eqnarray*}
Therefore, we obtain \eqref{eq:pd:key_pf_remainder2} if we show the uniform convergence
\begin{equation}\label{eq:pd:key_pf_remainder3}
	\sup_{x\in [-k_0/2^l, (k_0+1)/2^l]} f^l_n(x) \to 0
\end{equation}
as $n\to\I$. 
Since $(\mathcal{G}_n)^{-1}r_n$ converges to zero weakly in $\hat{L}^\alpha$ as $n\to\I$ by definition, $\lim_{n\to\I} f_n(x) = 0$
follows for each $x$.
Further, by the H\"older inequality,
\begin{eqnarray*}
\lefteqn{|f^l_n(x+\delta) -f^l_n(x)|}\\
	&\le& C \(\sup_n \norm{(\mathcal{G}_n)^{-1} r_n^{1}}_{\hat{L}^\alpha}\) \norm{\F \psi^{1}}_{L^{\alpha}
	([x,x+\delta]\cup[x+2^{-l},x+2^{-l}+\delta])}
\end{eqnarray*}
for small $\delta>0$.
The right hand side is independent of $n$ and tends to zero as $\delta \downarrow 0$.
Therefore, $\{f_n^l\}_n$ is equicontinuous.
By a similar argument, $\sup_{x\in [-k_0/2^l, (k_0+1)/2^l]} f^l_n(x)$ is bounded uniformly in $n$.
% $\mathcal{G}^{-1}_n r_n \rightharpoonup 0$ weakly in $\hat{L}^{r/3}$ as $n\to\I$. 
Therefore, the Ascoli-Arzela theorem gives us the desired convergence \eqref{eq:pd:key_pf_remainder3}.
This completes the proof of Lemma \ref{lem:decouple}.
\end{proof}%$\qed$

\subsection{Decomposition procedure}
The main technical issue of Theorem \ref{thm:pd1} is essentially settled
with the above preliminaries and so
now the theorem follows by a standard argument (see \cite{CK} and references therein).
We give a complete proof for
self-containedness and in order to give
a complete proof for the decoupling inequality \eqref{eq:pd:key_Pythagorean}.

\vskip2mm
\noindent
{\it Proof of Theorem \ref{thm:pd1}.}
We may suppose $\eta(u)>0$, otherwise the result holds with $\phi^j \equiv 0$ and
$r_n^j = u_n$ for all $j\ge1$.
Then, we can choose $\psi^1\in \mathcal{V}(u)$ so that $\ell(\psi^1) \ge \frac12 \eta(u)$
by definition of $\eta$.
Then, by definition of $\mathcal{V}(u)$, one finds $\mathcal{G}^1_n \in G$ such that
\[
	(\mathcal{G}^1_n)^{-1} u_n \rightharpoonup \psi^1 \wIN \hat{L}^\alpha
\]
as $n\to\I$ up to subsequence.
By lower semicontinuity of weak limit,
we obtain \eqref{eq:pd:key_bddness1} for $j=1$.
Define $r^1_n := u_n - \mathcal{G}^1_n \psi^1 $.
Then, it is obvious that
\begin{equation}\label{eq:pd:key_pf_wlim1}
	(\mathcal{G}_{n}^1)^{-1} r^1_n \rightharpoonup  \psi^1 -\psi^1 =0 \wIN \hat{L}^{\alpha}
\end{equation}
as $n\to\I$.
The boundedness \eqref{eq:pd:key_bddness2} for $j=1$ is also obvious 
by \eqref{eq:pd:key_bddness1}.
By Lemma \ref{lem:decouple},
\begin{equation}\label{eq:pd:key_pf_Pythagoreanp}
	\gamma \norm{P(\xi_0) u_n}_{\hat{M}^\alpha_{2,\sigma}}^{\sigma}
	\ge \norm{P(\xi_0)\mathcal{G}^1_{n} \psi^1}_{\hat{M}^\alpha_{2,\sigma}}^{\sigma}
	 + \norm{P(\xi_0) r_n^1}_{\hat{M}^\alpha_{2,\sigma}}^{\sigma}
	+o_\gamma (1)
\end{equation}
as $n\to\I$ for any constant $\gamma>1$ and $\xi_0\in\R$.
Since $\gamma>1$ and $\xi_0$ are arbitrary, the decoupling inequality \eqref{eq:pd:key_Pythagorean} holds for $J=1$.

% Let us go back to the decomposition procedure.
If $\eta(r^1)=0$ then the proof is completed by taking $\psi^j\equiv0$ for $j\ge2$.
Otherwise, we can choose $\psi^2 \in \mathcal{V}(r^1)$ so that $\ell(\psi^2) \ge \frac12 \eta(r^1)$.
Then, as in the previous step, one can take
% $\Gamma^2 \in 2^\Z \times \R \times \R \times \R$
$\mathcal{G}_n^2 \in G$ so that
\[
	(\mathcal{G}^2_n)^{-1} r^1_n \rightharpoonup \psi^2 \wIN \hat{L}^{\alpha}
\]
as $n\to\I$ up to subsequence.
In particular, $\psi^2\not\equiv0$.
Together with \eqref{eq:pd:key_pf_wlim1}, Lemma \ref{lem:orthty} gives us that
two families $\mathcal{G}^1_n$ and $\mathcal{G}^2_n$ are orthogonal. 
Then,
\[
	(\mathcal{G}^2_n)^{-1} u_n
	= (\mathcal{G}^2_n)^{-1} \mathcal{G}^1_n \psi^1
	+ (\mathcal{G}^2_n)^{-1} r^1_n \rightharpoonup 0+\psi^2
	\wIN \hat{L}^\alpha
\]
as $n\to\I$. Hence, we obtain $\psi^2 \in \mathcal{V}(u)$ and so
\eqref{eq:pd:key_bddness1} for $j=2$.
Set $r_n^2:=r_n^1-\mathcal{G}^2_n \psi^2$.
Then, \eqref{eq:pd:key_bddness2} for $j=2$ follows from 
\[
	\limsup_{n\to\I} \norm{r^2_n}_{\hat{L}^\alpha}
	\le \limsup_{n\to\I} \norm{r^1_n}_{\hat{L}^\alpha}
	+ \norm{\psi^2}_{\hat{L}^\alpha} \le 3 \limsup_{n\to\I} \norm{u_n}_{\hat{L}^\alpha}.
\]
Further, one deduces from Lemma \ref{lem:decouple} that
\[
	\gamma \norm{P(\xi_0) r_n^1}_{\hat{M}^\alpha_{2,\sigma}}^\sigma \\
	\ge  \norm{P(\xi_0) \mathcal{G}^2_n \psi^2}_{\hat{M}^\alpha_{2,\sigma}}^\sigma
+ \norm{P(\xi_0) r_n^2}_{\hat{M}^\alpha_{2,\sigma}}^\sigma +o_\gamma (1)
\]
as $n\to\I$ for any $\gamma>1$ and $\xi_0 \in \R$. This implies 
\eqref{eq:pd:key_Pythagorean}
for $J=2$ with the help of \eqref{eq:pd:key_pf_Pythagoreanp}.

Repeat this argument and construct $\psi^j \in \mathcal{V}(r^{j-1})$ and  $\mathcal{G}^j_n\in G$, inductively.
If we have $\eta(r^{j_0})=0$ for some $j_0$, then we define $\psi^j \equiv 0$ for 
$j\ge j_0+1$.
In what follows, we may suppose that $\eta(r^{j})>0$ for all $j\ge1$.
In each step, 
$r_n^j$ is defined by the formula $r_n^j = r_n^{j-1} - \mathcal{G}^j_n \psi^j$.
The property \eqref{eq:pd:key_decomp} is obvious by construction. 

Let us now prove that pairwise orthogonality.
To this end, we demonstrate that 
$\mathcal{G}^j_n$ is orthogonal to  $\mathcal{G}^k_n$ for $1\le k \le j-1$.
Since $(\mathcal{G}^j_n)^{-1} r_n^j \rightharpoonup \psi^j$
and $(\mathcal{G}^{j-1}_n)^{-1} r_n^j \rightharpoonup 0$ 
in $\hat{L}^{\alpha}$ as $n\to\I$,
Lemma \ref{lem:orthty} implies that $\mathcal{G}^j_n$ and $\mathcal{G}^{j-1}_n$ are orthogonal.
If $\mathcal{G}^j_n$ is orthogonal to  $\mathcal{G}^k_n$ for $k_0\le k \le j-1$ then
Lemma \ref{lem:orthty} yields
\[
	(\mathcal{G}^j_n)^{-1} r_n^{k_0-1} = \sum_{k=k_0}^{j-1} (\mathcal{G}^j_n)^{-1}\mathcal{G}^k_n \psi^{k} + (\mathcal{G}^j_n)^{-1} r_n^{j-1}
	\rightharpoonup \psi^j
\]
as $n\to\I$. On the other hand, $(\mathcal{G}^{k_0-1}_n)^{-1} r_n^{k_0-1} \rightharpoonup 0$ as $n\to\I$.
We therefore see from Lemma \ref{lem:orthty} 
that $\mathcal{G}^j_n$ and $\mathcal{G}^{k_0-1}_n$ are orthogonal.
Hence, $\mathcal{G}^j_n$ is orthogonal to $\mathcal{G}^k_n$ for $1\le k \le j-1$.
Then, by \eqref{eq:pd:key_decomp} and
by Lemma \ref{lem:orthty}, we have $\psi^j \in \mathcal{V}(u)$, from which
boundedness \eqref{eq:pd:key_bddness1}  and \eqref{eq:pd:key_bddness2} follow.

To conclude the proof,
we shall show \eqref{eq:pd:key_smallness} and \eqref{eq:pd:key_Pythagorean}.
Notice that the inductive construction gives us
\begin{equation}\label{eq:pd:key_pf_complete1}
	\ell (\psi^{j+1}) \ge \frac12 \eta(r^j)	
\end{equation}
for $j\ge1$ and
\begin{eqnarray}
	\lefteqn{\gamma \norm{P(\xi_0) r_n^j}_{\hat{M}^\alpha_{2,\sigma}}^\sigma}
	\label{eq:pd:key_pf_complete2}\\
	&\ge& \norm{P(\xi_0) \mathcal{G}^{j+1}_n \psi^{j+1}}_{\hat{M}^\alpha_{2,\sigma}}^\sigma
	 + \norm{P(\xi_0) r_n^{j+1}}_{\hat{M}^\alpha_{2,\sigma}}^\sigma
	+o_{\gamma,j} (1).\nonumber
\end{eqnarray}
as $n\to\I$ for (fixed) $j\ge1$ and any $\gamma>1$ and $\xi_0 \in \R$.
Combining \eqref{eq:pd:key_pf_Pythagoreanp} and 
\eqref{eq:pd:key_pf_complete2} for $1 \le j \le J$, we have
\begin{align*}
	\gamma^J 
	\norm{P(\xi_0) u_n}_{\hat{M}^\alpha_{2,\sigma}}^\sigma
	&{}\ge \sum_{j=1}^J \gamma^{J-j} 
	\norm{P(\xi_0) \mathcal{G}^j_n \psi^{j}}_{\hat{M}^\alpha_{2,\sigma}}^\sigma + \norm{P(\xi_0) r_n^{J}}_{\hat{M}^\alpha_{2,\sigma}}^\sigma +o_{\gamma,J} (1) \\
	&{}\ge \sum_{j=1}^J \gamma^{J-j} \ell(\psi^j)^\sigma + \ell(r_n^J)^\sigma +o_{\gamma,J} (1) .
\end{align*}
Take first infimum with respect to $\xi_0$ and then limit supremum in $n$ to obtain
\[
	\limsup_{n\to\I}\ell(u_n)^\sigma \ge \sum_{j=1}^J \gamma^{-j} \ell(\psi^j)^\sigma
	+\gamma^{-J} \limsup_{n\to\I}\ell(r_n^J)^\sigma.
\]
Since $\gamma>1$ is arbitrary, we obtain \eqref{eq:pd:key_Pythagorean}.
Finally, \eqref{eq:pd:key_Pythagorean} and \eqref{eq:pd:key_pf_complete1} imply \eqref{eq:pd:key_smallness}.
This completes the proof of Theorem \ref{thm:pd1}.

%%%%%%%%%%%%%%%%%%%%%%%%%%%%%%%%%%%%%%%%%%%%%%%%%%%%
%
%     Chapter 6
%
%%%%%%%%%%%%%%%%%%%%%%%%%%%%%%%%%%%%%%%%%%%%%%%%%%%%
\section{Concentration compactness}\label{sec:cc}

The second part of the proof of Theorem \ref{thm:pd_r}
is concentration compactness.
Intuitively, the meaning of the concentration compactness is as follows.
Let us consider a bonded sequence $\{ u_n \}_n \subset X$.
Here, $X$ is a Banach space.
In addition to the boundedness with respect to $X$, we make some additional assumption
on the sequence.
If the additional assumption is so strong that it removes almost all possible 
deformations
for $\{ u_n \}_n$ with few exceptions, say $G$, then
we can find a \emph{non-zero} weak limit modulo $G$.
In our case, $X=\hat{M}^\alpha_{2,\sigma}$
and we use
\[
	\norm{e^{-t\d_x^3} u_n}_{L(\R)} \ge m 
\]
as an additional assumption, where $m$ is some positive constant.
It will turn out that this assumption removes almost all deformations.
The exception is $G$ given in \eqref{eq:symG}.
This is the reason why we use the set $G$ of deformations in Theorems \ref{thm:pd_r}
or \ref{thm:pd1}.
The precise statement is as follows.
\begin{theorem}[Concentration compactness]\label{thm:cc}
Let $4/3<\alpha<2$ and $\alpha'<\sigma<\frac{6\alpha}{3\alpha-2}$.
Let a bounded sequence $\{u_n\} \subset \hat{L}^{\alpha}$ satisfy
\begin{eqnarray}
	\norm{u_n}_{\hat{M}^\alpha_{2,\sigma}}
	\le M\label{upM}
\end{eqnarray}
and
\begin{eqnarray}
	\norm{ e^{-t\d_x^3} u_n}_{L(\R)} \ge m
	\label{belowm}
\end{eqnarray}
for some positive constants $m,M$. Then, there exist $\mathcal{G}_n \in G$
and $\psi \in \hat{L}^{\alpha}$ such that, up to subsequence,
$\mathcal{G}_{n}^{-1} u_n \rightharpoonup \psi$ weakly in $\hat{L}^\alpha$
as $n\to\I$ and $\norm{u_n}_{\hat{M}^\alpha_{2,\sigma}} \ge \beta(m,M)$,
where $\beta(m,M)$ is a positive constant depending only on $m,M$. 
In particular,
$\eta(u) \ge C\beta(m,M)$ holds for some constant $C$.
\end{theorem}

Plugging Theorem \ref{thm:cc} to Theorem \ref{thm:pd1},
we obtain desired decomposition result.
\begin{theorem}[Decomposition of ``complex-valued'' 
functions]\label{thm:pd}
Let $4/3<\alpha<2$ and $\sigma\in(\alpha',\frac{6\alpha}{3\alpha-2})$.
Let $u=\{u_n\}_n$ be a bounded sequence of $\C$-valued functions in $\hat{L}^{\alpha}$.
Then, there exist $\{\psi^j\}_j, \{r_n^j\}_{n,j} \subset \hat{L}^\alpha$ and 
pairwise orthogonal families $\{\mathcal{G}_n^j\}_n \subset G$ $(j=1,2,\dots)$
such that, up to subsequence,
\begin{equation*}%\label{eq:pd:key_decomp}
	u_n = \sum_{j=1}^l \mathcal{G}_n^j \psi^j + r_n^l
\end{equation*}
for all $l\ge1$ with
\begin{equation}\label{eq:pd:limit}
	\limsup_{n\to\I} \norm{ e^{-t\d_x^3} r_n^l }_{L(\R)} \to 0
\end{equation}
as $l\to\I$. Further,  the decouple inequality
\begin{equation*}%\label{eq:pd:key_Pythagorean}
	\limsup_{n\to\I} \ell(u_n)^\sigma \ge \sum_{j=1}^\I \ell(\psi^j)^\sigma  +
	\limsup_{n\to\I} \ell(r_n^J)^\sigma
\end{equation*}
holds for any $J \ge 1$.
Moreover, it holds that
\begin{equation*}%\label{eq:pd:key_bddness1}
	\norm{\psi^j}_{\hat{L}^{\alpha}} \le \limsup_{n\to\I} \norm{u_n}_{\hat{L}^{\alpha}}
\end{equation*}
for any $j$.
\end{theorem}

% 今後追加する予定
% Another application would be as follows
% \begin{corollary}
% Let $u=\{u_n\}\subset \hat{L}^\alpha$ be a bounded sequence.
% If $\eta(u) = 0$ then
% \[
% 	\lim_{n\to0} \norm{e^{-t\d_x} u_n}_{L(\R)} = 0.
% \]
% \end{corollary}

Before proceeding to the proof of Theorem \ref{thm:cc},
we demonstrate how
we derive Theorem \ref{thm:pd} from Theorems \ref{thm:pd1} and \ref{thm:cc}.

\begin{proof}
%\vskip2mm
%\noindent
%{\it Proof of Theorem \ref{thm:pd}.} 
By means of Theorem \ref{thm:pd1},
it suffices to show \eqref{eq:pd:limit} as $l\to\infty$
%that \eqref{eq:pd:key_smallness} implies the convergences
%\begin{equation*}%\label{eq:pd:key_smallness}
%	\limsup_{n\to\I} \norm{ e^{-t\d_x^3} r_n^l }_{L(\R)} \to 0
%\end{equation*}
%as $l\to\I$.
Assume for contradiction that a sequence $r_n^l$ given in Theorem \ref{thm:pd1} satisfies
\[
	\limsup_{l\to\I} \limsup_{n\to\I} 
	\norm{e^{-t\d_x^3} r_n^l}_{L(\R)} >0.
\]
Then, we can choose $m>0$ and a subsequence $l_k$ with $l_k\to\I$ as $k\to\I$ such that
the assumption of Theorem \ref{thm:cc} is fulfilled for each $k$.
Then, Theorem \ref{thm:cc} implies $\eta(r^{l_k})\ge C \beta>0$,
which contradicts to \eqref{eq:pd:key_smallness}.
\end{proof}%$\qed$

\begin{remark}
We would emphasize that, in Theorem \ref{thm:cc},
$\{u_n\}_n$ should be a bounded sequence of $\hat{L}^{\alpha}$ functions
but the constant $\beta$ is chosen independently of the value of 
$\limsup_{n\to\I} \norm{u_n}_{\hat{L}^{\alpha}}$.
This respect is crucial to obtain Theorem \ref{thm:pd}
 because an $\hat{L}^{\alpha}$-bound on
$r_n^J$ given in Theorem \ref{thm:pd1} is no more than
\eqref{eq:pd:key_bddness2}.
\end{remark}
By a  similar argument, we obtain Theorem \ref{thm:BSS}.
We recall the theorem in terms of $\eta$.

% Combining Theorem \ref{thm:cc} and a version of small data scattering
% (Corollary \ref{cor:sds}), we have the following.
\begin{theorem}[Scattering due to irrelevant deformations]\label{thm:bss2}
Let $\{u_{0,n}\}_n \subset \hat{L}^\alpha$ be a bounded sequence.
Let $u_n(t)$ be a solution to \eqref{gKdV} with $u_n(0)=u_{0,n}$.
If $\eta(\{u_{0,n}\}_n)=0$ then there exists $N_0$ such that
$u_n(t)$ is global and scatters for both time direction as long as $n\ge N_0$.
Furthermore,
\[
	\norm{u_n}_{S(\R)} +
	\norm{u_n}_{L(\R)} \le 2\limsup_{n\to\I} \norm{u_{0,n}}_{\hat{L}^\alpha}
\]
for $n\ge N_0$.
\end{theorem}

\begin{proof}
%\vskip2mm
%\noindent
%{\it Proof of Theorem \ref{thm:bss2}.} 
Just as in the proof of Theorem \ref{thm:pd}, we deduce from $\eta(\{u_{0,n}\}_n)=0$ that
\begin{eqnarray*}
	\lim_{n\to\I} \norm{|\d_x|^{\frac1{3\alpha}} e^{-t\d_x^3} u_{0,n}}_{L^{3\alpha}_{t,x}(\R\times \R)}
	=0
\end{eqnarray*}
thanks to Theorem \ref{thm:cc}.
Then, the result follows from Corollary \ref{cor:sds}.
\end{proof}%$\qed$

\subsection{Refined Stein-Tomas estimate}
One of the key for the proof Theorem \ref{thm:cc} is the following refinement
of \eqref{ho}.
\begin{theorem}\label{thm:rST}
Let $4/3<\alpha<2$ and let $\sigma$ satisfy
$\sigma\in(0,\frac{6\alpha}{3\alpha-2})$. 
Then, there exists a positive constant $C=C(\alpha)$ such that
\begin{equation}\label{eq:ST3}
\begin{aligned}
	\norm{ e^{-t\d_x^3} f }_{ L(\R) } 
	\le{}& C \norm{f}_{\hat{M}^{\alpha}_{\frac{3\alpha}{2},\frac{6\alpha}{3\alpha-2}}}\\
	\le{}& 
	C \norm{f}_{\hat{M}^{\alpha}_{\frac{3\alpha}{2},\I}}^{1-\sigma(\frac12-\frac1{3\alpha})} \norm{f}_{\hat{M}^{\alpha}_{\frac{3\alpha}{2},\sigma}}^{\sigma(\frac12-\frac1{3\alpha})}
\\
	\le{}& C \norm{f}_{\hat{M}^{\alpha}_{\frac{3\alpha}{2},\I}}^{1-\sigma(\frac12-\frac1{3\alpha})} \norm{f}_{\hat{M}^{\alpha}_{2,\sigma}}^{\sigma(\frac12 -\frac1{3\alpha})} \\
	\le{}& C \norm{f}_{\hat{M}^{\alpha}_{2,\sigma}} 
% 	\\	\le {}& C \norm{f}_{\hat{L}^{\alpha}}
\end{aligned}
\end{equation}
for any %$f \in \hat{L}^{\alpha}$.
$f \in \hat{M}^{\alpha}_{2,\sigma}$.
\end{theorem}
This kind of refinement of the Airy Strichartz's inequality,
is known in the case $\alpha=2$ (see \cite{KPV3,Shao}).
We prove the first inequality of \eqref{eq:ST3} in Appendix \ref{sec:rST}.
The others follow from embeddings in Remark \ref{rem:gMorrey} in Section 2.
\subsection{A stronger orthogonality}
We have settled the notion of orthogonality of two families of deformations
in Definition \ref{def:orthty}. 
As seen in Lemma \ref{lem:orthty}, the orthogonality is rather a property
associated with the functions of space $\hat{L}^\alpha$.
The notion is not sufficient to yield weakness of interaction between
Airy evolutions of them.
A simple but essential example is 
\[
	u_n = P(n) f_+ + P(-n)f_- ,
\]
where $f_\pm \in \mathcal{S}$ is a real-valued function good enough. 
Remark that $\{P(n)\}_n$ and $\{P(-n)\}_n$ are orthogonal.
As for the Schr\"odinger evolution of $u_n$,
we have the following decoupling with respect to a space-time norm
\[
	\norm{S(t) u_n}_{L^{p}_{t,x}(\R^2)}^p
	= \norm{S(t) f_+}_{L^{p}_{t,x}(\R^2)}^p + \norm{S(t) f_-}_{L^{p}_{t,x}(\R^2)}^p + o(1).
\]
for $4 < p <\I$ as long as $\norm{S(t) f_\pm}_{L^{p}_{t,x}(\R^2)}$ are finite.
Indeed, the decoupling follows from
\[
	|S(\cdot)P(\pm n)f_\pm| (t,x) = |S(\cdot)f|(t,x \pm 2tn),
\]
which is a consequence of the Galilean transform \eqref{eq:Gt}.
In contrast, the Airy evolution of $u_n$ may not satisfy such kind of
space-time decoupling because
\[
	|A(\cdot) P(-n) f_-|(t,x) = |\overline{A(\cdot) P(n) f_-}|(t,x) = |A(\cdot) P(n) f_-|(t,x) .
\]
Thus, interaction betwenn $A(\cdot) P(-n) f_-$ and
$|A(\cdot) P(-n) f_+|$ is not always small.

Hence, we introduce a stronger notion of orthogonality which yields
such a decoupling with respect to a space-time norm of corresponding Airy evolutions.
\begin{definition}\label{def:orthty2}
Let $\{\mathcal{G}_n\}_n , \{\widetilde{\mathcal{G}}_n\}_n \subset G$ be
two families of deformations and let
$\Gamma_n ,\widetilde{\Gamma}_n 
\in 2^\Z \times \R\times \R \times \R$ be parameters associated with $\mathcal{G}_n$ and
$\widetilde{\mathcal{G}}_n$, respectively.
We say $\{\mathcal{G}_n\}_n$ and $\{\widetilde{\mathcal{G}}_n\}_n$ are 
\emph{space-time nonresonant} if 
\begin{multline}\label{eq:def_orthty2}
	\limsup_{n\to\I} \bigg(\abs{\log \frac{h_n}{\widetilde{h}_n}}+\abs{ |\xi_n| - \frac{\widetilde{h}_n}{h_n} |\widetilde{\xi}_n| } 
	+ \abs{ s_n -\( \frac{h_n}{\widetilde{h}_n}\)^3\widetilde{s}_n}(1+|\xi_n|) \\
	+ \abs{y_n- \frac{h_n}{\widetilde{h}_n} \widetilde{y}_n-3\( s_n -\( \frac{h_n}{\widetilde{h}_n}\)^3\widetilde{s}_n\)(\xi_n)^2 } \bigg)= +\I.
\end{multline}
\end{definition}

\begin{remark}\label{rem:realorthty}
(i) Obviously, if $\xi_n^j, \xi_n^k \ge 0$ then
the orthogonality and the space-time nonresonant property 
of $\mathcal{G}^j_n$ and $\mathcal{G}^k_n$ are equivalent.
Hence, we can replace the word ``orthogonal'' with ``space-time nonresonant''
in Theorem \ref{thm:pd_r}.

\vskip1mm
\noindent
(ii) Let $\{\mathcal{G}_n^j \}_{n}  \subset G$ ($j=1,2,\dots$) be pairwise orthogonal families.
For each $\{\mathcal{G}_n^j\}_n$, the family $\{\mathcal{G}_n^k\}_n$ that is orthogonal but not space-time nonresonant
is at most one.
Indeed, if $\{\mathcal{G}_n^j\}_n$ and $\{\mathcal{G}_n^k\}_n$ are orthogonal and not space-time nonresonant, then
\begin{multline*} %\label{eq:def_orthty2}
	\limsup_{n\to\I} \bigg(\abs{\log \frac{h_n^j}{{h}_n^k}}+\abs{ \xi_n^j + \frac{{h}_n^k}{h_n^j} {\xi}_n^k } 
	+ \abs{ s_n^j -\( \frac{h_n^j}{{h}_n^k}\)^3{s}_n^k}(1+|\xi_n^j|) \\
	+ \abs{y_n^j - \frac{h_n^j}{{h}_n^k} {y}_n^k-3\( s_n^j -\( \frac{h_n^j}{{h}_n^k}\)^3{s}_n^k\)(\xi_n^j)^2 } \bigg)< +\I
\end{multline*}
and 
\[
	\limsup_{n\to\I} \abs{ \xi_n^j - \frac{{h}_n^k}{h_n^j} {\xi}_n^k } =+\I.
\]
To simplify the formulation, we may assume that $h_n^j=h_n^k$.
If both $\{\mathcal{G}_n^{k_1}\}_n$ and $\{\mathcal{G}_n^{k_2}\}_n$ ($k_1\neq k_2$) satisfy the above,
then $h_n^{k_1}=h_n^{k_2}$ and
\begin{eqnarray*} %\label{eq:def_orthty2}
	\lefteqn{\limsup_{n\to\I} \bigg( %\abs{ \xi_n^{k_1} +  {\xi}_n^{k_2} }  +
	\abs{ s_n^{k_1} -{s}_n^{k_2}}(1+|\xi_n^{k_1}|) }\\
	& &\qquad\quad+ \abs{y_n^{k_1} - {y}_n^{k_2}-3\( s_n^{k_1} -{s}_n^{k_2}\)(\xi_n^{k_1})^2 } \bigg)< +\I
\end{eqnarray*}
and
\[
	|\xi_n^{k_1} - \xi_n^{k_2}| \le |\xi_n^{k_1} + \xi_n^j| + |\xi_n^{k_2} + \xi_n^j| \le C <\I.
\]
These inequalities imply $\{\mathcal{G}_n^{k_1}\}_n$ and $\{\mathcal{G}_n^{k_2}\}_n$ are not orthogonal,
a contradiction.
A similar argument shows that orthogonality and space-time nonresonant property 
of $\{\mathcal{G}^j_n\}_n$ and $\{\mathcal{G}^k_n\}$ are equivalent
as long as $\{\xi_n^j\}_n$ or $\{\xi_n^k\}_n$ is bounded.
\end{remark}

The following is the main conclusion of space-time nonresonant property of two families.

\begin{lemma}\label{lem:realorthty}
Suppose that $\{\mathcal{G}_n\}_n , \{\widetilde{\mathcal{G}}_n \}_n
\subset G$ are space-time nonresonant families.
Then, it holds for any $\psi,\widetilde{\psi} \in \hat{L}^{\alpha}$ that
\[
	\norm{[|\d_x|^{\frac1{3\alpha}} e^{-t\d_x^3} \mathcal{G}_{n} \psi]
	\overline{[|\d_x|^{\frac1{3\alpha}} e^{-t\d_x^3} \widetilde{\mathcal{G}}_{n} \widetilde{\psi}]}}_{L^{\frac{3\alpha}2}_{t,x}}
	\to 0
\]
as $n\to\I$.
\end{lemma}

\begin{proof}
%\vskip2mm
%\noindent
%{\it Proof of Lemma \ref{lem:realorthty}.} 
By a density argument, we assume that $\psi$ and $\widetilde{\psi}$ are Schwartz
function with compact Fourier support. Let $K$ be a compact set
containing the Fourier support of $\psi$ and $\widetilde{\psi}$.
Let $\Gamma_n=(h_n,\xi_n,y_n,s_n) ,\widetilde{\Gamma}_n=(\widetilde{h}_n,
\widetilde{\xi}_n,\widetilde{y}_n,\widetilde{s}_n)
\in 2^\Z \times \R\times \R \times \R$ be  parameters
associated with
$\{\mathcal{G}_n\}_n , \{\widetilde{\mathcal{G}}_n \}_n
\subset G$, respectively.

We first consider the case 
\[
	\limsup_{n\to\I} \bigg(\abs{\log \frac{h_n}{\widetilde{h}_n}}+ \abs{ |\xi_n| -\frac{\widetilde{h}_n}{h_n} |\widetilde{\xi}_n| }  \bigg)= +\I.
\]
Let us begin with the case when $\abs{\log \frac{h_n}{\widetilde{h}_n}}\to\I$ as $n\to\I$.
The integrand equals
\[
	\iint |\xi|^{\frac1{3\alpha}}|\eta|^{\frac1{3\alpha}} e^{ix(\xi-\eta)+ it(\xi^3-\eta^3)} \F [\mathcal{G}_{n} \psi](\xi) \overline{\F [\widetilde{\mathcal{G}}_{n} \widetilde{\psi}](\eta)} \, d\xi d\eta.
\]
Then, by Hausforff-Young inequality, our goal is to show
\begin{eqnarray}
	\ \ \ \ \iint_{K\times K} \frac{ |\xi|^{\frac1{3\alpha-2}} |\eta|^{\frac1{3\alpha-2}} |\F [\mathcal{G}_{n} \psi](\xi)|^{\frac{3\alpha}{3\alpha-2}}
	|\F [\widetilde{\mathcal{G}}_{n} \widetilde{\psi}](\eta)|^{\frac{3\alpha}{3\alpha-2}} }{|\xi+\eta|^{\frac2{3\alpha-2}}|\xi-\eta|^{\frac2{3\alpha-2}}} \, d\xi d\eta \to0\label{conv}
\end{eqnarray}
as $n\to\I$. As in \cite[Lemma 2.2]{MS}, it holds that
\[
		\iint_{K\times K} \frac{ |\xi|^{\frac1{3\alpha-2}} |\eta|^{\frac1{3\alpha-2}} |f(\xi)|^{\frac{3\alpha}{3\alpha-2}}
	|g(\eta)|^{\frac{3\alpha}{3\alpha-2}} }{|\xi+\eta|^{\frac2{3\alpha-2}}|\xi-\eta|^{\frac2{3\alpha-2}}} \, d\xi d\eta
	\le C \norm{f}_{L^{q_1}}^{\frac{3\alpha}{3\alpha-2}} \norm{g}_{L^{q_2}}^{\frac{3\alpha}{3\alpha-2}},
\]
where $q_1, q_2 \in (\frac{3\alpha}{3\alpha-2}, \frac{3\alpha}{3\alpha-4})$ satisfies
\[
	\frac1{q_1} + \frac1{q_2} = \frac{2}{\alpha'}.
\]	
% Before showing the above claim, we shall see it completes the proof.
Remark that
\begin{equation}\label{eq:ro2_pf1}
	|\F [\mathcal{G}_{n} \psi](\xi)|
	= (\hat{D}(h_n) \hat{P}(\xi_n) |\F \psi|) (\xi)
	= h_n^{-\frac1{\alpha'}} |\F \psi|\(\frac{\xi}{h_n} + \xi_n \).
\end{equation}
Similar formula holds for $|\F [\widetilde{\mathcal{G}}_{n} \widetilde\psi](\eta)|$.
Therefore,
\[
	\norm{\F [\mathcal{G}_{n} \psi]}_{L^{q_1}}
	\norm{\F [\widetilde{\mathcal{G}}_{n} \widetilde\psi]}_{L^{q_2}}
	= \(\frac{h_n}{\widetilde{h_n}}\)^{\frac1{q_1}-\frac{1}{\alpha'}}
	\norm{\F \psi}_{L^{q_1}}\norm{\F \widetilde\psi}_{L^{q_2}}.
\]
We have the desired smallness by taking ${q_1}<\alpha'$ if
$h_n/\widetilde{h}_n \to 0$ as $n\to\I$ and ${q_1}>\alpha'$ otherwise.

We proceed to the case when $\limsup_{n\to\I} \abs{\log \frac{h_n}{\widetilde{h}_n}}<\I$.
Changing notations and taking subsequence if necessary, we may suppose $\widetilde{h}_n=h_n$.
Recall that we need to show (\ref{conv}) as $n\to\infty$.
%\[
%	\iint \frac{ |\xi|^{\frac1{3\alpha-2}} |\eta|^{\frac1{3\alpha-2}} |\F [\mathcal{G}_{n} \psi](\xi)|^{\frac{3\alpha}{3\alpha-2}}
%	|\F [\widetilde{\mathcal{G}}_{n} \widetilde{\psi}](\eta)|^{\frac{3\alpha}{3\alpha-2}} }{|\xi+\eta|^{\frac2{3\alpha-2}}|\xi-\eta|^{\frac2{3\alpha-2}}} \, d\xi d\eta \to0
%\]
%as $n\to\I$.
By \eqref{eq:ro2_pf1} and by change of variable,
\begin{eqnarray*}
\lefteqn{(\text{LHS\ of}\ (\ref{conv}))}\\
%	& 	\iint \frac{ |\xi|^{\frac1{3\alpha-2}} |\eta|^{\frac1{3\alpha-2}} |\F [\mathcal{G}_{n} \psi](\xi)|^{\frac{3\alpha}{3\alpha-2}}
%	|\F [\widetilde{\mathcal{G}}_{n} \widetilde{\psi}](\eta)|^{\frac{3\alpha}{3\alpha-2}} }{|\xi+\eta|^{\frac2{3\alpha-2}}|\xi-\eta|^{\frac2{3\alpha-2}}} \, d\xi d\eta\\
	&=& (h_{n})^{\frac{4(\alpha-1)}{\alpha(3\alpha-2)}}
	\iint_{K\times K} \frac{ |\xi-\xi_n|^{\frac1{3\alpha-2}} |\eta-\widetilde\xi_n|^{\frac1{3\alpha-2}} |\F \psi(\xi)|^{\frac{3\alpha}{3\alpha-2}}
	|\F \widetilde{\psi}(\eta)|^{\frac{3\alpha}{3\alpha-2}} }{|\xi+\eta-(\xi_n +\widetilde\xi_n)|^{\frac2{3\alpha-2}}|\xi -\eta-(\xi_n -\widetilde\xi_n)|^{\frac2{3\alpha-2}}} \, d\xi d\eta\\
	&\le &  (h_{n})^{\frac{4(\alpha-1)}{\alpha(3\alpha-2)}}
	\( \sup_{(\xi,\eta) \in K\times K}\frac{ |\xi-\xi_n| |\eta-\widetilde\xi_n| }{ |\xi+\eta-(\xi_n +\widetilde\xi_n)|^2 |\xi -\eta-(\xi_n -\widetilde\xi_n)|^2} \)^{\frac1{3\alpha-2}} \\
	& &\qquad \times \iint_{K\times K}|\F \psi(\xi)|^{\frac{3\alpha}{3\alpha-2}}
	|\F \widetilde{\psi}(\eta)|^{\frac{3\alpha}{3\alpha-2}} \, d\xi d\eta.
\end{eqnarray*}
Notice that $||\xi_n|-|\widetilde\xi_n|| \to \I$ as $n\to\I$ by  orthogonality assumption.
We have
\[
	\max(|\xi|,|\widetilde\xi_n|)
	\le \frac12 (|\xi_n + \widetilde{\xi}_n| + |\xi_n - \widetilde{\xi}_n|).
\]
It therefore holds that
\[
	\frac{(1+ |\xi_n|)(1+ | \widetilde\xi_n|)}{|\xi_n-\widetilde\xi_n|^2|\xi_n+\widetilde\xi_n|^2}\le
	\frac1{(\min (|\xi_n-\widetilde\xi_n|,|\xi_n+\widetilde\xi_n|))^2}
	=\frac1{\abs{|\xi_n|-|\widetilde\xi_n|}^2}
\]
Hence, for large $n$,
\begin{eqnarray*}
\lefteqn{\sup_{(\xi,\eta) \in K\times K} \frac{ |\xi-\xi_n| |\eta-\widetilde\xi_n| }{ |\xi+\eta-(\xi_n+\widetilde\xi_n)|^2 |\xi -\eta-(\xi_n -\widetilde\xi_n)|^2}} \\
	&\le& \sup_{(\xi,\eta) \in K\times K} \frac{  2(C_K+|\xi_n|)(C_K+|\widetilde\xi_n|) }{ |\xi_n + \widetilde\xi_n|^2 |\xi_n -\widetilde\xi_n|^2}
	\le \frac{C_K}{\abs{|\xi_n|-|\widetilde\xi_n|}^2}.
\end{eqnarray*}
Therefore, we obtain the desired smallness (\ref{conv}).

Next, we assume that %the families $\mathcal{G}_n, \widetilde{\mathcal{G}}_n$ satisfy
$h_n = \widetilde{h_n}$, $\xi_n = \widetilde{\xi}_n$, and
\[
	\limsup_{n\to\I} \bigg( \abs{ s_n - \widetilde{s}_n}(1+|\xi_n|)
	+ \abs{y_n- \widetilde{y}_n-3( s_n - \widetilde{s}_n)(\xi_n)^2 } \bigg)= +\I.
\]
First, we further assume that $|\xi_n|\le C$. 
Then, we may let $\xi_n\equiv 0$ by extracting subsequence and changing notations.
In this case, the orthogonality implies $|s_n-\widetilde{s}_n| + |y_n-\widetilde{y}_n| \to \I$
as $n\to\I$. Since
\[
	|\d_x|^{\frac1{3\alpha}}A(t) \mathcal{G}_{n} \psi = 
	(h_n)^{\frac{4}{3\alpha}} [ |\d_x|^{\frac1{3\alpha}} A( \cdot )  \psi ] ((h_n)^3 t + s_n , h_n x - y_n)
\]
and a similar formula for $|\d_x|^{\frac1{3\alpha}}A(t) \widetilde{\mathcal{G}}_{n} \widetilde{\psi}$ hold, we see from change of variable that
\begin{multline*}
	\norm{[|\d_x|^{\frac1{3\alpha}} e^{-t\d_x^3} \mathcal{G}_{n} \psi]
	\overline{[|\d_x|^{\frac1{3\alpha}} e^{-t\d_x^3} \widetilde{\mathcal{G}}_{n} \widetilde{\psi}]}}_{L^{\frac{3\alpha}2}_{t,x}}\\
	= \norm{[|\d_x|^{\frac1{3\alpha}} A(\cdot)\psi] (t,x)
	\overline{[|\d_x|^{\frac1{3\alpha}}A(\cdot) \widetilde{\psi}](t-(s_n-\widetilde{s_n}),x+(y_n-\widetilde{y_n}))}}_{L^{\frac{3\alpha}2}_{t,x}}.
\end{multline*}
It is obvious from  $|\d_x|^{\frac1{3\alpha}} A(t)\psi, |\d_x|^{\frac1{3\alpha}} A(t)\widetilde{\psi} \in L^{3\alpha}_{t,x}(\R^2)$
that the left hand side tends to zero as $|s_n-\widetilde{s}_n| + |y_n-\widetilde{y}_n| \to \I$.
We finally consider the case $|\xi_n|\to\I$. Then, %by the formula,
\begin{align*}
	|\d_x|^\frac1{3\alpha} A(t) \mathcal{G}_{n} \psi 
	={}& \frac{(h_n)^{\frac{4}{3\alpha}}}{\sqrt{2\pi}} \int |\xi|^{\frac1{3\alpha}}e^{i(h_nx-y_n)\xi+i((h_n)^3 t + s_n) \xi^3} \hat\psi(\xi+\xi_n) \, d\xi \\
	={}& \frac{(h_n)^{\frac{4}{3\alpha}} 
	|\xi_n|^\frac1{3\alpha}}{\sqrt{2\pi}} e^{-i((h_n)^3 t + s_n) \xi_n^3
	-i(h_nx-y_n)\xi_n }\\
	&{}\times \int \abs{1-\frac{\xi}{\xi_n}}^{\frac1{3\alpha}} e^{i 
	\frac{t'}{3\xi_n}\xi^{3}}
	e^{-it'\xi^2 + ix'\xi}  \hat\psi(\xi) \,d\xi,  
\end{align*}
where $t'=3((h_n)^3t+s_{n})\xi_n $ and $x'=h_{n}x-y_{n}+
3((h_n)^3 t + s_n) \xi_{n}^2$.
Hence, by change of variables,
\begin{eqnarray*}
\lefteqn{\norm{ |\d_x|^{\frac1{3\alpha}} A(t) \mathcal{G}_{n} \psi  }_{L^{3\alpha}_{t,x}(\R^2)}}\\
	&=&
	3^{-\frac1{3\alpha}}\norm{  \frac{1}{\sqrt{2\pi}} \int \abs{1-\frac{\xi}{\xi_n}}^{\frac1{3\alpha}}  e^{i 
	\frac{t'}{3\xi_n}\xi^{3}}
	e^{-i t' \xi^2 + ix'\xi}  \hat\psi(\xi) \,d\xi }_{L^{3\alpha}_{t',x'}(\R^2)}.
\end{eqnarray*}
Since
\[
	\frac{1}{\sqrt{2\pi}} \int \abs{1-\frac{\xi}{\xi_n}}^{\frac1{3\alpha}} 
	 e^{i 
	\frac{t}{3\xi_n}\xi^{3}}
	e^{-i t \xi^2 + ix\xi}  \hat\psi(\xi) \,d\xi
	\to (S(\cdot)\psi)(t,x)
\]
as $n\to\I$ for any fixed $(t,x) \in \R^2$ and since
\[
	\abs{	\frac{1}{\sqrt{2\pi}} \int \abs{1-\frac{\xi}{\xi_n}}^{\frac1{3\alpha}} e^{i \frac{t\xi^3}{3\xi_n}}
	e^{-i t \xi^2 + ix\xi}  \hat\psi(\xi) \,d\xi} \le C_{\psi} (1+|t|)^{-\frac14} (1+|x|)^{-\frac14}
	\in L^{3\alpha}_{t,x}
\]
as in Shao \cite{Shao}, the dominated convergence theorem gives us
\[
	\frac{1}{\sqrt{2\pi}} \int \abs{1-\frac{\xi}{\xi_n}}^{\frac1{3\alpha}} e^{i \frac{\xi^3}{3\xi_n}}
	e^{-i t \xi^2 + ix\xi}  \hat\psi(\xi) \,d\xi
		\to (S(\cdot)\psi)(t,x)
\]
in $L^{3\alpha}_{t,x}(\R^2)$ as $n\to\I$.
Therefore, by H\"older's and Strichartz's estimates,
\begin{align*}
	&\norm{[|\d_x|^{\frac1{3\alpha}} e^{-t\d_x^3} \mathcal{G}_{n} \psi]
	\overline{[|\d_x|^{\frac1{3\alpha}} e^{-t\d_x^3} \widetilde{\mathcal{G}}_{n} \widetilde{\psi}]}}_{L^{\frac{3\alpha}2}_{t,x}} \\
	&{} \le \norm{[S(\cdot)\psi](t',x')
	\overline{[S(\cdot)\widetilde{\psi}](t'+3\xi_n (s_n-\widetilde{s}_n),x'-
	(y_n-\widetilde{y}_n -3\xi_n^2 (s_n-\widetilde{s_n}) )) }  }_{L^{\frac{3\alpha}2}_{t',x'}} \\
	&{}\quad +o(1)\\
	&{}= o(1)
\end{align*}
with  the help of orthogonality condition, which completes the proof.
\end{proof}%$\qed$

\vskip2mm

\begin{lemma}\label{lem:realorthty2}
Let $4/3<\alpha<2$ and $J\ge1$.
Let $\psi^j \in \hat{L}^{\alpha}$ ($1\le j \le J$).
Let $\{\mathcal{G}^j_n \}_n \subset G$ ($1\le j \le J$)
be mutually space-time nonresonant families.
Then,
\[
	\norm{\sum_{j=1}^J e^{-t\d_x^3} \mathcal{G}^j_{n} \psi^j}_{L(\R)}^{3\alpha}
	\le \sum_{j=1}^J \norm{e^{-t\d_x^3} \mathcal{G}^j_{n} \psi^j}_{L(\R)}^{3\alpha}
	+o(1)
\]
as $n\to\I$.
\end{lemma}

\begin{proof}
%\vskip2mm
%\noindent
%{\it Proof of Lemma \ref{lem:realorthty2}.} 
Remark that $3\alpha \in (4,6)$ is not necessarily an integer, and so we argue as in \cite{BV}. 
We consider the case $5<3\alpha<6$, the other case is simpler. 
Set $s:=3\alpha-5 \in (0,1)$ and $M:= \max_{j} \norm{\psi^j}_{\hat{L}^\alpha}$.
For simplicity, we denote $U^j_n:= |\d_x|^{\frac{1}{3\alpha}}
e^{-t\d_x^3} \mathcal{G}^j_{n} \psi^j$.
We have
\begin{align*}
	\norm{\sum_{j=1}^J U^j_n}_{L^{3\alpha}_{t,x}(\R^2)}^{3\alpha}
	&{} = \iint_{\R^2} \abs{\sum_{j=1}^J U^j_n }^2
	\abs{\sum_{k=1}^J U^k_n}^2
	\abs{\sum_{l=1}^J U^l_n}
	\abs{\sum_{m=1}^J U^m_n }^s
	\, dtdx \\
	&{}\le \sum_{j_1,j_2,k_1,k_2,l,m \in [1,J]} \iint_{\R^2} 
	|U^{j_1}_n \overline{U^{j_2}_n}|
	|U^{k_1}_n \overline{U^{k_2}_n}|
	|U^l_n| |U^m_n|^s
	\, dtdx.
\end{align*}
Hence, it suffices to show that
if $j_1=j_2=k_1=k_2=l=m$ fails then
\begin{equation}\label{eq:pf_realorthty2}
	A_n:= \iint_{\R^2} 
	|U^{j_1}_n | | U^{j_2}_n|
	|U^{k_1}_n | |U^{k_2}_n|
	|U^l_n| |U^m_n|^s
	\, dtdx=o(1)
\end{equation}
as $n\to\I$. If $j_1\neq j_2$ then \eqref{eq:ST3} and Lemma \ref{lem:realorthty} yield
\begin{align*}
	A_n &{}\le \norm{U^{j_1}_n \overline{U^{j_2}_n}}_{L^{\frac{3\alpha}2}_{t,x}}
	\norm{U^{k_1}_n}_{L^{3\alpha}_{t,x}} \norm{U^{k_2}_n}_{L^{3\alpha}_{t,x}} \norm{U^{l}_n}_{L^{3\alpha}_{t,x}} 
	\norm{U^{m}_n}_{L^{3\alpha}_{t,x}}^s \\
	&{}\le CM^{3\alpha-2} \norm{U^{j_1}_n \overline{U^{j_2}_n}}_{L^{\frac{3\alpha}2}_{t,x}} =o(1)
\end{align*}
as $n\to\I$. The same argument shows \eqref{eq:pf_realorthty2} holds if $j_1=j_2= k_1=k_2
=l$ fails. If $j_1=j_2= k_1=k_2=l\neq m$ then
\begin{align*}
	A_n\le \norm{U^{l}_n \overline{U^{m}_n}}_{L^{\frac{3\alpha}2}_{t,x}}^s
	\norm{U^{l}_n}_{L^{3\alpha}_{t,x}}^{3\alpha-2s}
	\le CM^{3\alpha-2s} \norm{U^{l}_n \overline{U^{m}_n}}_{L^{\frac{3\alpha}2}_{t,x}}^s =o(1)
\end{align*}
as $n\to\I$ as above. 
\end{proof}%$\qed$

\vskip2mm

\subsection{Proof of Theorem \ref{thm:cc}}
The proof is consists of three steps.
The argument is very close to that in the mass-critical case $\alpha=2$
such as \cite{MV,CK,Shao}.

\vskip1mm
\noindent
{\bf Step 1 -- Decomposition into a sum of scale pieces. } 
Let us begin the proof of Theorem \ref{thm:cc}
with a decomposition of bounded sequence into some pieces 
of which Fourier transforms 
have mutually disjoint compact supports and are bounded.

\begin{lemma}\label{lem:pd:step1}
Let $4/3<\alpha<2$ and $\sigma\in (\alpha',\frac{6\alpha}{3\alpha-2})$.
Suppose that a bounded sequence 
$\{ u_n\}_n \subset \hat{L}^{\alpha} $ satisfy $\norm{u_n}_{\hat{M}^{\alpha}_{2,\sigma}}\le M$.
Then, for any $\epsilon>0$,
there exist 
a subsequence of $\{u_n\}$ which denoted still by $\{u_n\}$, 
a number $J$,
$\{ I_n^j=[h_n^j \xi_n^j,h_n^j(\xi_n^j+1)] \}_{n} \subset \mathcal{D}$ ($1\le j \le J$),
$\{ f_n^j\}_{n} \subset \hat{L}^{\alpha}$ ($1 \le j \le J$), 
 and $ q_n \in \hat{L}^{\alpha}$
such that 
\[
	\abs{\log \frac{h_n^j}{h_n^k}}
	+ \abs{ \xi_n^j - \frac{{h}_n^k}{h_n^j} {\xi}_n^k }  \to \I
\]
as $n\to \I$ for $1\le j<k\le J$,
and $u_n$ is decomposed into
\begin{eqnarray}
	u_n = \sum_{j=1}^J f_n^j + q_n
	\label{ufq}
\end{eqnarray}
for all $n \ge 1$.
Moreover, it holds that
\[
	\norm{u_n}_{\hat{M}^{\alpha}_{2,\sigma}}^\sigma
 	\ge \sum_{j=1}^J \norm{f_n^j}_{\hat{M}^{\alpha}_{2,\sigma}}^\sigma
 	+ \norm{q_n}_{\hat{M}^{\alpha}_{2,\sigma}}^\sigma
\]
for all $n\ge 1$ and
\begin{eqnarray}
	\limsup_{n\to\I} 
	\norm{q_n}_{\hat{M}^\alpha_{\frac{3\alpha}2,\frac{6\alpha}{3\alpha-2}}} \le \epsilon. 
	\label{qe}
\end{eqnarray}
Further, 
there exists a bounded and compactly supported function $F_j$ such that
 $\widehat{f_n^j}$ satisfies 
\begin{equation}\label{eq:pd:step1_ptwisebdd}
	|I_n^j|^{\frac1{\alpha'}} \abs{\F f_n^j ( h_n^j (\xi +\xi_n^j)) } \le F_j (\xi)
\end{equation}
for any $n\ge1$.
\end{lemma}

\vskip2mm

\begin{remark}
In the above decomposition, not only the number $J$ but also
$f_n^j$, $q_n$, and $F^j$ depend on $\epsilon$.
\end{remark}

\begin{proof}
%\vskip2mm
%\noindent
%{\it Proof of Lemma \ref{lem:pd:step1}.} 
If $\limsup_{n\to\I} \norm{u_n}_{\hat{M}^\alpha_{{3\alpha}/2,6\alpha/(3\alpha-2)}} \le \eps$
 then 
there is nothing to prove. Otherwise, we can extract a subsequence so that
$\norm{u_n}_{\hat{M}^\alpha_{{3\alpha}/2,6\alpha/(3\alpha-2)}} > \epsilon$
for all $n$.
By means of \eqref{eq:ST3} and assumption, one sees that
\[
	\epsilon <  C 
% 	\norm{|I|^{\frac{r-6}{2r}} \norm{\F u_n}_{L^2(I)}}_{\ell^{p}_{\mathcal{D}}}
	\norm{u_n}_{\hat{M}^\alpha_{2,\sigma}}^{\theta}
% 	\left[ \sup_{I \in \mathcal{D}}  |I|^{-\frac1{r}} \norm{\hat{u_n}}_{L^{\frac{r}{r-2}}(I)} \right]
	\norm{u_n}_{\hat{M}^\alpha_{\frac{3\alpha}2,\I}}^{1-\theta} 
	\le  CM^\theta \norm{u_n}_{\hat{M}^\alpha_{\frac{3\alpha}2,\I}}^{1-\theta}.
% 	\left[ \sup_{I \in \mathcal{D}}  |I|^{-\frac1{r}} \norm{\hat{u_n}}_{L^{\frac{r}{r-2}}(I)} \right]^{1-\theta}.
\]
Hence, by definition of $\dot{M}^{\alpha}_{3\alpha/2,\I}$ norm,
there exists an interval $I_n^1 :=[h_n^1\xi_n^1,h_n^1(\xi_n^1+1)] \in \mathcal{D}$ such that
\begin{eqnarray}
	\int_{I_n^1} |\hat{u}_{n}|^{\frac{3\alpha}{3\alpha-2}} \,d\xi 
	&\ge& C(M)  \epsilon^{\frac{3\alpha}{(1-\theta)(3\alpha-2)}} |I_n^1|^{\frac1{3\alpha-2}}\label{up1}\\
	&=:& C_1  \epsilon^{\frac{3\alpha}{(1-\theta)(3\alpha-2)}} |I_n^1|^{\frac1{3\alpha-2}},
	\nonumber
\end{eqnarray}
where $C_1$ is a constant depending only on $\alpha$, $\sigma$, and $M$.
On the other hand, for any $A>0$, we have
\begin{eqnarray}
	\int_{I_n^1 \cap \{ |\hat{u}_{n}| \ge A\}} |\hat{u}_{n}|^{\frac{3\alpha}{3\alpha-2}} d\xi
	&\le& A^{-\frac{3\alpha-4}{3\alpha-2}} \norm{\hat{u}_{n}}_{L^2(I_n^1)}^{2} 
	\label{up2}\\
	&\le& A^{-\frac{3\alpha-4}{3\alpha-2}}|I_n^1|^{\frac{2-\alpha}{\alpha}}
	\norm{u_n}_{ \hat{M}^{\alpha}_{2,\sigma}}^2
% 	\norm{|I|^{\frac{p-2}{2p}} \norm{\F u_n}_{L^2(I)}}_{\ell^{p}_{\mathcal{D}}}^2
	\nonumber\\
	&\le& M^2 A^{-\frac{3\alpha-4}{3\alpha-2}}
	|I_n^1|^{\frac{2-\alpha}{\alpha}}.
	\nonumber
\end{eqnarray}
We choose $A=(\frac{2M^2}{C_1})^{\frac{3\alpha-2}{3\alpha-4}} \epsilon^{-\frac{3\alpha}{(1-\theta)(3\alpha-4)}} |I_n^1|^{-\frac{1}{\alpha'}}
=: C_\epsilon |I_n^1|^{-\frac{1}{\alpha'}}$ so that
\[
	M^2 A^{-\frac{3\alpha-4}{3\alpha-2}}|I_n^1|^{\frac{2-\alpha}{\alpha}} = \frac{C_1}2 \epsilon^{\frac{3\alpha}{(1-\theta)(3\alpha-2)}} |I_n^1|^{\frac1{3\alpha-2}} .
\]
From (\ref{up1}) and (\ref{up2}), we have
\begin{eqnarray}
	\lefteqn{\int_{I_n^1 \cap \{ |\hat{u}_{n} | \le C_\epsilon |I_n^1|^{-1/\alpha'} \}} |\hat{u}_{n}|^{\frac{3\alpha}{3\alpha-2}} d\xi }
	\label{do1}\\
	&\ge& \int_{I_n^1} |\hat{u}_{n}|^{\frac{3\alpha}{3\alpha-2}} d\xi
	-\int_{I_n^1 \cap \{ |\hat{u}_{n}| \ge A \}} |\hat{u}_{n}|^{\frac{3\alpha}{3\alpha-2}} d\xi
	\nonumber\\
	&\ge& \frac{C_1}2  \epsilon^{\frac{3\alpha}{(1-\theta)(3\alpha-2)}} |I_n^1|^{\frac1{3\alpha-2}}.
	\nonumber
\end{eqnarray}
H\"older's inequality implies that
\begin{eqnarray}
	\lefteqn{\int_{I_n^1 \cap \{ |\hat{u}_{n} | \le C_\epsilon |I_n^1|^{-1/\alpha'} \} } |\hat{u}_{n}|^{\frac{3\alpha}{3\alpha-2}} d\xi}
	\label{do2}\\
	&\le& \(\int_{I_n^1 \cap \{ |\hat{u}_{n} | \le C_\epsilon |I_n^1|^{-1/\alpha'} \} } |\hat{u}_{n}|^2 d\xi \)^{\frac{3\alpha}{2(3\alpha-2)}} |I_n^1|^{\frac{3\alpha-4}{2(3\alpha-2)}}.
	\nonumber
	%\\
%	&{}= \( |I_n^1|^{2(\frac12-\frac1\alpha)} \int_{I_n^1 \cap \{ |\hat{u_n} | \le C_\epsilon |I_n^1|^{-1/\alpha'} \} } |\hat{u_n}|^2 d\xi \)^{\frac{3\alpha}{2(3\alpha-2)}}  |I_n^1|^{\frac{1}{3\alpha-2}}
\end{eqnarray}
Combining the inequalities (\ref{do1}) and (\ref{do2}), 
we reach to the estimate
\begin{eqnarray}
	\ \ \ \ |I_n^1|^{\frac12-\frac1\alpha} \(\int_{I_n^1 \cap \{ |\hat{u}_{n} | \le C_\epsilon |I_n^1|^{-1/\alpha'} \} } |\hat{u}_{n}|^2 d\xi\)^{\frac12}
	\ge \(\frac{C_1}2\)^{\frac{3\alpha-2}{3\alpha}}  \epsilon^{\frac{1}{1-\theta}}.
	\label{d3}
\end{eqnarray}
We define $v_n^1$ by $\widehat{v_n^1}:= \hat{u}_{n} {\bf 1}_{I_n^1\cap \{  |
\hat{u}_{n} | \le C_\epsilon |I_n^1|^{-1/\alpha'}\}}$ and $q_n^1:=u_n-v_n^1$.
Then, (\ref{d3}) can be rewritten as 
$\norm{v_n^1}_{\hat{M}^{\alpha}_{2,\sigma}}
	\ge C' \epsilon^{\frac{1}{1-\theta}}$.
Further, we have 
\[
	|I_n^1|^{\frac{1}{\alpha'}} \abs{ \widehat{v_n^1} (h_n^1 (\xi + \xi_n^1))} \le C_\epsilon
	{\bf 1}_{[0,1]}(\xi).
\]

% If $$\limsup_{n\to\I}\norm{|I|^{\frac{r-6}{2r}} \norm{\F q_n^1}_{L^2(I)}}_{\ell^{\frac{2r}{r-2}}_{\mathcal{D}}} \le \epsilon,$$
If $\limsup_{n\to\I}
\norm{q_n^1}_{\hat{M}^{\alpha}_{3\alpha/2,6\alpha/(3\alpha-2)}}
% \norm{|I|^{-\frac{1}{r}} \norm{\F q_n^1}_{L^{\frac{r}{r-2}}(I)}}_{\ell^{\frac{2r}{r-2}}_{\mathcal{D}}} 
\le \epsilon$
then we have done.
Otherwise, the same argument with $u_n$ being replaced by $q_n^1$
enables us to define $I_n^2:=[h_n^2\xi_n^2,h_n^2(\xi_n^2+1)]$,
$v_n^2$, and $q_n^2$ (up to subsequence).
We repeat this argument and define $I_n^j:=[h_n^j\xi_n^j,h_n^j(\xi_n^j+1)]$, $v_n^j$, and $q_n^j$ inductively.
It is easy to see that
\[
	\norm{u_n}_{\hat{M}^{\alpha}_{2,\sigma}}^\sigma \ge
	\sum_{j=1}^N \norm{v_n^j}_{\hat{M}^{\alpha}_{2,\sigma}}^\sigma
	+ \norm{q_n^N}_{\hat{M}^{\alpha}_{2,\sigma}}^\sigma
\]
since supports of $\{v_n^j\}_{1\le j \le N}$ and $q_n^N$  are disjoint in the Fourier side and since $\sigma>2$.
Since $\|v_n^j\|_{\hat{M}^{\alpha}_{2,\sigma}}\ge C' \epsilon^{\frac{1}{1-\theta}}$
for each $j$, together with an embedding
$\|q_n^j\|_{\hat{M}^{\alpha}_{3\alpha/2,\sigma}}\le \|q_n^j\|_{\hat{M}^{\alpha}_{2,\sigma}}$,
we see that
\[
	\limsup_{n\to\I} \norm{q_n^J}_{\hat{M}^{\alpha}_{\frac{3\alpha}2,\frac{6\alpha}{3\alpha-2}}}
	\le \epsilon
\]
holds in $J=J(\epsilon)$ steps. Set $q_n:=q_n^J$.

We reorganize $v_n^j$ to obtain
mutual asymptotic orthogonality.
It is done as follows;
We collect all $k\ge2$ such that $|\log \frac{h_n^1}{h_n^k}| + \abs{\xi_n^1-\frac{h_n^k}{h_n^1}\xi_n^k}$ is bounded, and define $f_n^1:=v_n^1+\sum_{k} v_n^k$.
Since 
\begin{align*}
	&|I_n^1|^{\frac{1}{\alpha'}} \abs{ \widehat{v_n^k} (h_n^1 (\xi + \xi_n^1))} \\
	&{} =  \(\frac{h_n^1}{h_n^k}\)^{\frac{1}{\alpha'}} |I_n^k|^{\frac{1}{\alpha'}} \abs{ \widehat{v_n^k} \(  h_n^k \left[ \frac{h_n^1}{h_n^k} \left\{ \xi  + \( \xi_n^1 - \frac{h_n^k}{h_n^1} \xi_n^k \) \right\}   + \xi_n^k \right] \)} \\
	&{} \le  C_\eps \(\frac{h_n^1}{h_n^k}\)^{\frac{1}{\alpha'}} {\bf 1}_{[0,1]} \( \frac{h_n^1}{h_n^k} \left\{ \xi  + \( \xi_n^1 - \frac{h_n^k}{h_n^1} \xi_n^k \) \right\} \) ,
\end{align*}
we see that $|I_n^1|^{\frac{1}{\alpha'}} \F{f_n^1} (h_n^1 (\xi + \xi_n^1)) \le F_1(\xi)$
for some bounded and compactly supported function $F_1$.
Similarly, we define $f_n^j$ inductively.
It is easy to see that $f_n^j$ possesses all properties we want. 
This completes the proof of Lemma \ref{lem:pd:step1}. 
\end{proof}%$\qed$
%%%%%%%%%%%%%%%%%%%%%%%%%%%%%%%%%%%%%%%%%%%%%%%%%%%%%%%%%%%

\vskip2mm
\noindent
{\bf Step 2 -- Decomposition of each scale pieces.} 
We next decompose functions obtained in the previous decomposition.
This part is similar to \cite[Lemmas 3.4 and 3.5]{Shao}.

\begin{lemma}\label{lem:pd:step2}
Let $4/3<\alpha<2$ and $\sigma \in (\alpha',\frac{6\alpha}{3\alpha-2})$.
Let $\xi_n \in \R$ be a given sequence.
Let $F(\xi)$ be a nonnegative bounded function with compact support.
Suppose that a sequence $R_n \in \hat{L}^{\alpha}$ satisfy 
\begin{equation}\label{eq:pd:step2_support_assumption}
	|\widehat{R_n} (\xi) | \le F (\xi).
\end{equation}
Then, up to subsequence, there exist $\{\phi^a\}_a \subset \hat{L}^{\alpha} $ with $ |\widehat{\phi^a} (\xi)| \le F(\xi) $, 
$(y_n^a,s_n^a) \in \R^2$ with
\[
	\lim_{n\to\I} (|s_n^a - s_n^{\tilde{a}} |+ |\xi_n(s_n^a-s_n^{\tilde{a}})|+
	|y_n^a -y_n^{\tilde{a}} + 3(\xi_n)^2(s_n^a-s_n^{\tilde{a}})|) = \I
\]
for any $\tilde{a}\neq a$, and $\{R_n^a\}_{n,a}\subset \hat{L}^{\alpha}$ with $|\widehat{R_n^a} (\xi)| \le F(\xi)$ such that
\[
	R_n(x) = \sum_{a=1}^A  [P(\xi_n)^{-1} A(s_n^a) T({y_n^a}) P(\xi_n) \phi^a](x) + R_n^A(x)
\]
for any $A\ge1$. Moreover, it holds that
\begin{equation}\label{eq:pd:step2_Pythagorean}
	\sum_{a=1}^A 
	\norm{\psi^a}_{\hat{M}^{q}_{2,r}}^r
	+\limsup_{n\to\I} 
	\norm{R_n^A}_{\hat{M}^{q}_{2,r}}^r
	\le \limsup_{n\to\I} 
	\norm{R_n}_{\hat{M}^{q}_{2,r}}^r
	<\I
\end{equation}
for any $2<q'<r<\I$ and $A\ge1$. 
Furthermore, we have
\begin{equation}\label{eq:pd:step2_smallness}
	\limsup_{n\to\I} \norm{ e^{-t\d_x^3} [P({\xi_n}) R_n^A]}_{L(\R)}
	\to 0
\end{equation}
as $A\to\I$.
\end{lemma}

\vskip2mm

\begin{remark}
This lemma itself is a profile decomposition type result.
% The strategy of the proof is the same as that of Theorem \ref{thm:pd}.
The main differences between Lemma \ref{lem:pd:step2} and 
Theorem \ref{thm:pd}
are that (i) a strong assumption \eqref{eq:pd:step2_support_assumption} is made;
(ii) a family of deformations $\{\mathcal{G}_n\}_n \subset G$ is replaced by a family of deformations
of the form
$\{P(\xi_n)^{-1} A(s_n) T(y_n) P(\xi_n)\}_n$ with a \emph{fixed} sequence 
$\{\xi_n \}_n \subset \R$ and
some $\{(s_n,y_n)\} \subset \R^2$;
and (iii) a remainder is small as in \eqref{eq:pd:step2_smallness}.
\end{remark}

\begin{proof}
%\vskip2mm
%\noindent
%{\it Proof of Lemma \ref{lem:pd:step2}.} 
For a sequence $R=\{R_n\}_n \subset \hat{L}^{\alpha}$ with \eqref{eq:pd:step2_support_assumption},
we introduce $\mu_\xi(R)$  as follows:
\[
	\mu_\xi(R):= \sup \left\{ 
	\ell(\phi)
% 	\norm{\phi}_{\hat{M}^{p}_{2,\sigma}}
	\ \left|\ \phi \in {{\mathcal M}}_\xi(R) \right.\right\},
\]
where
\[
	\mathcal{M}_\xi(R):= \left\{ \phi \in \hat{L}^{\alpha}\ \left| \ 
	\begin{aligned}
	&\phi=\text{w-}\lim_{k\to\I} P(\xi_{n_k})^{-1}T(y_{n_k})^{-1}A(s_{n_k})^{-1} P(\xi_{n_k})R_{n_k},\, \\
	&\exists (s_n,y_n) \in \R^2,\, \exists n_k\text{ : subsequence}
	\end{aligned}
	\right.\right\}.
\]
Then, we first show the decomposition
with the smallness $\mu_\xi (R^A)\to0$ as $A\to\I$ instead of \eqref{eq:pd:step2_smallness}.
This part is done as in Theorem \ref{thm:pd1}.
Before the proof, remark that $P^{-1} (\xi_n) A(s_n)T(y_n)P(\xi_n)$ is a multiplier-like 
deformation.
Indeed, 
$\F {P}(\xi_n)^{-1}{A}(s_n){T}({y_n}){P}(\xi_n)\F^{-1}=
e^{-iy_n(\xi -\xi_n) + i s_n(\xi-\xi_n)^3 }$ is phase-like.
Therefore, if a sequence $R$ satisfies \eqref{eq:pd:step2_support_assumption} then
any $\phi \in {{\mathcal M}}_\xi (R)$ satisfies \eqref{eq:pd:step2_support_assumption}.

Since the decomposition is shown in the 
essentially same way as in Theorem \ref{thm:pd1},
we only treat the first step the decomposition, that is, the extraction of $\phi^1$ and $(s_n^1,y_n^1)$
under the assumption $\mu_\xi (R)>0$.
By assumption $\mu_\xi (R)>0$ there exist $\phi^1\in \mathcal{M}_\xi (R)$ such that
$\ell(\phi^1) \ge \frac12 \mu_\xi (R)$.
Further, by definition of $\mathcal{M}_\xi (R)$, there exists $\{(s_n^1,y_n^1)\}_n$ such that
\[
	P(\xi_n)^{-1} T({y_{n}^1})^{-1}A(s_{n}^1)^{-1} P(\xi_n) R_{n} \rightharpoonup \phi^1(x) \wIN \hat{L}^{\alpha}
\]
as $n\to\I$ up to subsequence. Then,
\[
	\hat{P}(\xi_n)^{-1} \hat{T}({y_n^1})^{-1} \hat{A}(s_n^1)^{-1} \hat{P}(\xi_n) \widehat{R_n}(\xi) \rightharpoonup \widehat{\phi^1}(\xi)
	\wIN {L}^{\alpha'}.
\]
As mentioned above, $\phi^1 $ satisfies \eqref{eq:pd:step2_support_assumption}.
% because
% \[
% 	|\hat{P}(\xi_n)^{-1} \hat{T}({y_n^1})^{-1} \hat{A}(s_n^1)^{-1} \hat{P}(\xi_n) \widehat{R_n}(\xi)|
% 	=|\widehat{R_n}(\xi)|.
% \]
We set 
$$R^1_n:= R_n - P(\xi_n)^{-1}A(s_n^1)T({y^1_n})P(\xi_n)\phi^1.$$
Then, 
\[
	P(\xi_n)^{-1} T({y_{n}^1})^{-1}A(s_{n}^1)^{-1} P(\xi_n) R_{n}^1 \rightharpoonup 0 \wIN \hat{L}^{\alpha}
\]
as $n\to\I$. 
Since $R_n$ and $\psi^1$ satisfy \eqref{eq:pd:step2_support_assumption},
we have $|\F R_n^1(\xi)| \le 2 F(\xi)$
for any $n$. 
The decouple inequality \eqref{eq:pd:step2_Pythagorean} is shown 
just as in Lemma \ref{lem:decouple}\footnote{In fact, the proof is even easier than 
Lemma \ref{lem:decouple} because the considering deformation $P^{-1} (\xi_n) A(s_n)T(y_n)P(\xi_n)$ is multiplier-like.}.
We note that 
$F \in L^{q'} \hookrightarrow M^{q'}_{2,r}$ for $2<q'<r< \I$
since $F$ is bounded and compactly supported,
and so that $\limsup_{n\to\I} \norm{R_n}_{\hat{M}^{q}_{2,r}} \le \norm{F}_{M^{q'}_{2,r}}$ is finite
by means of the assumption \eqref{eq:pd:step2_support_assumption}.
Repeat the above procedure to obtain all the results but \eqref{eq:pd:step2_smallness}.
% Obviously, $|\F R_n^J(\xi)| \le (1+J) F(\xi)$ for any $n,J \ge 1$. 

Let us show \eqref{eq:pd:step2_smallness}.
By extracting subsequence, we may suppose that either $ \xi_n \to \xi_0 \in \R$
or $|\xi_n| \to \I $ as $n\to\I$.
We first consider the case $\lim_{n\to\I} \xi_n =\xi_0$.
Define $\widetilde{\alpha}\in (4/3,\alpha)$ by $2/{\widetilde{\alpha}}=1/\alpha+3/4$. 
By H\"older's inequality,
\begin{eqnarray}
\lefteqn{{}\ \ \norm{e^{-t\d_x^3}\left[ P(\xi_n) R_n^A\right]}_{L(\R)}}
\label{const1}\\ 
	&\le& \norm{ |\d_x|^{\frac1{3\alpha}} 
	e^{-t\d_x^3} \left[P(\xi_n) R_n^A\right]}_{L^{3\widetilde{\alpha}}_{t,x}}^{\frac{\widetilde{\alpha}}{\alpha}}
	\norm{ |\d_x|^{\frac1{3\alpha}} e^{-t\d_x^3} 
	\left[P(\xi_n) P_n^A\right]}_{L^\I_{t,x}}^{1-\frac{\widetilde{\alpha}}{\alpha}} .
	\nonumber
\end{eqnarray}
Since $4/3<\widetilde{\alpha}<\alpha<2$, we have 
\begin{eqnarray*}
\lefteqn{\norm{ |\d_x|^{\frac1{3\alpha}} e^{-t\d_x^3}\left[ P(\xi_n) R_n^A \right]}_{L^{3\widetilde{\alpha}}_{t,x}}}\\
	&=& \norm{ |\d_x|^{\frac1{3\widetilde{\alpha}}} e^{-t\d_x^3}  |\d_x|^{\frac1{3\alpha}-\frac1{3\widetilde{\alpha}}} \left[P(\xi_n) R_n^A\right]}_{L^{3\widetilde{\alpha}}_{t,x}} \\
	&\le& C \norm{ |I|^{-\frac{1}{3\widetilde{\alpha}}} \norm{|\xi|^{\frac1{3\alpha}-\frac1{3\widetilde{\alpha}}} \F P(\xi_n) R_n^A}_{L^{(\frac{3\widetilde{\alpha}}{2})'}(I)} }_{\ell^{\frac{6\widetilde{\alpha}}{3\widetilde{\alpha}-2}}_{\mathcal{D}}} \\
	&\le& C \norm{ |I|^{-\frac{1}{3\widetilde{\alpha}}} 
	\norm{\F P(\xi_n) R_n^A}_{L^{2}(I)}  \norm{|\xi|^{\frac1{3\alpha}-\frac1{3\widetilde{\alpha}}}}_{L^{\frac{6\widetilde{\alpha}}{3\widetilde{\alpha}-4}}(I)} }_{\ell^{\frac{6\widetilde{\alpha}}{3\widetilde{\alpha}-2}}_{\mathcal{D}}} \\
	&\le& C \norm{ |I|^{\frac12 -(\frac2{3\widetilde{\alpha}}+\frac14)} 
	\norm{\F P(\xi_n) R_n^A}_{L^{2}(I)} }_{\ell^{\frac{6\widetilde{\alpha}}{3\widetilde{\alpha}-2}}_{\mathcal{D}}} \\
&\le& C 
	\norm{ P(\xi_n) R_n^A }_{\hat{M}^{\frac{12\widetilde{\alpha}}{3\widetilde{\alpha}+8}}_{2,\frac{6\widetilde{\alpha}}{3\widetilde{\alpha}-2}}}.
\end{eqnarray*}
Thus, thanks to \eqref{eq:ST3} and
\eqref{eq:pd:step2_Pythagorean} with $q={\frac{12\widetilde{\alpha}}{3\widetilde{\alpha}+8}}$
and $r=\frac{6\widetilde{\alpha}}{3\widetilde{\alpha}-2}$, we see that
\begin{eqnarray}
	\lefteqn{\sup_{A\ge1}\limsup_{n\to\I} \norm{ |D|^{\frac1{3\alpha}} e^{-t\d_x^3} \left[P(\xi_n) R_n^A\right] }_{L^{3\widetilde{\alpha}}_{t,x}}}
	\label{const2} \\
	&\le& C \limsup_{n\to\I}
	\norm{ R_n^{A} }_{\hat{M}^{q}_{2,r}}
	\le C\norm{ F }_{L^{q'}}
% 	\norm{ |I|^{\frac{\widetilde{r}-6}{2\widetilde{r}}} \norm{\F P_n}_{L^2(I)} }_{\ell^{\frac{2\widetilde{r}}{\widetilde{r}-2}}_{\mathcal{D}}} 
	<\I.\nonumber
\end{eqnarray}

Let $\chi(x)$  be a smooth function such that
$\widehat{\chi}$ is compactly supported and $\widehat{\chi}\equiv1$ on $\supp F$.
Set $\chi_{n}(x) = P(\xi_n)\chi(x)$.
Then, $\chi_n *_{x} (e^{-t\d_x^3}  [P(\xi_n) R_n^A])=e^{-t\d_x^3}  [P(\xi_n) R_n^A]$.
% We estimate
% \[
% 	\norm{ |\d_x|^{\frac1{3p}} e^{-t\d_x^3}  \pi_{\xi_n} R_n^A}_{L^\I_{t,x}}
% 	=\norm{ \chi_n *_{x} |\d_x|^{\frac1{3p}} e^{-t\d_x^3} \pi_{\xi_n}  R_n^A}_{L^\I_{t,x}}.
% \]
There exists $(s_n,y_n)$ such that
\begin{eqnarray}
	\lefteqn{\norm{ |\d_x|^{\frac1{3\alpha}} e^{-t\d_x^3} 
	\left[P(\xi_n) R_n^A\right]}_{L^\I_{t,x}}}
	\label{const3}\\
	&\le& 2 |( \chi_n*_{x} |\d_x|^{\frac1{3\alpha}} e^{-t\d_x^3} 
	\left[P(\xi_n) R_n^A\right] )(-s_n,y_n)|.
	\nonumber
	\end{eqnarray}
A computation shows that 
\begin{align*}
	&(\chi_n *_{x}|\d_x|^{\frac1{3\alpha}} e^{-t\d_x^3} 
	\left[P(\xi_n) R_n^A\right])(-s_n,y_n) \\
	&{}=\int \chi_n (-x) |\d_x|^{\frac1{3\alpha}}  [T(y_n)^{-1}A(s_n)^{-1}P(\xi_n) R_n^A](x) dx \\
	&{}=\int \( P(\xi_n)|\d_x|^{\frac1{3\alpha}} [P(\xi_n)^{-1}
	\chi(-\cdot)] \)(x)
	[P(\xi_n)^{-1} T(y_n)^{-1}A(s_n)^{-1}P(\xi_n) R_n^A](x) dx.
\end{align*}
Since $\xi_n\to \xi_0\in \R$ as $n\to\I$,
one verifies that $P(\xi_n)|\d_x|^{\frac1{3\alpha}} [P(\xi_n)^{-1} \chi(-\cdot)]$
converges to some function strongly in $\hat{L}^{\alpha'}$ as $n\to\I$.
Hence, by definition of $\mu_\xi$, we see that
\[
	\limsup_{n\to\I}| (\chi_n*|\d_x|^{\frac1{3\alpha}}e^{-t\d_x^3} P(\xi_n) R_n^A)(-s_n,y_n)|
	\le C_{\chi} \mu_\xi(R^A).
\]
Combining the above inequality with (\ref{const1}), (\ref{const2}) and (\ref{const3}), we obtain \eqref{eq:pd:step2_smallness} from $\mu_\xi(R^A)\to0$
as $A\to\I$.

If $\lim_{n\to\I} |\xi_n|=\I$ then it holds from H\"ormander-Mikhlin multiplier theorem that
\begin{eqnarray}
	\norm{ e^{-t\d_x^3} [P(\xi_n) R_n^A] }_{L(\R)}
	\le C|\xi_n|^{\frac1{3\alpha}} \norm{ e^{-t\d_x^3}  [P(\xi_n) R_n^A] }_{L^{3\alpha}_{t,x}(\R^2)}
	\label{high1}
\end{eqnarray}
for large $n$.
Let $\widetilde{\alpha} \in (4/3,\alpha)$ be the same number as in the previous case. 
Then,
\begin{eqnarray}
\lefteqn{|\xi_n|^{\frac1{3\alpha}}\norm{e^{-t\d_x^3}  [P({\xi_n}) R_n^A]}_{L^{3\alpha}_{t,x}}}\label{high2} \\
	&\le&
	\(|\xi_n|^{\frac1{3\widetilde{\alpha}}}\norm{ e^{-t\d_x^3}  [P(\xi_n) R_n^A]}_{L^{3\widetilde{\alpha}}_{t,x}}\)^{\frac{\widetilde{\alpha}}{\alpha}}
	\norm{e^{-t\d_x^3}  [P(\xi_n) R_n^A]}_{L^\I_{t,x}}^{1-\frac{\widetilde{\alpha}}{\alpha}}.
	\nonumber
\end{eqnarray}
Since support of $\F R_n^A$ is a subset of  $\supp F$, arguing as in Lemma \ref{lem:scale}, we have
\begin{align*}
	\norm{ e^{-t\d_x^3} P(\xi_n) R_n^A }_{L^{3\widetilde{\alpha}}_{t,x}}
	&{}= \norm{ |\d_x|^{\frac1{3\widetilde{\alpha}}} e^{-t\d_x^3}  |\d_x|^{-\frac1{3\widetilde{\alpha}}} P(\xi_n) R_n^A}_{L^{3\widetilde{\alpha}}_{t,x}} \\
	&{}\le C \norm{ |I|^{\frac12-\frac1{\widetilde{\alpha}}} \norm{|\xi-\xi_n|^{-\frac1{3\widetilde{\alpha}}} \F R_n^A}_{L^2(I)} }_{\ell^{\frac{6\widetilde{\alpha}}{3\widetilde{\alpha}-2}}_{\mathcal{D}}} \\
	&{}\le C_{F}  |\xi_n|^{-\frac1{3\widetilde{\alpha}}} \norm{ |I|^{\frac{\widetilde{\alpha}-2}{2\widetilde{\alpha}}} \norm{\F R_n^A}_{L^2(I)} }_{\ell^{\frac{6\widetilde{\alpha}}{3\widetilde{\alpha}-2}}_{\mathcal{D}}}
\end{align*}
for large $n$.
Thus, it follows from \eqref{eq:pd:step2_Pythagorean} that
\begin{eqnarray}
\lefteqn{\sup_{A\ge1} \limsup_{n\to\I}|\xi_n|^{\frac1{3\widetilde{\alpha}}}\norm{ e^{-t\d_x^3} P({\xi_n}) R_n^A }_{L^{3\widetilde{\alpha}}_{t,x}}}
\label{high3}\\
	&\le& C \limsup_{n\to\I}
% 	|\xi_n|^{\frac1{3\widetilde{p}}}
% 	\norm{ |I|^{\frac{\widetilde{p}-2}{2\widetilde{p}}} \norm{\F R_n}_{L^2(I)} }_{\ell^{\frac{6\widetilde{p}}{3\widetilde{p}-2}}_{\mathcal{D}}} 
	\norm{R_n}_{\hat{M}^{\widetilde{\alpha}}_{2,\frac{6\widetilde{\alpha}}{3\widetilde{\alpha}-2}}}
<\I.\nonumber
\end{eqnarray}
We estimate $L^\I_{t,x}$-norm.
There exists $(s_n,y_n)$ such that
\[
	\norm{ e^{-t\d_x^3}  [P(\xi_n) P_n^A]}_{L^\I_{t,x}}
	\le 2 | e^{-t\d_x^3}  [P(\xi_n) P_n^A](s_n,y_n)|.
\]
The estimate for 
\begin{eqnarray}
\limsup_{n\to\I}\norm{e^{-t\d_x^3}  [P(\xi_n) R_n^A]}_{L^\I_{t,x}}\to0
\label{high4}
\end{eqnarray}
as $A\to\infty$ is essentially the same as in the previous case.
The difference is that we do not have to care about unboundedness of $\xi_n$
because the derivative $|\d_x|^{1/3\alpha}$ is removed. 
From (\ref{high1}), (\ref{high2}), (\ref{high3}) and (\ref{high4}), 
we have \eqref{eq:pd:step2_smallness}. This completes the 
proof of Lemma \ref{lem:pd:step2}.
\end{proof}%$\qed$

%%%%%%%%%%%%%%%%%%%%%%%%%%%%%%%%%%%%%%%%%%%%%%%%%%%%%%%%%%%
\vskip2mm
\noindent
{\bf Step 3 --Completion of the proof of Theorem \ref{thm:cc}.} 
Let $\{u_{n}\}\subset\hat{L}^{\alpha}$ be bounded 
sequence satisfying (\ref{upM}) and (\ref{belowm}). 
Let $\eps=\eps(m,M)>0$ to be chosen later.
Let $J=J(\eps)\ge1$,
$\{ I_n^j=[h_n^j \xi_n^j,h_n^j(\xi_n^j+1)] \}_n \subset \mathcal{D}$ ($1\le j\le J$),
 $\{ f_n^j \}_n \subset \hat{L}^{\alpha}$ ($1\le j\le J$), and $q_n$ be sequences 
given in Lemma \ref{lem:pd:step1}. 
Set
\[
	\widehat{R_n^j}(\xi):=|h_n^j|^{\frac{1}{\alpha'}} \widehat{f_n^j} (h_n^j (\xi + \xi_n^j)).
\]
Namely, $R_n^j = P({\xi_n^j})^{-1} {D} (h_n^j)^{-1}  {f}_n^j$.
Then, 
by means of \eqref{eq:pd:step1_ptwisebdd}, $\{R_n^j\}_n$ satisfies assumption of
Lemma \ref{lem:pd:step2} with $\{ \xi_n \}_n :=\{ \xi^j_n \}_n$ for each $j$.
Then, thanks to Lemma \ref{lem:pd:step2}, 
for every $1\le j \le J$, there exists a family $\{ \phi^{j,a} \}_{a} \subset
\mathcal{M}_{\xi^j}(R^j)$,
and a family $\{ (y^{j,a}_n,s^{j,a}_n) \}_{n,a} \in \R \times \R$
such that
\[
	R_n^j = \sum_{a=1}^A P(\xi_n^j)^{-1} A(s_n^{j,a}) T(y_n^{j,a}) P(\xi_n^j) \phi^{j,a} +  R_n^{j,A}
\]
with
\[
	\limsup_{n\to\I} \norm{e^{-t\d_x^3} [P(\xi_n^j) R_n^{j,A} ]}_{L(\R)}
	\to 0
\]
as $A\to\I$
and that 
\[
	\lim_{n\to\I} (|s_n^{j,a} - s_n^{j,\tilde{a}} |+ |\xi_n^j(s_n^{j,a} -s_n^{j,\tilde{a}})|
	+|y_n^{j,a} -y_n^{j,\tilde{a}} + 3(\xi_n^j)^2(s_n^{j,a}-s_n^{j,\tilde{a}})|) = \I
\]
for any $a\neq\tilde{a}$.
Remark that 
\begin{eqnarray}
	f_n^j &=& D(h_n^j) P(\xi_n^{j})  R_n^j
	\label{fff}\\
	&=& \sum_{a=1}^A  D(h_n^j) A(s_n^{j,a}) T(y_n^{j,a}) P(\xi_n^j) \phi^{j,a} +  D(h_n^j) P(\xi_n^{j}) R_n^{j,A}.
	\nonumber
\end{eqnarray}
We choose $A=A(\eps)$ so that 
\begin{eqnarray}
	\limsup_{n\to\I} \norm{e^{-t\d_x^3} D(h_n^j) P(\xi^j_n) R_n^{j,A}}_{L(\R)}
	\le \frac{\epsilon}{J}
	\label{ej}
\end{eqnarray}
holds for any $1 \le j \le J$.
Notice that this is possible by means of the scale invariance
\[
	\norm{e^{-t\d_x^3} D(h_n^j)P(\xi^j_n) R_n^{j,A}}_{L(\R)}
	=\norm{ e^{-t\d_x^3} P(\xi^j_n) R_n^{j,A}}_{L(\R)}.
\]
Let $r_n:=\sum_{j=1}^J  D(h_n^j) P(\xi_n^{j}) R_n^{j,A}+ q_n$.
By Lemma \ref{lem:pd:step1} (\ref{ufq}) and (\ref{fff}), we have
\begin{eqnarray}
	u_n= \sum_{j=1}^J f_n^j + q_n = \sum_{j=1}^J \sum_{a=1}^{A}  \mathcal{G}_n^{j,a} \phi^{j,a} + r_n,
	\label{ufqr}
\end{eqnarray}
where $\mathcal{G}_n^{j,a}:=D(h_n^j) A(s_n^{j,a}) T(y_n^{j,a}) P(\xi_n^j)$.
It is easy to see that $\{\mathcal{G}_n^{j,a}\}_n \subset G$ 
% defined by the parameters $\Gamma_{n}^{j,a} := (h_n^j, \xi_n^j, s_n^{j,a}, y_n^{j,a})$
are pairwise orthogonal families.
Recall that, for each $(j,a)$, the number of 
the pair $(\widetilde{j},\widetilde{a})$ such that
$\{\mathcal{G}_{n}^{j,a}\}_n$ and $\{\mathcal{G}_{n}^{\widetilde{j},\widetilde{a}}\}_n$ are not
space-time nonresonant is at most one (see, Remark \ref{rem:realorthty}).
Let $\mathcal{A} = \mathcal{A}(\eps):= \{(j,a)\ |\ 1\le j \le J,\, 1\le a \le A \}$.
We divide $\mathcal{A}$ into disjoint subsets $\mathcal{A}_k$ ($1\le k \le K$) so that
\begin{enumerate}
\item $1\le \# \mathcal{A}_k \le 2$ is non-decreasing in $k$;
\item if $(j,a) \in \mathcal{A}_k$ and $(\widetilde{j},\widetilde{a}) \in \mathcal{A}_{\widetilde{k}}$, $k\neq \widetilde{k}$,
then $\{\mathcal{G}_{n}^{j,a}\}_n$ and $\{\mathcal{G}_{n}^{\widetilde{j},\widetilde{a}}\}_n$ are space-time nonresonant;
\item if $\# A_k =2$ and if
$(j,a), (\widetilde{j},\widetilde{a}) \in \mathcal{A}_k$,
$(j,a) \neq (\widetilde{j},\widetilde{a})$, then
$\{\mathcal{G}_{n}^{j,a}\}_n$ and $\{\mathcal{G}_{n}^{\widetilde{j},\widetilde{a}}\}_n$ are not space-time
nonresonant.
\end{enumerate}
Let $K'$ be the number such that $\max \{k \ |\ \# \mathcal{A}_k = 1\}$.
For $1 \le k \le K'$, we identify $k$ and the unique pair $(j,a) \in \mathcal{A}_k$.
Then, we have
\[
	u_n=  \sum_{k=1}^{K'} \mathcal{G}^k_n \phi^{k} + \sum_{k=K'+1}^K \sum_{(j,a)\in \mathcal{A}_k}  \mathcal{G}^{j,a}_{n} \phi^{j,a} + r_n.
\]
By definition of $r_n$, (\ref{ej}) and Lemma \ref{lem:pd:step1} 
(\ref{qe}), we have
\begin{align*}
	\norm{ e^{-t\d_x^3} u_n}_{L(\R)}
	&{}\le \norm{ e^{-t\d_x^3} (u_n - r_n)}_{L(\R)}
	+\norm{ e^{-t\d_x^3} r_n}_{L(\R)} \\
	&{}\le \norm{e^{-t\d_x^3} (u_n - r_n)}_{L(\R)}
	+C\eps.
\end{align*}
Combining the above inequality with the argument 
used in the proof of Lemma \ref{lem:realorthty2}, 
one can verify that
\begin{align*}
	\norm{ e^{-t\d_x^3} u_n}_{L(\R)}^{3\alpha}
	\le{}&  \sum_{k=1}^{K'} \norm{e^{-t\d_x^3} \mathcal{G}^{k}_{n} \phi^{k}}_{L(\R)}^{3\alpha}
	+  \sum_{k=K'+1}^{K} \norm{ \sum_{(j,a)\in \mathcal{A}_k} e^{-t\d_x^3} \mathcal{G}^{j,a}_{n} \phi^{j,a}}_{L(\R)}^{3\alpha}  \\
	&{} +  C(1+M^{3\alpha-1})\epsilon +o(1)
\end{align*}
as $n\to\I$. Further,  by the triangle inequality,
\[
	\norm{ \sum_{(j,a)\in \mathcal{A}_k} e^{-t\d_x^3} \mathcal{G}^{j,a}_{n} \phi^{j,a}}_{L(\R)}^{3\alpha} 
	\le 2^{3\alpha} \sum_{(j,a)\in \mathcal{A}_k}  \norm{ e^{-t\d_x^3} \mathcal{G}^{j,a}_{n} \phi^{j,a}}_{L(\R)}^{s\alpha}.
\]
Combining the above estimates and going back to the notation $(j,a)$, one has
\[
	\norm{ e^{-t\d_x^3} u_n}_{L(\R)}^{3\alpha}
	\le  2^{3\alpha}\sum_{(j,a)\in \mathcal{A}} 
	\norm{ e^{-t\d_x^3} \mathcal{G}^{j,a}_{n} \phi^{j,a} }_{L(\R)}^{3\alpha}  +  C(1+M^{3\alpha-1})\epsilon +o(1).
\]
By (\ref{belowm}), we can take $\epsilon=\epsilon(m,M)$ small
and $n$ large enough to get
\[
	C_\alpha m^{3\alpha} \le \sum_{(j,a)\in \mathcal{A}} 
	\norm{ e^{-t\d_x^3} \mathcal{G}^{j,a}_{n} \phi^{j,a} }_{L(\R)}^{3\alpha} .
\]
By the refined Stein-Tomas inequality (Theorem \ref{thm:rST} \eqref{eq:ST3})  and Lemma \ref{lem:scale},
\[
	\norm{ e^{-t\d_x^3} \mathcal{G}^{j,a}_{n} \phi^{j,a} }_{L(\R)} 
	\le C 
	\norm{ \mathcal{G}^{j,a}_{n}\phi^{j,a}  }_{\hat{M}^{\alpha}_{2,\sigma}}
	\le C 
	\norm{ \phi^{j,a}  }_{\hat{M}^{\alpha}_{2,\sigma}}\]
for $\alpha' < \sigma < \frac{6\alpha}{3\alpha-2}$.
Since $3\alpha>\sigma$, we have
\begin{align*}
	C_\alpha m^{3\alpha} &{}\le C\( \sup_{j,a} \norm{ e^{-t\d_x^3} \mathcal{G}^{j,a}_{n} \phi^{j,a} }_{L(\R)}  \)^{3\alpha-\sigma}  \sum_{(j,a) \in \mathcal{A} } \norm{ \phi^{j,a}  }_{\hat{M}^{\alpha}_{2,\sigma}}^\sigma \\
	&{}\le C\(  \sup_{j,a} \norm{ e^{-t\d_x^3} \mathcal{G}^{j,a}_{n} \phi^{j,a} }_{L(\R)}  \)^{3\alpha-\sigma}  M^\sigma.
\end{align*}
Thus, there exists $(j_0,a_0)$ such that
\begin{equation}\label{eq:pf_lbound}
	\norm{ e^{-t\d_x^3} \mathcal{G}^{j_0,a_0}_{n} \phi^{j_0,a_0} }_{L(\R)} 
	\ge C_\alpha\(\frac{m^{3\alpha}}{M^\sigma}\)^{\frac1{3\alpha-\sigma}}.
\end{equation}

Now, up to subsequence, we have
\[
	(\mathcal{G}^{j_0,a_0}_{n})^{-1} u_n \rightharpoonup \phi^{j_0,a_0} + q=:\psi \wIN \hat{L}^\alpha
\]
as $n\to\I$, where $q$ is a weak limit of $(\mathcal{G}^{j_0,a_0}_{n})^{-1} q_n$.
Indeed, by Lemma \ref{lem:pd:step1} (\ref{ufq}), we have
\[
	u_n = \sum_{1\le j \le J ,\,j\neq j_0} f^j_n + f^{j_0}_n + q_n .
\]
As in the proof of the first assertion of Lemma \ref{lem:orthty},
one has
$(\mathcal{G}^{j_0,a_0}_{n})^{-1} f^j_n \rightharpoonup 0$ weakly in $\hat{L}^\alpha$
as $n\to\I$ for $j\neq j_0$.
Further, $(\mathcal{G}^{j_0,a_0}_{n})^{-1} f^{j_0}_n \rightharpoonup \phi^{j_0,a_0}$ as $n\to\I$.
Therefore, we have the above limit.
Then, one sees from Lemma \ref{lem:scale} and 
Lemma \ref{lem:pd:step1} (\ref{qe}) that, for any bounded set $K \subset \Z^2$,
\begin{align*}
	&\sum_{(j,k) \in K} \( |\tau_k^j |^{-\frac{1}{3\alpha}} \norm{\F q}_{L^{(\frac{3\alpha}{2})'}(\tau_k^j)}\)^{\frac{6\alpha}{3\alpha-2}} \\
	&{}\le \limsup_{n\to\I} \sum_{(j,k) \in K} \( |\tau_k^j |^{-\frac{1}{3\alpha}} \norm{\F (\mathcal{G}^{j_0,a_0}_{n})^{-1} q_n}_{L^{(\frac{3\alpha}{2})'}(\tau_k^j)}\)^{\frac{6\alpha}{3\alpha-2}} \\
 	&{}\le C  \limsup_{n\to\I} 
	\norm{q_n}_{\hat{M}^{\alpha}_{\frac{3\alpha}2,\frac{6\alpha}{3\alpha-2}}}^{\frac{6\alpha}{3\alpha-2}}
	\le C\epsilon^{\frac{6\alpha}{3\alpha-2}}.
\end{align*}
Taking supremum in $K$, one obtains
	$\norm{q}_{\hat{M}^{\alpha}_{\frac{3\alpha}2,\frac{6\alpha}{3\alpha-2}}}
% 	\norm{|I|^{-\frac1{r}} \norm{\F q}_{L^{\frac{r}{r-2}}(I)}}_{\ell^{\frac{2r}{r-2}}_{\mathcal{D}}}
	\le C \epsilon$.
Thanks to \eqref{eq:ST3} and Lemma \ref{lem:scale},
\[
	\norm{e^{-t\d_x^3} \mathcal{G}^{j_0,a_0}_{n} q}_{L(\R)} 
	\le \norm{ \mathcal{G}^{j_0,a_0}_{n} q}_{\hat{M}^{\alpha}_{\frac{3\alpha}2,\frac{6\alpha}{3\alpha-2}}}
	\le C \epsilon.
\]
Finally, using Theorem \ref{thm:rST} \eqref{eq:ST3}, 
Lemma \ref{lem:scale}, and \eqref{eq:pf_lbound}, and
choosing $\epsilon=\epsilon(m,M)>0$ even smaller if necessary, we reach to the estimate
\begin{align*}
	\norm{\psi}_{\hat{M}^{\alpha}_{2,\frac{6\alpha}{3\alpha-2}}}
	&{}\ge C\norm{ e^{-t\d_x^3} \mathcal{G}^{j_0,a_0}_{n} \psi}_{L(\R)} \\
	&{}\ge C\norm{ e^{-t\d_x^3} \mathcal{G}^{j_0,a_0}_{n} \phi^{j_0,a_0}}_{L(\R)} -C\eps \\
	&{}\ge \frac{C}2\(\frac{m^{3\alpha}}{M^\sigma}\)^{\frac1{3\alpha-\sigma}} =: \beta(m,M),
\end{align*}
which completes the proof of Theorem \ref{thm:cc}. 
%$\qed$

\subsection{Two improvements in special cases} 
We consider two improvements of Theorem \ref{thm:pd},
under some additional assumptions.

The first one is the case when functions  in a sequence are real-valued.
This is nothing but the case of Theorem \ref{thm:pd_r}.

\vskip2mm
\noindent
{\it Proof of Theorem \ref{thm:pd_r}.}
In addition to the assumption of Theorem \ref{thm:pd},
we assume that $u_n$ is real valued.
We already have a decomposition
\[
	u_n = \sum_{j=1}^J \mathcal{G}^j_n \psi^j + r_n^J
\]
by Theorem \ref{thm:pd}.
We now show that this is rewritten as in \eqref{eq:pf:decomp}.
Fix $j$.
If $|\xi_n^j|$ is bounded in $n$ then, extracting subsequence if necessary,
$\xi_n^j \to \xi^j\in \R$ as $n\to\I$. Then, $(\mathcal{G}^j_n)^{-1} u_n \rightharpoonup
\psi^j$ in $\hat{L}^{\alpha}$ implies
\[
	A(s_n^j)^{-1} T (y_n^j)^{-1} D(h_n^j)^{-1} u_n \rightharpoonup P(\xi^j) \psi^j
	\IN \hat{L}^{\alpha}.
\]
Since the left hand side is real-valued, so is the right hand side.
Then, $P(\xi^j)\psi^j  =\Re (P(\xi^j)\psi^j)$. 
Denoting $P(\xi^j) \psi^j$ again by $\psi^j$, we may let $\xi^j_n \equiv 0$.

Next consider the case $|\xi_n^j| \to \I$ as $n\to\I$.
In particular, assume that $\xi_n^j \to\I$ as $n\to\I$.
Then, the convergence $\overline{(\mathcal{G}^j_n)^{-1} u_n} \rightharpoonup
\overline{\psi^j}$ in $\hat{L}^{\alpha}$ implies
\[
	P(-\xi_n^j)^{-1} A(s_n^j)^{-1} T (y_n^j)^{-1} D(h_n^j)^{-1} u_n \rightharpoonup \overline{\psi^j}
	\IN \hat{L}^{\alpha}.
\]
Therefore,  there exists $k$ such that
$\{\mathcal{G}_n^k\}_n$ is not orthogonal to the family $\{\overline{\mathcal{G}^j}_n\}_n:= \{ D(h_n^j)T (y_n^j) A(s_n^j) P(-\xi_n^j)\}_n$.
% $(h_n^k, y_n^k, s_n^k, \xi_n^k, \psi^k) = (h_n^j, y_n^j, s_n^j, -\xi_n^j, \overline{\psi^j})$.
Indeed, if not then the above convergence implies $\eta(r_n^J)\ge \hMn{\overline{\psi^j}}{\alpha}{2}{\sigma}$ for all $J\ge1$, a contradiction.
Then, one can replace $\{ \mathcal{G}_n^k\}_n$ and $\psi^k$
by $\{\overline{\mathcal{G}^j}_n\}_n$ and $\overline{\psi^j}$, respectively.
Denoting $\psi^j/2$ again by $\psi^j$, we obtain the result.
This is the reason why $c_j=2$ when $|\xi_n^j|\to\I$ as $n\to\I$. 
This completes the proof of Theorem \ref{thm:pd_r}. 
%$\qed$

\vskip2mm

The second one is exclusion of deformations $D(h)$ and $P(\xi)$ under
uniform boundedness in a stronger topologies.
This is the key for Theorem \ref{thm:minimal2}.

\begin{proposition}\label{prop:pd:addtionalbound}
(i) Under the assumptions in Theorem \ref{thm:pd},
assume in addition that $\{u_n\}_n$ is uniformly bounded in 
$\hat{L}^{\alpha_1} \cap \hat{L}^{\alpha_2}$ for some 
$1< \alpha_1 < \alpha < \alpha_2 < \I$.
Then, the assertions of Theorem \ref{thm:pd} hold with 
$h_n^j \equiv 1$. Furthermore, we have 
$$\norm{\psi^j}_{\hat{L}^{\rho}} \le \limsup_{n\to\I}\norm{u_n}_{\hat{L}^{\rho}} $$
for all $j\ge 1$ and $\alpha_1 \le \rho \le \alpha_2$.

\vskip1mm
\noindent
(ii) In addition to the assumption of Theorem \ref{thm:pd},
if $\{u_n\}_n$ is uniformly bounded in $\hat{L}^{\alpha_1} \cap \dot{H}^s$ for some
$1< \alpha_1 < \alpha $ and $s>0$
then, the assertions of Theorem \ref{thm:pd} hold with $h_n^j \equiv 1$,
$\xi_n^j \equiv 0$. Furthermore, we have
$$\norm{\psi^j}_{\dot{H}^s} \le \limsup_{n\to\I}\norm{u_n}_{\dot{H}^s} $$
for all $j\ge1$.
\end{proposition}

\begin{proof}
%\vskip2mm
%\noindent
%{\it Proof of Proposition \ref{prop:pd:addtionalbound}.} 
Suppose that $u_n$ is uniformly bounded in $\hat{L}^{\alpha_1} \cap \hat{L}^{\alpha_2}$,
$\alpha_1<\alpha<\alpha_2$, and that
\[
	P(\xi_n)^{-1} A(s_n)^{-1} T (y_n)^{-1} D_\alpha(h_n)^{-1} u_n \rightharpoonup \psi
	\IN \hat{L}^{\alpha}
\]
for some $\psi \not \equiv 0$.
Then, for $g \in C^\I$ such that $\hat{g}$ has a compact support,
\begin{align*}
	|(\psi, g)|
	&{}\le 2\abs{\int (P(\xi_n)^{-1} A(s_n)^{-1} T (y_n)^{-1} D_\alpha(h_n)^{-1} u_n)(x)
	\overline{g(x)} dx} \\
 	&{}= 2\abs{\int u_n(x)
 	\overline{(D_{\alpha'}(h_n)T(y_n) A(s_n)P(\xi_n) g)( x)} dx} \\
	&{}\le 2\norm{u_n}_{\hat{L}^{\alpha_1} \cap \hat{L}^{\alpha_2}}
	\norm{D_{\alpha'}(h_n)T(y_n) A(s_n)P(\xi_n) g}_{\hat{L}^{\alpha_1'} + \hat{L}^{\alpha_2'} } \\
	&{}\le C \norm{D_{\alpha'}(h_n)T(y_n) A(s_n)P(\xi_n) g}_{\hat{L}^{\alpha_1'} + \hat{L}^{\alpha_2'} }
\end{align*}
for large $n$.
If $h_n \to 0$ as $n\to\I$ then
\[
	\norm{D_{p'}(h_n)T(y_n) A(s_n)P(\xi_n) g}_{\hat{L}^{\alpha_1'}}
	= (h_n)^{\frac1{\alpha_1}-\frac1\alpha} \norm{g}_{\hat{L}^{\alpha_1'}}
	\to 0
\]
as $n\to\I$.
Similarly,
if $h_n \to \I$ as $n\to\I$ then 
$$\norm{D_{\alpha'}(h_n)T(y_n) A(s_n)P(\xi_n) g}_{\hat{L}^{\alpha_2'}}
= (h_n)^{\frac1{\alpha_2}-\frac1\alpha} \norm{g}_{\hat{L}^{\alpha_2'}}
\to 0$$ as $n\to\I$.
In the both cases, we have $\psi \equiv 0$, a contradiction.
Thus, we conclude that $|\log h_n|$ is bounded.
Extracting subsequence, we have $h_n \to h_0 >0$ as $n\to\I$.
Then,
\begin{align*}
	&{} P(h_n \xi_n)^{-1} A(s_n/(h_n)^3)^{-1} T (y_n/h_n)^{-1}  u_n \\
	&{}= D_\alpha(h_n) P(\xi_n)^{-1} A(s_n)^{-1} T (y_n)^{-1} D_\alpha(h_n)^{-1} u_n 
	\rightharpoonup D_\alpha(h_0)\psi
	\IN \hat{L}^{p}.
\end{align*}
Hence, denoting $(h_n \xi_n, s_n/(h_n)^3,y_n/h_n)$ and $D_\alpha(h_0)\psi$
 again by $(\xi_n,s_n,y_n)$ and $\psi$, respectively,
we may let $h_n\equiv 1$. Under the new notation,
we have
\[
	\norm{\psi}_{\hat{L}^{\rho}}
	\le \limsup_{n\to\I} \norm{P(\xi_n)^{-1} A(s_n)^{-1} T (y_n)^{-1}  u_n}_{\hat{L}^\rho}
	= \limsup_{n\to\I} \norm{u_n}_{\hat{L}^\rho}.
\]
for all $\alpha_1 \le \rho \le \alpha_2$.

Next, let us suppose that $u_n$ is bounded in $\hat{L}^{\alpha_1} \cap \dot{H}^{s}$ ($\alpha_1<\alpha$, $s>0$).
Note that this implies $u_n$ is bounded in $L^2$.
Hence, the above argument gives us $h_n^j\equiv 1$ for all $j\ge 1$.
Let us show $\xi_n^j\equiv 0$ for all $j\ge1$.
For $g \in C^\I$ such that $\hat{g}$ has a compact support, we have
\begin{align*}
	|(\psi, g)|
	&{}\le 2\abs{\int (P(\xi_n)^{-1} A(s_n)^{-1} T (y_n)^{-1}u_n)(x)
	\overline{g(x)} dx} \\
% 	&{}= \abs{\int (D_p(h_n)^{-1} u_n)(x)
% 	\overline{(T(y_n) A(s_n)P(\xi_n) g)(x)} dx} \\
% 	&{}= \abs{\int h_n^{-\frac1p} u_n(x/h_n)
% 	\overline{(T(y_n) A(s_n)P(\xi_n) g)(x)} dx} \\
 	&{}= 2\abs{\int u_n(x)
 	\overline{(T(y_n) A(s_n)P(\xi_n) g)( x)} dx} \\
	&{}\le 2\norm{u_n}_{\dot{H}^s}
	\norm{T(y_n) A(s_n)P(\xi_n) g}_{\dot{H}^{-s}} \le C \norm{P(\xi_n) g}_{\dot{H}^{-s}}
\end{align*}
for large $n$.
If $|\xi_n| \to \I$ as $n\to\I$ then $\norm{P(\xi_n) g}_{\dot{H}^{-s}} \to 0$
as $n\to\I$.
This gives us $\psi \equiv 0$, a contradiction.
Hence, $\xi_n$ is bounded. By extracting subsequence, $\xi_n \to \xi_0 \in \R$
as $n\to \I$.
Then,
\[
	A(s_n)^{-1} T (y_n)^{-1}  u_n 
	= P(\xi_n) P(\xi_n)^{-1} A(s_n)^{-1} T (y_n)^{-1} u_n 
	\rightharpoonup P(\xi_0)\psi
	\IN \hat{L}^{p}.
\]
Thus, denoting $P(\xi_0)\psi$ again by $\psi$, we may let $\xi_n \equiv 0$
and we have the bound $\norm{\psi}_{\dot{H}^s} \le \limsup_{n\to\I} \norm{u_n}_{\dot{H}^s}$.
\end{proof}%$\qed$

%%%%%%%%%%%%%%%%%%%%%%%%%%%%%%%%%%%%%%%%%%%%%%%%%%%%
%
%     Chapter 7
%
%%%%%%%%%%%%%%%%%%%%%%%%%%%%%%%%%%%%%%%%%%%%%%%%%%%%

\section{Quick review on well-posedness of \eqref{NLS}}\label{sec:NLS}

In this section, we briefly summarize well-posedness   
and stability  results for \eqref{NLS} which are need to 
prove Theorem \ref{thm:ENK}. 

%----------------------------------------
%   Well-posedness for NLS
%----------------------------------------
\subsection{Well-posedness for NLS}

We first consider well-posedness for \eqref{NLS} in 
$\hat{L}^{\alpha}$-space and $\hM{\alpha}{\rho}\sigma$-space. 
% In principle, well-posedness in these spaces are consequence of result by
% Vargas and Vega \cite{VV}.
The initial value problem \eqref{NLS} is formulated as
\begin{equation}\label{INLS}
	v(t) = e^{-it\d_x^2} v_0 + i\mu\int_{0}^{t} e^{-i(t-t')\d_x^2} 
	(|v|^{2\alpha}v)(t') dt' .
\end{equation}

The following well-posedness result  plays an important role in this subsection. 
This kind of result is well known (see \cite{Kato2, VV}, for example).

\begin{proposition}\label{prop:LWP_NLS} 
Let $4/3<\alpha < 4$. Then 
there exists a number $\delta>0$ such that
if a data $v_0 \in \mathcal{S}'$ and
an interval $I \ni 0$  satisfies
\[
	\norm{e^{-it\d_x^2} v_0}_{L^{3\alpha}_{t,x}(I\times \R)}
	\le \delta
\]
then there exists a unique solution $v(t)$ to \eqref{INLS} which satisfies
\[
	\norm{v}_{L^{3\alpha}_{t,x}(I\times \R)} \le 2 \norm{e^{-it\d_x^2} v_0}_{L^{3\alpha}_{t,x}(I\times \R)}.
\]
Further, the solution belongs to
$L^p_t(I,L^q_x)$ for any $p,q \in (2,\I)$ with $2/p + 1/q = 1/\alpha$.
\end{proposition}

\begin{proof}
%\vskip2mm
%\noindent
%{\it Proof of Proposition \ref{prop:LWP_NLS}.} 
Proposition \ref{prop:LWP_NLS} is an immediate consequence of the estimate
\[
	\norm{\Phi[v]}_{L^{3\alpha}_{t,x}(I\times \R)}
	\le \norm{e^{-it\pt_{x}^{2}}v_0}_{L^{3\alpha}_{t,x}(I\times \R)}
	+ C\norm{v}_{L^{3\alpha}_{t,x}(I\times \R)}^{2\alpha+1}
\]
for $4/3<\alpha<4$,
where $\Phi[v]$ is the right hand side of \eqref{INLS}.
This inequality follows from Strichartz's estimate for non-admissible pairs 
(see Kato \cite{Kato2} or Lemma \ref{Nn} (ii), below),
and H\"older inequality. 
\end{proof}%$\qed$

\vskip2mm

\begin{remark}
Well-posedness of \eqref{INLS} in a space like $L^{p}_t(I;L^q_x)$ also
holds for $\frac{1+\sqrt{17}}{4}<\alpha \le 4/3$ if we allow the case $p\neq q$.
\end{remark}

To prove well-posedness in 
$\hat{L}^{\alpha}$-space, we show the following 
generalized Strichartz estimate for the Schr\"{o}dinger equation.

\begin{lemma}\label{Nn} 
(i) (homogeneous estiamtes) Let $I$ be an interval. 
Let $(p,q)$ satisfy
\[
	0\le\frac{1}{p}<\frac14,\quad\quad 0 \le \frac1q < \frac12 - \frac1p.
\]
Then, for any $f\in\hat{L}^r$,
\begin{equation}\label{eq:mixed}
	\norm{|\d_x|^\tau e^{-it\d_x^2} f}_{L^p_x L^q_t(I)} \le C\norm{f}_{\hat{L}^r},
\end{equation}
where
\[
	\frac1r = \frac2p + \frac1q,\quad \tau=-\frac1p+\frac1q.
\]
and positive constant $C$ depends 
only on $r$ and $s$. 

\vskip1mm
\noindent
(ii) (inhomogeneous estimates) Let $4/3<r<4$ and let 
$(p_{j},q_{j})$ ($j=1,2$) satisfy
\[
	0\le\frac{1}{p_{j}}<\frac14,\quad\quad 0 \le \frac{1}{q_{j}} < 
	\frac12 - \frac1p_{j}.
\]
Then, the inequalities 
\begin{equation}
\left\|
\int_0^te^{-i(t-t')\pt_{x}^{2}}F(t')dt'
\right\|_{L_t^{\infty}(I;\hat{L}_x^{r})}
\le C_{1}
\||\d_x|^{-\tau_{2}}
F\|_{L_x^{p_{2}'}L_{t}^{q_{2}'}(I)},
\label{k}
\end{equation}
and
\begin{equation}
\left\||\d_x|^{\tau_{1}}
\int_0^te^{-i(t-t')\pt_{x}^{2}}F(t')dt'
\right\|_{L_x^{p_{1}}L_{t}^{q_{1}}(I)}
\le C_{2}
\||\d_x|^{-\tau_{2}}
F\|_{L_x^{p_{2}'}L_{t}^{q_{2}'}(I)} \label{lllll}
\end{equation}
hold for any $F$ satisfying $|\d_x|^{-\tau_{2}}
F\in L_x^{p_{2}'}L_{t}^{q_{2}'}$, where
\[
	\frac1r = \frac{2}{p_{1}} + \frac{1}{q_{1}},
	\quad \tau_{1}=-\frac{1}{p_{1}}+\frac{1}{q_{1}}
\]
and
\[
	\frac{1}{r'} = \frac{2}{p_{2}} + \frac{1}{q_{2}},
	\quad \tau_{2}=-\frac{1}{p_{2}}+\frac{1}{q_{2}},
\]
where 
the constant $C_{1}$ depends on $r$, $\tau_{1}$ and $I$, and 
the constant $C_{2}$ depends on $r$, $\tau_{1}$, $\tau_{2}$ and $I$.
\end{lemma}

\begin{remark}
Remark that we take a space-time norm of the form $L_x^p L_t^q$ in \eqref{eq:mixed}.
This is why we gain derivative by $|\d_x|^\tau$.
Also remark that a similar estimate for a space-time norm of the form $L_t^p L_x^q$ is known in \cite{HT}.
\end{remark}

\begin{proof}
%\vskip2mm
%\noindent
%{\it Proof of Lemma \ref{Nn}.} The homogeneous estimate 
(\ref{eq:mixed}) is obtained by interpolating the Kato's 
smoothing effect \cite[Theorem 4.1]{KPV1}, 
the Kenig-Ruiz estimate \cite[Theorem 2.5]{KPV1} 
and the Stein-Tomas estimate for the Schr\"{o}dinger equation
\footnote{This estimate goes back to \cite{To}. 
We can prove this inequality 
by using the argument similar to \cite[Lemma 2.2]{MS}}
\begin{eqnarray*}
\norm{e^{-it\d_x^2} f}_{L^{3r}_x L^{3r}_t(I)} \le C\norm{f}_{\hat{L}^r}
\end{eqnarray*}
for $r>4/3$. 

The inhomogeneous estimates (\ref{k}) and (\ref{lllll}) follows 
from the homogeneous estimate (\ref{eq:mixed}) and the Christ-Kiselev 
lemma by \cite[Lemma 2]{MR}. 
\end{proof}%$\qed$

\vskip2mm

Inequality \eqref{eq:mixed} and the following inequality yields the
local well-posedness in $\hat{L}^{\alpha}$ and $\hM{\alpha}{\frac{3\alpha}2}{2(\frac{3\alpha}2)'}$,
respectively;
\begin{proposition}\label{prop:rSTNLS}
Assume that $ \alpha >4/3$. Then,
\[
	\norm{e^{-it\d_x^2} f}_{L^{3\alpha}_{t,x}(\R\times \R)}
	\le C \hMn{f}{\alpha}{\frac{3\alpha}2}{2(\frac{3\alpha}2)'} 
\]
holds for all $f\in \hM{\alpha}{\frac{3\alpha}2}{2(\frac{3\alpha}2)'}$.
Further, the embedding $\hat{L}^\alpha \hookrightarrow \hM{\alpha}{\frac{3\alpha}2}{2(\frac{3\alpha}2)'}$ holds
if $\alpha > 4/3 $.
% $\hMn{f}{\alpha}{2}{\sigma} \le C \norm{f}_{\hat{L}^{\alpha}}$
% holds for all $f\in \hat{L}^{\alpha}$.
\end{proposition}

The inequality is shown as in Theorem \ref{thm:rST} (see Remark \ref{rem:rSTNLS}).
The $\alpha=2$ case is given in \cite{BV,CK}.
Now, let us see how the well-posedness results are deduced.
If either $v_0 \in \hM{\alpha}{\frac{3\alpha}2}{2(\frac{3\alpha}2)'}$ or $v_0 \in \hat{L}^{\alpha}$ then
the above inequalities imply that $\|e^{-it\d_x^2} v_0\|_{L^{3\alpha}_{t,x}(I\times \R)}
\le \delta$ holds at least for small interval $I=I(v_0)$.
Then, we obtain a solution $u(t)$ on $I$ belonging to $L^{3\alpha}_{t,x}(I \times \R)$ thanks to
Proposition \ref{prop:LWP_NLS}.
Further, by applying \eqref{k}, we see that 
\[
	\norm{\Phi[v]-e^{-it\d_x^2}v_0 }_{L^\I_t(I,\hat{L}^{\alpha}_x)}
	\le C\norm{v}_{L^{3\alpha}_{t,x}(I\times \R)}^{2\alpha+1}.
\]
Finally, the linear part $e^{-it\d_x^2}v_0$ belongs to $C(\R; \hat{L}^\alpha)$
(resp. $C(\R;\hat{M}^{\alpha}_{\frac{3\alpha}2,2(\frac{3\alpha}2)'})$)
if $v_0\in\hat{L}^\alpha$ (resp. $v_0\in \hat{M}^\alpha_{\frac{3\alpha}2,2(\frac{3\alpha}2)'}$).
Thus, we obtain the following.

\begin{proposition}[Local well-posedness in $\hat{L}^{\alpha}$ and $\hM{\alpha}{\frac{3\alpha}2}{2(\frac{3\alpha}2)'}$]
Let $4/3<\alpha<4$. 

\vskip1mm
\noindent
(i) For any $u_0 \in \hat{L}^{\alpha}_x$, there exists a unique solution 
$u(t) \in C (I; \hat{L}^{\alpha})$.
% $u(t) \in C (I; \hat{L}^{\alpha}) \cap L^p_t( I ; L^q_x)$, where 
%  $p,q \in (2,\I)$ are any numbers
% satisfying $2/p + 1/q = 1/\alpha$.

\vskip1mm
\noindent
(ii) For any $u_0 \in \hM{\alpha}{\frac{3\alpha}2}{2(\frac{3\alpha}2)'}$, 
there exists a unique solution 
$u(t) \in C (I; \hM{\alpha}{\frac{3\alpha}2}{2(\frac{3\alpha}2)'})$.
Furthermore, $u(t) - e^{-it\d_x^2}u_0 \in C(I, \hat{L}^{\alpha})$ holds.
% there exists a unique solution 
% $u(t) \in C (I; \hM{\alpha}{2}{\sigma}) \cap L^p_t( I ; L^q_x)$, where 
%  $p,q \in (2,\I)$ are any numbers satisfying $2/p + 1/q = 1/\alpha$.
% Furthermore, $u(t) - e^{-it\Delta}u_0 \in C(I, \hat{L}^{\alpha})$ holds.
\end{proposition}

\begin{remark}
It is obvious from the proof that a similar well-posedness result holds in all 
$\hat{M}^\alpha_{\rho,\sigma}$ space satisfying
\[
	\hat{L}^{\alpha} \hookrightarrow \hat{M}^{\alpha}_{\rho,\sigma} \hookrightarrow
\hM{\alpha}{\frac{3\alpha}2}{2(\frac{3\alpha}2)'}.
\]
Notice that the $\hat{M}^{\alpha}_{2,\sigma}$ space satisfies the above relation
if $4/3 < \alpha < 2$ and $\alpha' < \sigma \le 2(\frac{3\alpha}2)'= 6\alpha/(3\alpha-2)$.
This is nothing but Theorem \ref{thm:LWP_N2}.
On the other hand,
the first assertion of the above proposition is Theorem \ref{thm:LWP_N1}.
\end{remark}

As a corollary of this proposition, 
we obtain small data scattering in $\hM{\alpha}{\frac{3\alpha}2}{2(\frac{3\alpha}2)'}$.

\begin{corollary}
Let $4/3<\alpha<4$ and . Assume that $v_0 \in \hat{L}^{\alpha}$
or $v_0 \in \hM{\alpha}{\frac{3\alpha}2}{2(\frac{3\alpha}2)'}$. %Let $v(t)$ be a unique solution given in the previous theorem.
There exists $\eps >0$ such that
if $\hM{\alpha}{\frac{3\alpha}2}{2(\frac{3\alpha}2)'} < \eps$ then $v_0 \in \mathcal{S}_{\mathrm{NLS}}$.
\end{corollary}

\subsection{Persistence of regularity for NLS}
Next we show the persistent property of solution 
to (\ref{NLS}). 

\begin{lemma}[Persistence of $L^p_xL^q_t$- and $L^p_tL^q_x$-regularities]
\label{pri} 
Let $4/3<\alpha<4$ and $s\ge0$. Let $\hat{t}
\in\rre$ and let $I$ be a time interval containing 
$\hat{t}$. 
Assume that $v\in 
C(I;\hat{L}_{x}^{\alpha}(\rre))$ 
is a solution to (\ref{NLS}) satisfying 
$\|v\|_{L^{3\alpha}_{t,x}(I\times \R)}\le M$ for some $M$.
Then, the following two assertions hold:

\vskip1mm
\noindent
(i) If $|\d_x|^{s}v(\hat{t})\in \hat{L}^{\alpha}(\rre)$
then, for any %$1/\alpha-3/4<\tau<3/2-2/\alpha$,
\[
	\tau \in 
	\begin{cases}
	(\frac1\alpha-\frac34, \frac32 - \frac2\alpha) & \text{ if } \alpha <2,\\
	[-\frac1{2\alpha}, \frac1\alpha] & \text{ if } \alpha \ge 2,
	\end{cases}
\]
there exists
a constant $C=C(\alpha, s, \tau, M)$ such that
\begin{equation} \label{pri1}
\||\d_x|^{s}v
\|_{L_{t}^{\infty}\hat{L}_{x}^{\alpha}(I\times\rre)}
+\||\d_x|^{s+\tau}v\|_{L_{x}^{p}L_{t}^{q}(I)}
\le C\||\d_x|^{s}v(\hat{t})
\|_{\hat{L}^{\alpha}},
\end{equation}
holds, where $(p,q)$ satisfies 
\begin{eqnarray}
\frac{1}{\alpha}=\frac2p+\frac1q,\qquad 
\tau=-\frac1p+\frac1q. \label{prime3}
\end{eqnarray}

\vskip1mm
\noindent
(ii) If $v(\hat{t})\in \dot{H}^s(\rre)$
then, there exists $C=C(M)$ such that
\begin{equation}\label{pri2}
\|v\|_{L_{t}^{\infty}(I;\dot{H}^{s}(\rre))}
+\| |\d_x|^s v
\|_{L_{t}^{p}(I;L_{x}^{q}(\R))} 
\le C\|v(\hat{t})\|_{\dot{H}^{s}}.
\end{equation}
holds, where $(p,q)$ satisfies 
\begin{eqnarray}
0\le\frac 1p\le\frac14,\qquad
\frac12=\frac2p+\frac1q.
\label{prime4}
\end{eqnarray}
\end{lemma}

\begin{proof}
%\vskip2mm
%\noindent
%{\it Proof of Lemma \ref{pri}.} 
% We abbreviate $L_{2}$ to $L$. 
Without loss of generality, we may assume 
that $\hat{t}=0$ and $\inf I=0$. 
We divide the time interval $I$ into 
$N$ subintervals such that 
\begin{eqnarray*}
N\le1+\left(\frac{M}{\eta}\right)^{3\alpha},
\qquad I=\bigcup_{j=1}^{N}I_{j},\qquad 
I_{j}=[t_{j-1},t_{j}]
\end{eqnarray*}
with 
$\|v\|_{L^{3\alpha}_{t,x}(I_{j}\times \R)}\le\eta$
for any $1\le j\le N$, where $\eta$ is fixed later. 
Notice that such subdivision exists by the argument 
similar to the proof of Proposition \ref{lsy}. 

We shall prove \eqref{pri1}. To this end, we show
\begin{eqnarray}
\qquad\||\d_x|^{s}v
\|_{L_{t}^{\infty}\hat{L}_{x}^{\alpha}
(I_{j}\times\rre)}
+\||\d_x|^{s+\tau}v
\|_{L_{x}^{p}L_{t}^{q}(I_{j})}
\le C\||\d_x|^{s}v(t_{j})
\|_{\hat{L}^{\alpha}}
\label{r10}
\end{eqnarray}
for any $1\le j\le N$, where $p,q$ satisfy (\ref{prime3}). 
We first consider the case $j=1$. 
By Lemma \ref{Nn}, we have
\begin{eqnarray*}
\lefteqn{\||\d_x|^{s}v
\|_{L_{t}^{\infty}\hat{L}_{x}^{\alpha}
(I_{j}\times\rre)}
+\||\d_x|^{s+\tau}v
\|_{L_{x}^{p}L_{t}^{q}(I_{j})}+
\||\d_x|^{s}v\|_{L^{3\alpha}_{t,x}(I_{j}\times \R)}
}\\
&\le&C\||\d_x|^{s}v(0)
\|_{\hat{L}^{\alpha}}
+C\||\d_x|^{s}(|v|^{2\alpha}v)
\|_{L_{t,x}^{\frac{3\alpha}{2\alpha+1}}(I_{j})}\\
&\le&C\||\d_x|^{s}v(0)
\|_{\hat{L}^{\alpha}}
+C\|v\|_{L^{3\alpha}_{t,x}(I_{j}\times \R)}^{2\alpha}
\||\d_x|^{s}v\|_{L^{3\alpha}_{t,x}(I_{j}\times \R)}\\
&\le&C\||\d_x|^{s}v(0)
\|_{\hat{L}^{\alpha}}
+C\eta^{2\alpha}
\||\d_x|^{s}v\|_{L^{3\alpha}_{t,x}(I_{j}\times \R)}.
\end{eqnarray*}
Choosing $\eta$ sufficiently small so that 
$C\eta^{2\alpha}<1$, we have (\ref{r10}) for 
$j=1$. In particular, we obtain  
$\||\d_x|^{s}v(t_{1})\|_{\hat{L}^{\alpha}}\le C$. 
Hence a similar argument as above we have 
(\ref{r10}) for $j=2$. Repeating this argument, 
we obtain (\ref{r10}) for any $1\le j\le N$. 
Summing the inequalities (\ref{r10}) over all 
subintervals, we have (\ref{pri1}). 

The proof of \eqref{pri2} is done in a similar way.
We use (usual) Strichartz's estimates instead.
\end{proof}%$\qed$

%----------------------------------------
%   Stability for NLS
%----------------------------------------
\subsection{Stability for NLS}

In this section we consider the 
nonlinear Schr\"{o}dinger equation 
with the perturbation:
\begin{eqnarray}
\left\{
\begin{array}{l}
\displaystyle{
i\pt_t\tilde{v}-\pt_x^2\tilde{v}=
-\mu|\tilde{v}|^{2\alpha}\tilde{v}+e,
\qquad t,x\in\rre,}\\
\displaystyle{\tilde{v}(\hat{t},x)=\tilde{v}_{0}(x),
\qquad\qquad\qquad\qquad \   x\in\rre}
\end{array}
\right.
\label{NS}
\end{eqnarray}
with the perturbation $e$ small in a suitable 
sense and the initial data $\tilde{v}_{0}$ 
close to $v_{0}$. 

\vskip2mm

\begin{proposition}[Long time stability 
for NLS]\label{lsys}
Assume $4/3<\alpha<4$ and $\hat{t}\in\rre$. 
Let $I$ be a time interval containing $\hat{t}$ 
and let $\tilde{v}$ be a solution to (\ref{NS}) 
on $I\times\rre$ for some function $e$. 
Assume that $\tilde{v}$ satisfies 
% \begin{eqnarray*}
% \|\tilde{v}\|_{L_{t}^{\infty}(I;
% \hat{L}_{x}^{\alpha})}&\le& A,\\
% \|\tilde{v}\|_{L^{3i}_{t,x}(I\times \R)}
% &\le&L,
% \end{eqnarray*}
\[
	\|\tilde{v}\|_{L^{3\alpha}_{t,x}(I\times \R)}\le M,
\]
for some %$A>0$ and 
$M>0$. 
Then there exists 
$\varepsilon_{1}=\varepsilon_1(M)
%\varepsilon_{1}(A,A',L)
>0$ 
such that if $v(\hat{t})$ and $\tilde{v}(\hat{t})$ satisfy
\begin{eqnarray*}
\|e^{-i(t-\hat{t})\pt_{x}^{2}}
(v(\hat{t})-\tilde{v}(\hat{t}))
\|_{L^{3\alpha}_{t,x}(I\times \R)}+
\|e\|_{L_{t,x}^{\frac{3\alpha}{2\alpha + 1}}(I\times\rre)}
\le\varepsilon
\end{eqnarray*}
and $0<\varepsilon<\varepsilon_{1}$, 
then there exists a solution 
$v\in %C(I;\hat{L}_{x}^{\alpha}(\rre))
{L^{3\alpha}_{t,x}(I\times \R)}$ 
to (\ref{NLS}) on 
$I\times\rre$ satisfies 
\begin{eqnarray}
\|v-\tilde{v}\|_{L^{3\alpha}_{t,x}(I\times \R)}
&\le&C\varepsilon,\label{tt10s}\\
\|
|v|^{2\alpha}v
-|\tilde{v}|^{2\alpha}\tilde{v}
\|_{L_{t,x}^{\frac{3\alpha}{2\alpha+1}}
(I\times\rre)}&\le&C\varepsilon,\label{tt11s}
% \\
% \|v-\tilde{v}\|_{L_{t}^{\infty}(I;
% \hat{L}_{x}^{\alpha})}
% &\le&C,\label{tt12s}
% \\
% \|v\|_{L^{3\alpha}_{t,x}(I\times \R)}
% &\le&C,\label{tt13s}
\end{eqnarray}
where the constant $C$ depends on %$A,A'$ and 
$L$. If, further, if $v(\hat{t})-\tilde{v}(\hat{t}) \in \hat{L}^{\alpha}$ then
\begin{equation}\label{tt12s}
	\|v -\tilde{v}\|_{L^\I(I; \hat{L}^{\alpha}_x)}
	\le \| v(\hat{t})-\tilde{v}(\hat{t}) \|_{\hat{L}^{\alpha}_x}+ C\varepsilon.
\end{equation}
\end{proposition}

%\vskip2mm
%\noindent 
%{\it Proof of Proposition \ref{lsys}.} 
\begin{proof}
The proof follows from the argument similar to the proof of Proposition \ref{lsy}
or as in \cite{M1}.
We omit the detail. 
\end{proof}%$\qed$

%%%%%%%%%%%%%%%%%%%%%%%%%%%%%%%%%%%%%%%%%%%%%%%%%%%%
%
%     Chapter 8
%
%%%%%%%%%%%%%%%%%%%%%%%%%%%%%%%%%%%%%%%%%%%%%%%%%%%%
\section{Embedding NLS into gKdV}\label{sec:ENK}

In this section, we prove  
Theorem \ref{thm:ENK}. 
As we mentioned in Introduction, 
we prove existence of a global solution $u_{n}$ to (\ref{gKdV}) 
by constructing approximating solution via the solution to 
the one dimensional nonlinear 
Schr\"{o}dinger equation 
\begin{equation}\label{NLS2}
	i\pt_tv-\pt_x^2v=-\mu C_0 |v|^{2\alpha}v,
\end{equation}
where 
\[
	C_0 = \frac{2\Gamma(\alpha+\frac{3}2)}{3\sqrt{\pi} \Gamma(\alpha+2)}.
\]
With this constant, assumption \eqref{asmp:gKdV_NLS} is written as 
$$ d_{+,\mathrm{gKdV}} < 2^{1-\frac1\sigma}(C_0)^{-\frac1{2\alpha}} d_{\mathrm{NLS}}.$$
Let $v$ be a solution to \eqref{NLS2} with
the following conditions;
\begin{eqnarray}
\left\{
\begin{array}{l}
\displaystyle{
v(T_{0})=e^{-iT_{0}\pt_{x}^{2}}\phi
\qquad\qquad\qquad\qquad\qquad\quad\text{if}\ |T_{0}|<\infty,
}\\
\displaystyle{\lim_{t\to T_{0}}\|v(t)-e^{-it\pt_{x}^{2}}\phi
\|_{\hat{L}_{x}^{\alpha}}=0\qquad\qquad
\qquad\text{if}\ T_{0}=\pm\infty}.
\end{array}
\right.\label{ic11}
\end{eqnarray}
We now claim that $v$ global and scatters for both time direction.
Let us begin with the case $T_0\in\R$.
Remark that if $v$ solves \eqref{NLS2} then $(C_0)^{\frac1{2\alpha}} v$ solves \eqref{NLS}.
Hence, assumption of the theorem yields
\[
	\hMn{ (C_0)^{\frac1{2\alpha}} \phi }{\alpha}{2}{\sigma} < 2^{1-\frac1\sigma} (C_0)^{\frac1{2\alpha}} d_{+,\mathrm{gKdV}} < d_{\mathrm{NLS}}.
\]
Since $e^{-t\d_x^3}$ is isometry in $\hat{M}^{\alpha}_{2,\sigma}$, $(C_0)^{\frac1{2\alpha}} e^{-T_0\d_x^3}\phi \in \mathcal{S}_{+,\mathrm{NLS}}
\cap \mathcal{S}_{-,\mathrm{NLS}}$
and so $v$ scatters for both time direction.
Next, if $T_0=\I$ then by definition $v$ scatters for positive time direction and
$\hMn{v(t)}{\alpha}{2}{\sigma} \to \hMn{\phi}{\alpha}{2}{\sigma}$ as $t\to\I$.
Therefore, we can take $T\in\R$ from maximal existence time of $v$ so that
$\hMn{(C_0)^{1/2\alpha} v(T)}{\alpha}{2}{\sigma}< d_{\mathrm{NLS}}$.
This implies that $v$ scatters also for negative time.
The case $T_0=-\I$ is handled in the same way.
Thus,
\begin{eqnarray*}
v\in C(\rre;\hat{L}_{x}^{\alpha}(\rre))
\cap L_{t,x}^{3\alpha}(\rre\times\rre).
\end{eqnarray*}
We let $v_{\pm}\in
\hat{L}_{x}^{\alpha}$ be scattering states
such that 
\begin{eqnarray}
\lim_{T\to\infty}\|v(\pm T)-e^{\mp iT\pt_{x}^{2}}v_{\pm}
\|_{\hat{L}_{x}^{\alpha}}=0.
\label{sca}
\end{eqnarray}
We further introduce $v_n$ as a solution of \eqref{NLS2}
with
\begin{equation}
\left\{
\begin{aligned}
&v_n (T_{0})=P_{|\xi| \le \xi_n^{1/4}} e^{-iT_{0}\pt_{x}^{2}}\phi
&& \text{if}\ |T_{0}|<\infty,
\\
&\lim_{t\to T_{0}}\|v_n(t)- P_{|\xi| \le \xi_n^{1/4}} e^{-it\pt_{x}^{2}}\phi
\|_{\hat{L}_{x}^{\alpha}}=0
&& \text{if}\ T_{0}=\pm\infty,
\end{aligned}
\right.\label{ic12}
\end{equation}
where $P_{|\xi|\le a}={{\mathcal F}}^{-1}\varphi(\xi){{\mathcal F}}$ 
with even bump function $\varphi$ satisfying $supp\varphi\subset[-a,a]$. 
The long time stability for NLS (Proposition 
\ref{lsys}) yields
\begin{eqnarray}
v_{n}\to v\ \text{in}\ 
C(\rre;\hat{L}_{x}^{\alpha}(\rre))
\cap L_{t,x}^{3\alpha}(\rre\times\rre).
\label{con}
\end{eqnarray}
In particular, $v_{n}$ satisfies the uniform (in $n$) space-time 
bound
\begin{eqnarray*}
% \|v_{n}\|_{L} + 
\|v_n\|_{L_{t,x}^{3\alpha}(\rre\times\rre)} 
\le C(\phi).
\end{eqnarray*} 
By the persistence of regularity for (\ref{NLS}) 
(Lemma \ref{pri}), we obtain
\begin{eqnarray}
\||\d_{x}|^{s}v_{n}\|_{L_{t}^{\infty}
\hat{L}_{x}^{\alpha}}
+\||\d_{x}|^{s+\tau}v
\|_{L_{x}^{p}L_{t}^{q}}
\le C\xi_{n}^{s/4}\label{43}
\end{eqnarray}
for any $s\ge0$, where $1/\alpha-3/4<\tau<3/2-2/\alpha$ 
and $(p,q)$ satisfies (\ref{prime3}). 
Further, since $\|P_{|\xi| \le \xi_n^{1/4}} e^{-iT_{0}\pt_{x}^{2}}\phi\|_{H^s_x} =
O(\xi_n^{\frac{s}4 -\frac18+\frac1{4\alpha}})$
for any $s\ge 0$, it follows that
\begin{eqnarray}
	\| |\d_x|^{s} v_n \|_{L^p_t(\R, L^q_x)} 
	&=& O(\xi_n^{\frac{s}4 -\frac18+\frac1{4\alpha}}),\label{eq:432}\\
	\| |\d_x|^{s} \d_t v_n \|_{L^p_t(\R, L^q_x)} 
	&=& O(\xi_n^{\frac{s+2}4 -\frac18+\frac1{4\alpha}})\nonumber
\end{eqnarray}
for any Schr\"odinger admissible pair $(p,q)$ 
(i.e., $(p,q)$ satisfies (\ref{prime4})) and $0\le s < 2\alpha $.

The convergence \eqref{con} gives us 
\begin{equation}\label{451}
\sup_n \|v_n \|_{L(|t|>T)} \to 0
\end{equation}
as $T\to\infty$.
% Indeed, for any $\eps>0$ there exists $n_0=n_0(\eps)$ such that
% \[
% \sup_{n\ge n_0 }\|v_{n}\|_{L(|t|>T)} \le \sup_{n\ge n_0 } \|v_{n}-v\|_{L} +\|v\|_{L(|t|>T)}
% \le \|v\|_{L(|t|>T)} + \frac{\eps}2.
% \]
% Thus, we can take $T_0=T_0(\eps)$ so that if $T\ge T_0$ then
% it holds that $\|v\|_{L(|t|>T)} \le \frac{\eps}2$
% and that $\|v_n\|_{L(|t|>T)} \le \eps$ for all $n\le n_0(\eps)$.
Similarly, by \eqref{sca} and \eqref{con},
\begin{equation}\label{452}
\sup_n \|v_{n}(\pm T)-e^{\mp iT\pt_{x}^{2}}v_{\pm}
\|_{\hat{L}_{x}^{\alpha}}
\to0
\end{equation} 
as $T\to\infty$.

Next, we construct a global solution 
$u_{n}$ to (\ref{gKdV}). As in \cite{KKSV}, we 
introduce an approximate solution $\tilde{u}$ to (\ref{gKdV}):
\begin{equation}
\label{ap}
\tilde{u}_{n}(t,x):=\left\{
\begin{aligned}
&Re[e^{-ix\xi_{n}-it\xi_{n}^{3}}
v_{n}(-3\xi_{n}t,x+3\xi_{n}^{2}t)],
&&\text{if}\  |t| \le \frac{T}{3\xi_{n}},\\
&e^{-(t-\frac{T}{3\xi_{n}})\pt_{x}^{3}}
Re[e^{-ix\xi_{n}-\frac{i}{3}T\xi_{n}^{2}}
v_{n}(-T,x+\xi_{n}T)],
&&\text{if}\ t>\frac{T}{3\xi_{n}}, \\
&e^{-(t+\frac{T}{3\xi_{n}})\pt_{x}^{3}}
Re[e^{-ix\xi_{n}+\frac{i}{3}T\xi_{n}^{2}}
v_{n}(T,x-\xi_{n}T)],
&& \text{if}\ t<-\frac{T}{3\xi_{n}},
\end{aligned}
\right.
\end{equation}
where $T$ is a large parameter independent of $n$ 
which will be chosen later.

\begin{lemma}[Space-time bound for $\tilde{u}_{n}$] 
\label{qa}
Assume $5/3<\alpha<2$. We have
\begin{eqnarray}
\|\tilde{u}_{n}
\|_{L_{t}^{\infty}(\rre;\hat{L}_{x}^{\alpha})}
+\|\tilde{u}_{n}\|_{L(\rre)}
+\|\tilde{u}_{n}\|_{S(\rre)}
\le C,\label{qaz}
\end{eqnarray}
where $C$ is a positive constant independent of $T$ 
and $n$.
\end{lemma}

\begin{proof}
%\vskip2mm
%\noindent
%{\it Proof of Lemma \ref{qa}.} 
We split the interval of integrals into 
$|t|>T/(3\xi_{n})$ and $|t|\le T/(3\xi_{n})$. 
In the interval $|t|>T/(3\xi_{n})$, each norms 
appearing in the left hand side of (\ref{qaz}) 
are uniformly bounded in $n$ by the homogenous estimate 
for Airy equation (Proposition \ref{ho}) 
and the uniform space-time bound for $v_{n}$ (\ref{43}). 
In the interval $|t|\le T/(3\xi_{n})$, the space-time bound for $v_{n}$ 
(\ref{43}) yields 
\begin{eqnarray*}
\|\tilde{u}_{n}
\|_{L_{t}^{\infty}([-\frac{T}{3\xi_{n}},\frac{T}{3\xi_{n}}];\hat{L}_{x}^{\alpha})}
+
\|\tilde{u}_{n}
\|_{S(|t|\le\frac{T}{3\xi_{n}})}
\le C,
\end{eqnarray*}
where $C$ is a positive constant independent of $T$ and $n$. 
Combining the interpolation and (\ref{43}), we see
\begin{align*}
\|\tilde{u}_{n}
\|_{L(|t|\le\frac{T}{3\xi_{n}})}
\le{}&C\|\tilde{u}_{n}
\|_{L^{3\alpha}_{t,x}(|t|\le{T}/{3\xi_{n}})}^{1-\frac1{3\alpha}}
\|\pt_{x}\tilde{u}_{n}
\|_{L^{3\alpha}_{t,x}(|t|\le{T}/{3\xi_{n}})}^{\frac1{3\alpha}} \\
\le{}&C\xi_{n}^{-\frac{1}{3\alpha}(1-\frac1{3\alpha})
+(1-\frac{1}{3\alpha})\frac1{3\alpha}}=C.
\end{align*}
Collecting the above inequalities, we obtain (\ref{qaz}). 
\end{proof}%$\qed$

\vskip2mm

\begin{lemma}\label{osi} 
Assume $5/3<\alpha<2$. Let $\phi\in
\hat{L}_{x}^{\alpha}$ 
and let $\{\xi_{n}\}_{n\ge1}\subset(0,\infty)$ 
such that $\xi_{n}\to\infty$ as $n\to\infty$. 
Then, we have
\begin{equation}\label{osii}
\begin{aligned}
&\lim_{T\to\infty}
\lim_{n\to\infty}
\|
e^{-t\pt_{x}^{3}}[e^{-ix\xi_{n}}e^{-iT\pt_{x}^{2}}\phi]
\|_{L([0,\infty))}
=0, \\
&\lim_{T\to-\infty}
\lim_{n\to\infty}
\|
e^{-t\pt_{x}^{3}}[e^{-ix\xi_{n}}e^{-iT\pt_{x}^{2}}\phi]
\|_{L((-\I,0])}
=0.
\end{aligned}
\end{equation}
\end{lemma}

%\vskip2mm
%\noindent
%{\it Proof of Lemma \ref{osi}.} 
\begin{proof} 
By Proposition 
\ref{prop:ho_inho}, 
it suffices to prove this lemma when $\phi$ 
satisfies $\hat{\phi}\in C_{c}^{\infty}(\rre)$. 
By the argument similar to \cite{KKSV}, we obtain
\begin{eqnarray*}
\||\d_{x}|^{\frac1{3\alpha}}
e^{-t\pt_{x}^{3}}[e^{-ix\xi_{n}}e^{-iT\pt_{x}^{2}}\phi]
\|_{L_{x}^{\infty}}
\le C\frac{\xi_{n}^{\frac1{3\alpha}}}{
(T-3\xi_{n}t)^{1/2}}\|\phi\|_{L_{x}^{1}}
\end{eqnarray*}
and
\begin{eqnarray*}
\||\d_{x}|^{\frac1{3\alpha}}
e^{-t\pt_{x}^{3}}[e^{-ix\xi_{n}}e^{-iT\pt_{x}^{2}}\phi]
\|_{L_{x}^2}
\le C\xi_{n}^{\frac1{3\alpha}}
\|\phi\|_{H_{x}^{\frac1{3\alpha}}}.
\end{eqnarray*}
Interpolating between the above two inequalities, 
we have
\begin{eqnarray*}
\||\d_{x}|^{\frac1{3\alpha}}
e^{-t\pt_{x}^{3}}[e^{-ix\xi_{n}}e^{-iT\pt_{x}^{2}}\phi]
\|_{L_{x}^{3\alpha}}^{3\alpha}
\le C(\phi)\frac{\xi_{n}}{
(T-3\xi_{n}t)^{\frac{3\alpha}{2}-1}}.
\end{eqnarray*}
Integrating with respect to $t$ variable, we obtain 
\begin{eqnarray*}
\|
e^{-t\pt_{x}^{3}}[e^{-ix\xi_{n}}e^{-iT\pt_{x}^{2}}\phi]
\|_{L([0,\infty))
}^{3\alpha}
\le C(\phi)
T^{-\frac{3\alpha}{2}+1}\to0,
\end{eqnarray*}
as $T\to\infty$. This completes 
the proof. 
\end{proof}%$\qquad\qed$

\vskip2mm

\begin{lemma}[Approximation of gKdV for large time]
\label{gal} 
Assume $5/3<\alpha<2$. Let 
$\tilde{u}_{n}$ be given by (\ref{ap}). 
Then we have
\begin{equation}
\label{gall}
\lim_{T\to\infty}
\limsup_{n\to\infty}
\||\d_x|^{-1} \{ 
(\pt_{t}+\pt_{x}^{3})\tilde{u}_{n}
-\mu\pt_{x}(|\tilde{u}_{n}|^{2\alpha}
\tilde{u}_{n})\}\|_{N(|t|>\frac{T}{3\xi_{n}})}
=0.
\end{equation}
\end{lemma}

\begin{proof}
%\vskip2mm
%\noindent
%{\it Proof of Lemma \ref{gal}.} 
Since $\tilde{u}_{n}$ satisfies the Airy 
equation for $|t|>T/(3\xi_n)$, the linear part of \eqref{gall} vanishes.
Hence, we estimate
$\| |\tilde{u}_{n}|^{2\alpha}\tilde{u}_{n}\|_{N(|t|>\frac{T}{3\xi_{n}})}$.
We consider the case $t>T/(3\xi_{n})$ only 
since the case $t<-T/(3\xi_{n})$ being similar. 
Lemma \ref{lem:nlest} (i) implies
\begin{equation}\label{N4}
\|
|\tilde{u}_{n}|^{2\alpha}\tilde{u}_{n}\|_{N([\frac{T}{3\xi_{n}},\I))}
\le C
\|\tilde{u}_{n}
\|_{S([\frac{T}{3\xi_{n}},\I))}^{2\alpha}
\|\tilde{u}_{n}\|_{L([\frac{T}{3\xi_{n}},\I)}.
\end{equation}
By \eqref{qaz}, we have the bound $\|\tilde{u}_{n} \|_{S([\frac{T}{3\xi_{n}},\I))} \le C$.
On the other hand, 
Proposition \ref{prop:ho_inho} (\ref{ho}) %and Lemma \ref{osi} 
yields
\begin{align*}
% \label{N6}\\
\|\tilde{u}_{n}\|_{L([\frac{T}{3\xi_{n}},\I))}
={}&\|e^{-t\pt_{x}^{3}}[e^{-ix\xi_{n}}v_{n}(T)]\|_{L([0,\I))}\\
\le{}&
\|e^{-t\pt_{x}^{3}}[e^{-ix\xi_{n}}(v_{n}(T)-e^{-iT\pt_{x}^{2}}v_+)]
\|_{L([0,\I))}\\
&+\|e^{-t\pt_{x}^{3}}[e^{-ix\xi_{n}}e^{-iT\pt_{x}^{2}}v_+]\|_{L([0,\I))}
\\
\le{}&
C\|v_{n}(T)-e^{-iT\pt_{x}^{2}}v_+
\|_{\hat{L}_{x}^{\alpha}}
+\|e^{-t\pt_{x}^{3}}[e^{-ix\xi_{n}}e^{-iT\pt_{x}^{2}}v_+]\|_{L([0,\I))}.
\end{align*}
This implies
\begin{multline*}
	\limsup_{n\to\I} \|\tilde{u}_{n}\|_{L([\frac{T}{3\xi_{n}},\I))} 
	\le C\sup_n \|v_{n}(T)-e^{-iT\pt_{x}^{2}}v_+
\|_{\hat{L}_{x}^{\alpha}} \\
	+ \limsup_{n\to\I} \|e^{-t\pt_{x}^{3}}[e^{-ix\xi_{n}}e^{-iT\pt_{x}^{2}}v_+]\|_{L([0,\I))}
	=0
\end{multline*}
as $ T\to\infty$ together with \eqref{452} and Lemma \ref{osi}. 
Hence we obtain (\ref{gall}). 
\end{proof}%$\qed$

\vskip2mm

Next, we consider the approximation of 
gKdV in the middle interval $ |t| \le T/(3\xi_{n})$. 
A direct calculation yields
\begin{align}\label{l}
(\pt_t+\pt_x^3)\tilde{u}_{n}
={}&3
\mu C_{0}\xi_{n}Im[e^{-ix\xi_{n}-it\xi_{n}^{3}}
(|v_{n}|^{2\alpha}v_{n})(-3\xi_{n}t,x+3\xi_{n}^{2}t)]
\\
&+Re[e^{-ix\xi_{n}-it\xi_{n}^{3}}
(\pt_{x}^{3}v_{n})(-3\xi_{n}t,x+3\xi_{n}^{2}t)]
\nonumber
\end{align}
and
\begin{eqnarray*}
\mu\pt_x(|\tilde{u}_{n}|^{2\alpha}
\tilde{u}_{n})
&=&(2\alpha+1)\mu\xi_{n}
|Re[e^{-ix\xi_{n}-it\xi_{n}^{3}}
v_{n}(-3\xi_{n}t,x+3\xi_{n}^{2}t)]|^{2\alpha}\\
& &\qquad\times
Im[e^{-ix\xi_{n}-it\xi_{n}^{3}}
v_{n}(-3\xi_{n}t,x+3\xi_{n}^{2}t)]\\
& &+(2\alpha+1)\mu
|Re[e^{-ix\xi_{n}-it\xi_{n}^{3}}
v_{n}(-3\xi_{n}t,x+3\xi_{n}^{2}t)]|^{2\alpha}\\
& &\qquad\times
Re[e^{-ix\xi_{n}-it\xi_{n}^{3}}
(\pt_{x}v_{n})(-3\xi_{n}t,x+3\xi_{n}^{2}t)].
\end{eqnarray*}
Notice that 
\begin{eqnarray*}
\lefteqn{|Re[e^{-ix\xi_{n}-it\xi_{n}^{3}}
v_{n}(-3\xi_{n}t,x+3\xi_{n}^{2}t)]|^{2\alpha}
Im[e^{-ix\xi_{n}-it\xi_{n}^{3}}
v_{n}(-3\xi_{n}t,x+3\xi_{n}^{2}t)]}\\
&=&
(|v_{n}|^{2\alpha+1})(-3\xi_{n}t,x+3\xi_{n}^{2}t)
|Re[e^{-ix\xi_{n}-it\xi_{n}^{3}+iArg v_{n}}
]|^{2\alpha}\\
& &\qquad\qquad\times
Im[e^{-ix\xi_{n}-it\xi_{n}^{3}+iArg v_{n}}]\\
&=&
(|v_{n}|^{2\alpha+1})(-3\xi_{n}t,x+3\xi_{n}^{2}t)
\sum_{k=1}^{\infty}
C_{k}Im[e^{-ik(x\xi_{n}+t\xi_{n}^{3}-Arg v_{n})}]\\
&=&
C_{1}Im[e^{-ix\xi_{n}-it\xi_{n}^{3}}
(|v_{n}|^{2\alpha}v_{n})(-3\xi_{n}t,x+3\xi_{n}^{2}t)
]\\
& &+
\sum_{k=2}^{\infty}
C_{k}Im[e^{-ik(x\xi_{n}+t\xi_{n}^{3})}(|v_{n}|^{2\alpha+1-k}v_{n}^{k})
(-3\xi_{n}t,x+3\xi_{n}^{2}t)],
\end{eqnarray*}
where $Arg v_{n}=Arg v_{n}(-3\xi_{n}t,x+3\xi_{n}^{2}t)$ and 
$C_{k}$ is a $k$-th Fourier-sin coefficients for 
an odd function $f(\theta)=|\cos \theta |^{2\alpha}\sin \theta$, i.e.,
$C_k$ is the constant appearing in the expansion
\[
	|\cos \theta|^{2\alpha} \sin \theta = \sum_{k=1}^\I C_k \sin (k\theta).
\]
Or equivalently,
\begin{eqnarray*}
C_{k}=\frac{1}{\pi}
\int_{-\pi}^{\pi}f(\theta)\sin (k\theta) d\theta.
\end{eqnarray*}
An elementary computation shows that
\[
	C_1 = \frac{2 \Gamma(\alpha + \frac12) \Gamma(\frac32)}{\pi \Gamma (\alpha+2)} 
	= \frac{2 \Gamma(\alpha+ \frac{3}2) }{\sqrt{\pi} (2\alpha+1) \Gamma (\alpha+2)}  = \frac3{2\alpha+1} C_0.
\]
Then we have
\begin{equation}\label{g}
\begin{aligned}
&(\pt_t+\pt_x^3)\tilde{u}_{n}-\mu\pt_x(|\tilde{u}_{n}|^{2\alpha}
\tilde{u}_{n})
\\
={}&Re[e^{-ix\xi_{n}-it\xi_{n}^{3}}
(\pt_{x}^{3}v_{n})(-3\xi_{n}t,x+3\xi_{n}^{2}t)]
\\
&{} -(2\alpha+1)\mu
|Re[e^{-ix\xi_{n}-it\xi_{n}^{3}}
v_{n}(-3\xi_{n}t,x+3\xi_{n}^{2}t)]|^{2\alpha}
\\
& \qquad\times
Re[e^{-ix\xi_{n}-it\xi_{n}^{3}}
(\pt_{x}v_{n})(-3\xi_{n}t,x+3\xi_{n}^{2}t)]
\\
& -
(2\alpha+1)\mu\xi_{n}
\sum_{k=2}^{\infty}
C_{k}Im[e^{-ik(x\xi_{n}+t\xi_{n}^{3})}(|v_{n}|^{2\alpha+1-k}v_{n}^{k})
(-3\xi_{n}t,x+3\xi_{n}^{2}t)]
\\
=:{}&R_{n}^{1}+R_{n}^{2}+R_{n}^{3}.
\end{aligned}
\end{equation}
To evaluate the right hand side of (\ref{g}), 
we introduce a function
$e_{n}$ defined by
\begin{eqnarray}
\left\{
\begin{array}{l}
\displaystyle{
(\pt_{t}+\pt_{x}^{3})e_{n}
=R_{n}^{1}+R_{n}^{2}+R_{n}^{3},
}\\
\displaystyle{
e_{n}(0,x)=0.
}
\end{array}
\right.\label{er}
\end{eqnarray} 
Set $e_{n}=:e_{n,1}+e_{n,2}$, where 
\begin{eqnarray*}
\lefteqn{e_{n,1}=(2\alpha+1)\mu\xi_{n}^{-2}}\\
& &\times
\sum_{k=2}^{\infty}
C_{k}Im\left[e^{-ikx\xi_{n}}
\frac{e^{-ikt\xi_{n}^{3}}-e^{-ik^{3}t\xi_{n}^{3}}
}{i(k-k^{3})}
(|v_{n}|^{2\alpha+1-k}v_{n}^{k})(-3\xi_{n}t,x+3\xi_{n}^{2}t)
\right].
\end{eqnarray*}
A direct calculation yields 
\begin{eqnarray*}
(\pt_{t}+\pt_{x}^{3})e_{n,1}
&=&R_{n}^{3}+R_{n}^{4},\qquad\quad e_{n,1}(0,x)=0,\\
(\pt_{t}+\pt_{x}^{3})e_{n,2}
&=&R_{n}^{1}+R_{n}^{2}-R_{n}^{4},\qquad e_{n,2}(0,x)=0
\end{eqnarray*}
where $R_{n}^{4}$ is given by
\begin{eqnarray*}
R_{n}^{4}
=
\sum_{\ell=1}^{4}
\sum_{k=2}^{\infty}
Im\left[G_{n}^{\ell,k}
(-3\xi_{n}t,x+3\xi_{n}^{2}t)
(e^{-ikt\xi_{n}^{3}}-e^{-ik^{3}t\xi_{n}^{3}})e^{-ikx\xi_{n}}\right]
\end{eqnarray*}
with
\begin{eqnarray*}
G_{n}^{1,k}(t,x)
&=&3(2\alpha+1)\mu\frac{C_{k}}{ik}
\pt_{x}(|v_{n}|^{2\alpha+1-k}v_{n}^{k})(t,x),\\
G_{n}^{2,k}(t,x)
&=&-3(2\alpha+1)\mu\frac{C_{k}}{1-k^{2}}\xi_{n}^{-1}
\pt_{x}^{2}(|v_{n}|^{2\alpha+1-k}v_{n}^{k})(t,x),\\
G_{n}^{3,k}(t,x)
&=&-3(2\alpha+1)\mu\frac{C_{k}}{i(k-k^{3})}\xi_{n}^{-1}
\pt_{t}(|v_{n}|^{2\alpha+1-k}v_{n}^{k})(t,x),\\
G_{n}^{4,k}(t,x)
&=&(2\alpha+1)\mu\frac{C_{k}}{i(k-k^{3})}\xi_{n}^{-2}
\pt_{x}^{3}(|v_{n}|^{2\alpha+1-k}v_{n}^{k})(t,x).
\end{eqnarray*}
%where $C_{\alpha,\mu}^{1},\ldots,C_{\alpha,\mu}^{4}$ 
%are some constants depend only on $\alpha$ and $\mu$. 

\vskip2mm

\begin{lemma}[Error control] \label{err}
Fix $T>0$. 
Let $e_{n}$ be a solution to (\ref{er}). 
Then,
\begin{equation}\label{er1}
\lim_{n\to\infty} \(
\|e_{n}\|_{L_{t}^{\infty}\hat{L}_{x}^{\alpha}
([-\frac{T}{3\xi_n},\frac{T}{3\xi_{n}}])}
+\|e_{n}\|_{L([-\frac{T}{3\xi_n},\frac{T}{3\xi_{n}}])}
+\|e_{n}\|_{S([-\frac{T}{3\xi_n},\frac{T}{3\xi_{n}}])}\)=0.
\end{equation}
\end{lemma}

\begin{proof}
%\vskip2mm
%\noindent
%{\it Proof of Lemma \ref{err}.}
By the definition of $e_{n,1}$, we have
\begin{eqnarray*}
\lefteqn{\|e_{n,1}\|_{L_{t}^{\infty}\hat{L}_{x}^{\alpha}
([-\frac{T}{3\xi_n},\frac{T}{3\xi_{n}}])}
}\\
&\le&C\xi_{n}^{-2}
\sum_{k\ge2}
\frac{|C_{k}|}{k^{3}}
\|(|v_{n}|^{2\alpha+1-k}v_{n}^{k})
(3\xi_{n}t,x+3\xi_{n}^{2}t)
\|_{L_{t}^{\infty}\hat{L}_{x}^{\alpha}
([-\frac{T}{3\xi_n},\frac{T}{3\xi_{n}}])}.
\end{eqnarray*}
Since $L^{\alpha}\hookrightarrow\hat{L}^{\alpha}$ and 
$\dot{H}^{\frac12-\frac{1}{\alpha(2\alpha+1)}}\hookrightarrow
L^{\alpha(2\alpha+1)}$ for $1<\alpha\le 2$, 
we see from \eqref{eq:432} that
\begin{align*}
&\|(|v_{n}|^{2\alpha+1-k}v_{n}^{k})
(-3\xi_{n}t,x+3\xi_{n}^{2}t)
\|_{L_{t}^{\infty}\hat{L}_{x}^{\alpha}
([-\frac{T}{3\xi_n},\frac{T}{3\xi_{n}}])}\\
&{}=C
\|(|v_{n}|^{2\alpha+1-k}v_{n}^{k})(t,x)
\|_{L_{t}^{\infty}\hat{L}_{x}^{\alpha}
([-T,T])}\\
&{} \le C
\||v_{n}|^{2\alpha+1-k}v_{n}^{k}
\|_{L_{t}^{\infty}L_{x}^{\alpha}
([-T, T])}\\
&{}=C
\|v_{n}\|_{L_{t}^{\infty}L_{x}^{\alpha(2\alpha+1)}
([-T,T])}^{2\alpha+1}\\
&{}\le C
\|v_{n}\|_{L_{t}^{\infty}
{H}_{x}^{\frac12-\frac{1}{\alpha(2\alpha+1)}}}^{2\alpha+1}\\
&{}\le C\xi_n^{\frac12},
\end{align*}
which implies 
\begin{eqnarray}
\|e_{n,1}\|_{L_{t}^{\infty}\hat{L}_{x}^{\alpha}
([-\frac{T}{3\xi_n},\frac{T}{3\xi_{n}}])}
&\le&C\xi_{n}^{-\frac32}
\sum_{k=2}^{\infty}\frac{|C_{k}|}{k^{3}}
\le C\xi_{n}^{-\frac32}\to0
\label{fg1}
\end{eqnarray}
as $n\to\infty$.

Next we evaluate the $L$-norm of $e_{n,1}$. An interpolation shows
\begin{eqnarray*}
\|e_{n,1}\|_{L
([-\frac{T}{3\xi_n},\frac{T}{3\xi_{n}}])}
\le
\|e_{n,1}\|_{L^{3\alpha}_{t,x}
([-\frac{T}{3\xi_n},\frac{T}{3\xi_{n}}])}^{1-\frac1{3\alpha}}
\|\pt_{x}e_{n,1}\|_{L^{3\alpha}_{t,x}
([-\frac{T}{3\xi_n},\frac{T}{3\xi_{n}}])}^{\frac1{3\alpha}}. 
\end{eqnarray*}
By the definition of $e_{n,1}$, we see
\begin{eqnarray*}
\lefteqn{
\|\pt_{x}e_{n,1}\|_{L^{3\alpha}_{t,x}
([-\frac{T}{3\xi_n},\frac{T}{3\xi_{n}}])}}\\
&\le&
C\xi_{n}^{-1}
\sum_{k=2}^{\infty}
\frac{|C_{k}|}{k^{2}}
\|(|v_{n}|^{2\alpha+1-k}v_{n}^{k})
(-3\xi_{n}t,x+3\xi_{n}^{2}t)
\|_{L^{3\alpha}_{t,x} ([-\frac{T}{3\xi_n},\frac{T}{3\xi_{n}}])}\\
& &+
C\xi_{n}^{-2}
\sum_{k=2}^{\infty}
\frac{|C_{k}|}{k^{3}}
\|\pt_{x}(|v_{n}|^{2\alpha+1-k}v_{n}^{k})
(-3\xi_{n}t,x+3\xi_{n}^{2}t)
\|_{L^{3\alpha}_{t,x} ([-\frac{T}{3\xi_n},\frac{T}{3\xi_{n}}])}.
\end{eqnarray*}
Change of variables, the embedding $\dot{W}_{x}^{\frac12-\frac1{\alpha(2\alpha+1)},
\frac{6\alpha(2\alpha+1)}{6\alpha^2+3\alpha-4}}
\hookrightarrow L_x^{3\alpha(2\alpha+1)}$ and \eqref{eq:432} 
yield
\begin{eqnarray*}
\lefteqn{
\|\pt_{x}^{j}(|v_{n}|^{2\alpha+1-k}v_{n}^{k})
(-3\xi_{n}t,x+3\xi_{n}^{2}t)
\|_{L^{3\alpha}_{t,x} ([-\frac{T}{3\xi_n},\frac{T}{3\xi_{n}}])}}\\
&\le&
\|\pt_{x}^{j}(|v_{n}|^{2\alpha+1-k}v_{n}^{k})
(t,x)
\|_{ L^{3\alpha}_{t,x} ([-T,T])}\\
&\le&
\|v_{n}\|_{L_{t,x}^{3\alpha(2\alpha+1)}
([-T,T])}^{2\alpha}
\|\pt_{x}^{j}v_{n}
\|_{L_{t,x}^{3\alpha(2\alpha+1)}
([-T, T])}\\
&\le&
\||\pt_{x}|^{\frac12-\frac1{\alpha(2\alpha+1)}}v_{n}\|_{L_t^{3\alpha(2\alpha+1)}L_{x}^{\frac{6\alpha(2\alpha+1)}{6\alpha^2+3\alpha-4}}
([-T,T])}^{2\alpha}\\
& &\qquad\times
\||\pt_{x}|^{j+\frac12-\frac1{\alpha(2\alpha+1)}}v_{n}
\|_{L_{t,x}^{3\alpha(2\alpha+1)}L_{x}^{\frac{6\alpha(2\alpha+1)}{6\alpha^2+3\alpha-4}}
([-T, T])}\\
&\le&C\xi^{-\frac{1}{3\alpha}+\frac12+\frac{j}{4}}
\end{eqnarray*}
for $j=0,1$. Hence we obtain
\begin{eqnarray*}
\|\pt_{x}e_{n,1}\|_{ L^{3\alpha}_{t,x}
([-\frac{T}{3\xi_n},\frac{T}{3\xi_{n}}])}
&\le& C\xi_{n}^{-\frac12-\frac{1}{3\alpha}}
\sum_{k=2}^{\infty}
\frac{|C_{k}|}{k^{2}}
+C\xi_{n}^{-\frac54-\frac{1}{3\alpha}}
\sum_{k=2}^{\infty}
\frac{|C_{k}|}{k^{3}}\\
&\le& C\xi_{n}^{-\frac12-\frac{1}{3\alpha}}.
\end{eqnarray*}
In a similar way, 
\begin{eqnarray*}
\|e_{n,1}\|_{ L^{3\alpha}_{t,x}
([-\frac{T}{3\xi_n},\frac{T}{3\xi_{n}}])}
\le
C\xi_{n}^{-\frac32-\frac{1}{3\alpha}}.
\end{eqnarray*}
Hence, we have
\begin{eqnarray}
\|e_{n}\|_{L
([-\frac{T}{3\xi_n},\frac{T}{3\xi_{n}}])}
\le C\xi_{n}^{-\frac32}.\label{fg2}
\end{eqnarray}

Next we evaluate the $S([-\frac{T}{3\xi_n},\frac{T}{3\xi_{n}}])$-norm of $e_{n,1}$. 
We easily see 
\begin{align*}
	&\norm{ (|v_n|^{2\alpha+1-k} v_n^k)(-3\xi_n t , x+ 3\xi_n^2 t) 
	}_{L^{\frac{5\alpha}2}_x L^{5\alpha}_t ([-\frac{T}{3\xi_n},\frac{T}{3\xi_{n}}])}
	\\
	&{}=
	\norm{ v_n(-3\xi_n t , x+ 3\xi_n^2 t) 
	}_{L^{\frac{5\alpha(2\alpha+1)}2}_x L^{5\alpha(2\alpha+1)}_t 
	([-\frac{T}{3\xi_n},\frac{T}{3\xi_{n}}])}^{2\alpha+1} =: I^{2\alpha+1}.
\end{align*}
Let us estimate $I$. 
For simplicity, we put $\rho =5\alpha(2\alpha+1)$.
Change of variable and the Gagliardo-Nirenberg inequality yield 
\begin{eqnarray*}
\lefteqn{\norm{ v_n(-3\xi_n t , x+ 3\xi_n^2 t) }_{ L^{\rho}_t ([-\frac{T}{3\xi_n},\frac{T}{3\xi_{n}}])}}\\
&\le&C \xi_n^{-\frac1\rho} \norm{v_n(t, x- \xi_n t)}_{L^\rho_t (\R)}\\
&\le&C \xi_n^{-\frac1\rho}	C\norm{v_n(t, x- \xi_n t)}_{L^{\frac{\rho}{2}}_t (\R)}^{1-\frac1\rho}
	\norm{\d_t (v_n(t, x- \xi_n t))}_{L^{\frac{\rho}{2}}_t (\R)}^{\frac1\rho}.
\end{eqnarray*}
Hence,
\[
	I \le C\xi_n^{-\frac1\rho} \norm{v_n}_{L^{\frac{\rho}{2}}_{t,x} (\R^2)}^{1-\frac1\rho}
	\norm{\d_t v_n - \xi_n \d_x v_n }_{L^{\frac{\rho}{2}}_{t,x} (\R^2)}^{\frac1\rho}.
\]
Since $(\frac\rho2, \frac{2\rho}{\rho-8})$ is a Schr\"odinger admissible pair,
it follows from \eqref{eq:432} that
\[
	\norm{v_n}_{L^{\frac{\rho}{2}}_{t,x}}
	\le \norm{|\d_x|^{\frac12 -\frac6\rho} v_n }_{L^{\frac{\rho}{2}}_t L^{\frac{2\rho}{\rho-8}}_x}
	=O(\xi_n^{-\frac3{2\rho}+\frac1{4\alpha}}).
\]
Similar estimates hold for $\d_t v_n$ and $\d_x v_n$.
%\[
%	\norm{\d_t v_n}_{L^{\frac{\rho}{2}}_{t,x}}
%	\le \norm{|\d_t||\d_x|^{\frac12 -\frac6\rho} v_n }_{L^{\frac{\rho}{2}}_t L^{\frac{2\rho}{\rho-8}}_x}
%	=O(\xi_n^{\frac12-\frac3{2\rho}+\frac1{4\alpha}}),
%\]
%and
%\[
%	\norm{\d_x v_n}_{L^{\frac{\rho}{2}}_{t,x}}
%	\le \norm{|\d_x|^{\frac32 -\frac6\rho} v_n }_{L^{\frac{\rho}{2}}_t L^{\frac{2\rho}{\rho-8}}_x}
%	=O(\xi_n^{\frac14-\frac3{2\rho}+\frac1{4\alpha}}).
%\]
Combining above estimates, we conclude that
\[
	I % O(\xi_n^{-\frac1\rho+(1-\frac1\rho)(-\frac3{2\rho}+\frac1{4\alpha}) + \frac1\rho(1+ \frac14 -\frac3{2\rho}+\frac1{4\alpha})})
	=O(\xi_n^{\frac1{2(2\alpha+1)}}).
\]
Thus,
\[
	\norm{ e_{n,1} }_{S ([-\frac{T}{3\xi_n},\frac{T}{3\xi_{n}}])}
	= O(\xi_n^{-\frac32}).
\]
To evaluate $e_{n,2}$, we employ
% \begin{eqnarray*}
% N_3(I):=L_{x}^{p(N_3)}L_{t}^{q(N_3)}(I),
% \end{eqnarray*}
% where 
% \begin{eqnarray*}
% (p(N_{3}),q(N_{3}))=
% \left(\frac{5\alpha}{3\alpha+2},
% \frac{5\alpha}{4\alpha+1}\right).
% \end{eqnarray*}
the inhomogeneous estimate for Airy equation (\ref{inho}).
Since $(1,\alpha)$ is a conjugate-acceptable pair,
\begin{eqnarray}
\lefteqn{\|e_{n,2}\|_{L_{t}^{\infty}\hat{L}_{x}^{\alpha}
([-\frac{T}{3\xi_n},\frac{T}{3\xi_{n}}])}
+\|
e_{n,2}\|_{L
([-\frac{T}{3\xi_n},\frac{T}{3\xi_{n}}])}
+\|e_{n,2}\|_{S([-\frac{T}{3\xi_{n}},\frac{T}{3\xi_{n}}])}
}\label{rr3}\\
&\le&
\|R_{n}^{1}\|_{L_{t}^1\hat{L}_{x}^{\alpha}
([-\frac{T}{3\xi_n},\frac{T}{3\xi_{n}}])}
+
\|
R_{n}^{2}\|_{L^{\widetilde{p}(1,\alpha)}_x L_t^{\widetilde{q}(1,\alpha)}
([-\frac{T}{3\xi_n},\frac{T}{3\xi_{n}}])}\nonumber\\
& &
+
\|
R_{n}^{4}\|_{L^{\widetilde{p}(1,\alpha)}_x L_t^{\widetilde{q}(1,\alpha)}
([-\frac{T}{3\xi_n},\frac{T}{3\xi_{n}}])}.\nonumber
\end{eqnarray}
By (\ref{43}), we have
\begin{equation}\label{r3}
\|R_{n}^{1}\|_{L_{t}^1\hat{L}_{x}^{\alpha}
([-\frac{T}{3\xi_n},\frac{T}{3\xi_{n}}])}
\le C\xi_{n}^{-\frac14}T\|
v_{n}\|_{L_{t}^{\infty}\hat{L}_{x}^{\alpha}([-T,T])}\to0
\end{equation}
as $n\to\infty$ and
\begin{eqnarray}
\lefteqn{\|R_{n}^{2}
\|_{ L^{\widetilde{p}(1,\alpha)}_x L_t^{\widetilde{q}(1,\alpha)}
([-\frac{T}{3\xi_n},\frac{T}{3\xi_{n}}])}}\label{r4}\\
&\le& 
C\xi^{-\frac{1}{\widetilde{q}(1,\alpha)}}
\|v_{n}\|_{L_{x}^{p(0,\alpha)}L_{t}^{q(0,\alpha)}
([-T, T])}^{2\alpha}
\|\pt_{x}v_{n}
\|_{L_{x}^{p(1,\alpha)}L_{t}^{q(1,\alpha)}
([-T,T])}
\nonumber\\
&\le& 
% C\xi_{n}^{-\frac{1}{\widetilde{q}(1,\alpha)}-\frac{\alpha}{2}
% t(S)+\frac14(1-t(L_{3}))}\\
% &{}=
C\xi_{n}^{-\frac{15}{16 \widetilde{q}(1,\alpha)}}\to0
\nonumber
\end{eqnarray}
as $n\to\infty$.
In a similar way
\begin{eqnarray}
\|R_{n}^{4}
\|_{ L^{\widetilde{p}(1,\alpha)}_x L_t^{\widetilde{q}(1,\alpha)}
([-\frac{T}{3\xi_n},\frac{T}{3\xi_{n}}])}
\le C\xi_{n}^{-\frac{15}{16 \widetilde{q}(1,\alpha)}}\to0
\label{r5}
\end{eqnarray}
as $n\to\infty$. Combining (\ref{rr3}), (\ref{r3}), 
(\ref{r4}) and (\ref{r5}), we have
\begin{eqnarray}
\|e_{n}\|_{S([-\frac{T}{3\xi_n},\frac{T}{3\xi_{n}}])}
\to0\label{fg3}
\end{eqnarray}
as $n\to\infty$. From (\ref{fg1}), (\ref{fg2}) and (\ref{fg3}), 
we have (\ref{er1}). 
\end{proof}%$\qed$

\begin{lemma}(Approximation of gKdV for middle interval) 
\label{gas} Fix $T\in\rre$. 
Let $\tilde{u}_{n}$ and $e_{n}$ be given by 
(\ref{ap}) and (\ref{er}). Then we have
\begin{eqnarray}
\lefteqn{\lim_{n\to\infty}
\||\d_{x}|^{-1}
[(\pt_{t}+\pt_{x}^{3})(\tilde{u}_{n}-e_{n})}\label{galli}\\
& &\qquad\qquad
-\mu\pt_{x}\{|\tilde{u}_{n}-e_{n}|^{2\alpha}
(\tilde{u}_{n}-e_{n})\}]\|_{N([-\frac{T}{3\xi_n},\frac{T}{3\xi_{n}}])}
=0.
\nonumber
\end{eqnarray}
\end{lemma}

\begin{proof}
%\vskip2mm
%\noindent
%{\it Proof of Lemma \ref{gas}.} 
In the proof we omit $([-\frac{T}{3\xi_n},\frac{T}{3\xi_{n}}])$, for simplicity.
We first note
\begin{eqnarray*}
\lefteqn{(\pt_{t}+\pt_{x}^{3})(\tilde{u}_{n}-e_{n})
-\mu\pt_{x}\{|\tilde{u}_{n}-e_{n}|^{2\alpha}
(\tilde{u}_{n}-e_{n})\}}\\
&=&
\mu\pt_{x}\{|\tilde{u}_{n}|^{2\alpha}\tilde{u}_{n}
-|\tilde{u}_{n}-e_{n}|^{2\alpha}(\tilde{u}_{n}-e_{n})\}.
\end{eqnarray*}
Lemma \ref{lem:nlest} implies
\begin{eqnarray*}
\lefteqn{\||\tilde{u}_{n}|^{2\alpha}\tilde{u}_{n}
-|\tilde{u}_{n}-e_{n}|^{2\alpha}
(\tilde{u}_{n}-e_{n})\|_{N}}\\
&\le&
% (\|\tilde{u}_{n}\|_{S}
% ^{\frac{3\alpha}2}+
% \|e_{n}\|_{S}
% ^{\frac{3\alpha}2})\\
% &\times
C(\|\tilde{u}_{n}\|_{L}+\|e_{n}\|_{L})
(\|\tilde{u}_{n}\|_{S}+\|e_{n}\|_{S})^{2\alpha-1}
\|e_{n}\|_{S}\\
& &+
C(\|\tilde{u}_{n}\|_{S}+\|e_{n}\|_{S})^{2\alpha}
\|e_{n}\|_{L}.
\end{eqnarray*}
By Lemma \ref{err}, 
letting $n\to\infty$ in the above inequalities, we obtain 
(\ref{galli}). 
\end{proof}%$\qed$

\vskip2mm

\begin{lemma}[Initial condition] \label{ic} 
Take a parameter $T$ so that $T>T_{0}$ if $|T_{0}|<\infty$ and
arbitrarily positive if $T_{0}=\pm\infty$. 
Let $u_{n}(t_{n})$ and $\tilde{u}_{n}(t)$ be 
given by (\ref{tttg}
%\ref{r39}
) and (\ref{ap}), respectively. 
Then we have
\begin{eqnarray}
\lim_{n\to\infty}
\|u_{n}(t_{n})-\tilde{u}_{n}(t_{n})\|_{
\hat{L}_{x}^{\alpha}}=0.
\label{ic1}
\end{eqnarray}
\end{lemma}

\begin{proof}
%\vskip2mm
%\noindent
%{\it Proof of Lemma \ref{ic}.} 
We first consider the case $|T_{0}|<\infty$. 
Notice that in this case we necessarily have 
$t_{n}\to0$ as $n\to\infty$. 
Since $|t_{n}|\le T/(3\xi_{n})$ 
for $n$ sufficiently large, we have 
\begin{eqnarray*}
\lefteqn{
\|u_{n}(t_{n})-\tilde{u}_{n}(t_{n})\|_{
\hat{L}_{x}^{\alpha}}}\\
&\le&
\|e^{-t_{n}\pt_{x}^{3}}[e^{-ix\xi_{n}}\phi(x)]
-e^{-ix\xi_{n}-it_{n}\xi_{n}^{3}}
v_{n}(-3\xi_{n}t_{n},x+3\xi_{n}^{2}t_{n})\|_{
\hat{L}_{x}^{\alpha}}\\
&=&
\|e^{it_{n}(\xi-\xi_{n})^{3}}\hat{\phi}(\xi)
-e^{-it_{n}\xi_{n}^{3}+3it_{n}\xi_{n}^{2}\xi}
\hat{v}_{n}(-3\xi_{n}t_{n},\xi)\|_{
L_{\xi}^{\alpha'}}\\
&=&
\|e^{it_{n}\xi^{3}-3it_{n}\xi_{n}\xi^{2}}\hat{\phi}(\xi)
-
\hat{v}_{n}(-3\xi_{n}t_{n},\xi)\|_{
L_{\xi}^{\alpha'}}.
\end{eqnarray*}
Since $t_{n}\to0$ and $-3t_{n}\xi_{n}\to T_{0}$, 
we have 
\begin{eqnarray*}
\lefteqn{
\|e^{it_{n}\xi^{3}-3it_{n}\xi_{n}\xi^{2}}\hat{\phi}(\xi)
-
\hat{v}_{n}(-3\xi_{n}t_{n},\xi)\|_{
L_{\xi}^{\alpha'}}
}\\
&\le&
\|(e^{it_{n}\xi^{3}}-1)e^{-3it_{n}\xi_{n}\xi^{2}}\hat{\phi}(\xi)\|_{
L_{\xi}^{\alpha'}}
+\|e^{-3it_{n}\xi_{n}\xi^{2}}\hat{\phi}(\xi)
-\hat{v}(-3\xi_{n}t_{n},\xi)\|_{
L_{\xi}^{\alpha'}}\\
& &+
\|\hat{v}(-3\xi_{n}t_{n},\xi)-\hat{v}_{n}(-3\xi_{n}t_{n},\xi)\|_{
L_{\xi}^{\alpha'}}\\
&\to&0
\end{eqnarray*}
as $n\to\infty$, where we used the dominated convergence 
theorem, (\ref{ic11}) and (\ref{con}). This proves 
(\ref{ic1}) for $|T_{0}|<\infty$. 

Next we consider the case $T_{0}=\pm\infty$. 
We treat the case $T_{0}=\infty$ only since the 
case $T_{0}=-\infty$ being similar. 
For $n$ sufficiently large, $-3t_{n}\xi_{n}>T$. 
Hence
\begin{eqnarray*}
\lefteqn{
\|u_{n}(t_{n})-\tilde{u}_{n}(t_{n})\|_{
\hat{L}_{x}^{\alpha}}}\\
&\le&
\|e^{-(t_{n}+\frac{T}{3\xi_{n}})\pt_{x}^{3}}
[e^{\frac{T}{3\xi_{n}}\pt_{x}^{3}}
[e^{-ix\xi_{n}}\phi(x)]
-e^{-ix\xi_{n}+\frac{i}{3}T\xi_{n}^{2}}
v_{n}(T,x-\xi_{n}T)]\|_{
\hat{L}_{x}^{\alpha}}\\
&=&
\|e^{\frac{T}{3\xi_{n}}\pt_{x}^{3}}
[e^{-ix\xi_{n}}\phi(x)]
-e^{-ix\xi_{n}+\frac{i}{3}T\xi_{n}^{2}}
v_{n}(T,x-\xi_{n}T)\|_{
\hat{L}_{x}^{\alpha}}\\
&=&
\|e^{-i\frac{T}{3\xi_{n}}(\xi-\xi_{n})^{3}}
\hat{\phi}(\xi)
-e^{\frac{i}{3}T\xi_{n}^{2}-iT\xi_{n}\xi}
\hat{v}_{n}(T,\xi)
\|_{
L_{\xi}^{\alpha'}}\\
&=&
\|e^{-i\frac{T}{3\xi_{n}}\xi^{3}+iT\xi^{2}}
\hat{\phi}(\xi)
-\hat{v}_{n}(T,\xi)
\|_{
L_{\xi}^{\alpha'}}.
\end{eqnarray*}
Since $\xi_{n}\to\infty$, the dominated 
convergence theorem, (\ref{ic11}) and (\ref{con}) 
yield
\begin{eqnarray*}
\lefteqn{
\|e^{-i\frac{T}{3\xi_{n}}\xi^{3}+iT\xi^{2}}
\hat{\phi}(\xi)
-\hat{v}_{n}(T,\xi)
\|_{
L_{\xi}^{\alpha'}}
}\\
&\le&\|(e^{-i\frac{T}{3\xi_{n}}\xi^{3}}-1)
e^{iT\xi^{2}}\hat{\phi}(\xi)\|_{
L_{\xi}^{\alpha'}}
+\|e^{iT\xi^{2}}\hat{\phi}(\xi)-\hat{v}(T,\xi)\|_{
L_{\xi}^{\alpha'}}\\
& &+\|\hat{v}(T,\xi)-\hat{v}_{n}(T,\xi)
\|_{
L_{\xi}^{\alpha'}}\\
&\to&0
\end{eqnarray*}
as $n\to\infty$, which proves (\ref{ic1}) for the case 
$T_{0}=\infty$. This completes the proof of 
Lemma \ref{ic}. 
\end{proof}%$\qed$

\vskip2mm
\noindent
{\it Proof of Theorem \ref{thm:ENK}.}
By Lemma \ref{qa}, there exist two positive constants 
$A$ and $M$ which are independent of $T$ and $n$ 
such that 
\begin{eqnarray*}
\|\tilde{u}_{n}\|_{L_{t}^{\infty}(\rre;
\hat{L}_{x}^{\alpha})}&\le&A,\\
\|\tilde{u}_{n}\|_{S(\rre)}+
\|\tilde{u}_{n}\|_{L(\rre)}&\le&M.
\end{eqnarray*}
For the above $M$, let $\varepsilon_{1}=\varepsilon_{1}(M)$ be given by 
Lemma \ref{lsy} and let $C$ be a constant appearing 
in Lemma \ref{lsy}. Then. Lemma \ref{gal} yields that 
for any $\varepsilon$ satisfying $0<\varepsilon<C\varepsilon_{1}$, 
there exists a positive constant $T_{\varepsilon}$ 
such that if $T\ge T_{\varepsilon}$, then 
\begin{eqnarray}
\label{y10}\\
\lim_{n\to\infty}
\||\d_{x}|^{-1}\{(\pt_{t}+\pt_{x}^{3})\tilde{u}_{n}
-\mu\pt_{x}(|\tilde{u}_{n}|^{2\alpha}\tilde{u}_{n})\}
\|_{ N (|t|>\frac{T}{3\xi_{n}} )}<\frac{\varepsilon}{2}.
\nonumber
\end{eqnarray}
We now choose
\begin{eqnarray*}
T:=\left\{
\begin{array}{l}
\displaystyle{\max\{T_{\varepsilon},2|T_{0}|\}
\quad\ if\ T_{0}=\pm\infty,}\\
\displaystyle{T_{\varepsilon}\qquad\qquad\quad\quad\  
\ if\ |T_{0}|<\infty.}
\end{array}
\right.
\end{eqnarray*}
We first apply the long time stability for gKdV 
in the time interval $\{|t| \le T/(3\xi_{n})\}$. 
Lemmas \ref{gas} and \ref{ic} lead that 
there exits a nonnegative integer $N_{1}=N_{1}(\varepsilon,
T_{\varepsilon})$ such that if $n\ge N_{1}$, then 
$|t_{n}|\le T/(3\xi_{n})$ and 
\begin{eqnarray*}
\lefteqn{\|u_{n}(t_{n})-\tilde{u}_{n}(t_{n})
\|_{\hat{L}_{x}^{\alpha}}}\\
& &+\||\d_{x}|^{-1} \{(\pt_{t}+\pt_{x}^{3})\tilde{u}_{n}
-\mu\pt_{x}(|\tilde{u}_{n}|^{2\alpha}\tilde{u}_{n})\}
\|_{N (|t|\le\frac{T}{3\xi_{n}})}\le\frac{\varepsilon}{C}.
\end{eqnarray*}
Hence, by Proposition \ref{lsy}, there exists a unique 
solution $u\in C(I;\hat{L}_{x}^{\alpha})$ to (\ref{gKdV}) 
satisfying 
\begin{equation}
\|u_{n}-\tilde{u}_{n}
\|_{L_{t}^{\infty}(I;\hat{L}_{x}^{\alpha})}
+\|u_{n}-\tilde{u}_{n}\|_{S(I)}+
\|u_{n}-\tilde{u}_{n}\|_{L(I)}
\le \frac{\varepsilon}{2}, \label{y11}
\end{equation}
where $I=[-\frac{T}{3\xi_n},\frac{T}{3\xi_{n}}]$.
Especially, we have
\begin{equation}
\norm{u_{n}\(\pm\frac{T}{3\xi_{n}}\)
-\tilde{u}_{n}\(\pm\frac{T}{3\xi_{n}}\)
}_{\hat{L}_{x}^{\alpha}}
\le \frac{\varepsilon}{2}.
\label{y12}
\end{equation}
Next we apply the long time stability for gKdV 
in the time intervals $t\ge T/(3\xi_{n})$ and $t\le-T/(3\xi_{n})$, respectively.
Combining (\ref{y10}), (\ref{y11}), (\ref{y12}) 
and Lemma \ref{lsy}, we find that there exists a unique 
global solution $u\in C(\R;\hat{L}_{x}^{\alpha})$ to (\ref{gKdV}) 
satisfying 
\begin{eqnarray*}
\|u_{n}-\tilde{u}_{n}
\|_{L_{t}^{\infty}(\R;\hat{L}_{x}^{\alpha})}
+\|u_{n}-\tilde{u}_{n}\|_{S}+
\|u_{n}-\tilde{u}_{n}\|_{L}
\le C\varepsilon.
\end{eqnarray*} 
Combining the above inequality and Lemma \ref{qa}
we have Theorem \ref{thm:ENK}. 
%$\qquad\qed$

%%%%%%%%%%%%%%%%%%%%%%%%%%%%%%%%%%%%%%%%%%%%%%%%%%%%
%
%     Appendix
%
%%%%%%%%%%%%%%%%%%%%%%%%%%%%%%%%%%%%%%%%%%%%%%%%%%%%

\appendix

\section{On generalized Morrey spaces}

In this appendix, we give the following interpolation type inequality
for the generalized Morrey spaces.

\begin{proposition}\label{prop:gm_interpolation}
Suppose that $0<q< p<r < \I $.  
If $s$ satisfies
\[
	\frac{1}{s}\times\(1-\frac{p}r\) + \frac{1}p\times\frac{p}r < \frac1q
\]
then, for any $f\in L^q(\R)$, we have
\[
	\norm{f}_{{M}^p_{q,r}}
	\le C \norm{f}_{{M}^p_{s,\I} }^{1-\frac{p}r} \norm{f}_{{M}^p_{p,\I}}^{\frac{p}r}.
\]
In particular, $L^p \hookrightarrow M^{p}_{q,r}$.
\end{proposition}

\begin{proof}
%\vskip2mm
%\noindent
%{\it Proof of Proposition \ref{prop:gm_interpolation}.}
Set
\[
	{f}_{n,I}(x) := {f}(x)  {\bf 1}_{I\cap \{ 2^{n}  \le |I|^{\frac1p} |f(x)| \le  2^{n+1}\}}(x)
\]
for $I \in \mathcal{D}$ and $n\in \Z$. 
Let $\theta=1-p/r$. 
By the H\"older inequality in $x$, 
%\begin{equation}%\label{eq:pf_ST2_4}
\begin{eqnarray}
\label{aa1}\\
	 \int |{f}_{n,I}(x)|^{q} \,dx
	&=& \int |{f}_{n,I}(x)|^{\theta q} |{f}_{n,I}(x)|^{(1-\theta)q} \,dx\nonumber\\
	&\le& \(\int |{f}_{n,I}(x)|^{\frac{\theta q p}{p-(1-\theta)q}} \,dx\)^{1-\frac{(1-\theta)q}{p}}\(\int |{f}_{n,I}(x)|^{p} \,dx\)^{\frac{(1-\theta)q}{p}}.\nonumber
\end{eqnarray}
%\end{equation}
By definition of ${f}_{n,I}$, we have 
\begin{align*}
	\(\int |{f}_{n,I}(x)|^{\frac{\theta q p}{p-(1-\theta)q}} \,dx\)^{1-\frac{(1-\theta)q}{p}}
	&{}\le  C2^{\theta qn} |I|^{-\frac{\theta q}{p}}  \(\int_{I\cap \{ |f| \ge  2^{n} |I|^{-\frac1p} \}} \,dx \)^{1-\frac{(1-\theta)q}{p}}.
\end{align*}
It is obvious that
\[
	\(\int_{I\cap \{ |f| \ge  2^{n} |I|^{-\frac1p} \}} \,dx \)^{1-\frac{(1-\theta)q}{p}}
	\le |I|^{1-\frac{(1-\theta)q}{p}}.
\]
One sees from Chebyshev's inequality that
\begin{align*}
	&\(\int_{I\cap \{ |f| \ge  2^{n} |I|^{-\frac1p} \}} \,dx \)^{1-\frac{(1-\theta)q}{p}}\\
	&{}\le \(\frac{\int_I |{f}(x)|^s \,dx}{2^{ns} |I|^{-\frac{s}p}}\)^{1-\frac{(1-\theta)q}{p}} \\
	&{}\le 2^{-(1-\frac{(1-\theta)q}{p}) sn} |I|^{1-\frac{(1-\theta)q}{p}} \(\sup_{I \in \mathcal{D}}
	 |I|^{\frac1p-\frac1s} \norm{f}_{L^s(I)}\)^{(1-\frac{(1-\theta)q}{p})s}.
\end{align*}
Namely,
\begin{eqnarray*}
	\lefteqn{\(\int |{f}_{n,I}(x)|^{\frac{\theta p q}{p-(1-\theta)q}} \,dx\)^{1-\frac{(1-\theta)q}{p}} }\\
	&\le& C |I|^{1-\frac{q}p}  
	\min \( 2^{\theta q n} , 2^{\theta q n -(1-\frac{(1-\theta)q}{p}) sn}
% 	\(\sup_{I \in \mathcal{D}}|I|^{\frac1q-\frac1s} \norm{f}_{L^s(I)}\)
	\norm{f}_{M^{p}_{s,\I}}^{(1-\frac{(1-\theta)q}{p})s} \) \\
	 &=& C |I|^{1-\frac{q}p} \norm{f}_{M^{p}_{s,\I}}^{\theta q}
	\min \( 2^{\theta q (n-n_0)} , 2^{(\theta -\frac{s}q+\frac{(1-\theta)s}{p} )q(n-n_0)} \),
\end{eqnarray*}
where, we chose $n_0\in \R$ by $2^{n_0} = \norm{f}_{M^{q}_{s,\I}}$.
Since 
\[
	\theta -\frac{s}q+\frac{(1-\theta)s}{p} <0< \theta 
\]
by assumption, there exists $\delta=\delta(p,q,s,\theta)>0$ such that
\begin{eqnarray}
	\(\int |{f}_{n,I}(x)|^{\frac{\theta p q}{q-(1-\theta)p}} \,dx\)^{1-\frac{(1-\theta)p}{q}}
	\le C 2^{-\delta |n-n_0|} |I|^{1-\frac{q}p} \norm{f}_{M^{p}_{s,\I}}^{\theta q}
	\label{aa2}
\end{eqnarray}
for all $n\in \Z$ and $I \in \mathcal{D}$.

Note that
\[
	\int_I |f(x)|^q dx = \sum_{n\in\Z} \int_{\R} |f_{n,I}(x)|^q dx
\]
for any $I \in \mathcal{D}$ since $q<\I$ by assumption.
The inequalities (\ref{aa1}) and (\ref{aa2}) yield 
\begin{align*}
	&\norm{f}_{M^p_{q,r}}^r \\
	&{} = \sum_{I \in \mathcal{D}} \( \sum_{n\in\Z} |I|^{\frac{q}{p}-1} 
	\norm{f_{n,I}}_{L^q(\R)}^q\)^{{r}/q} \\
	&{} \le C\sum_{I \in \mathcal{D}} \( \sum_{n\in\Z} 
	2^{-\delta |n-n_0|} \norm{f}_{M^{p}_{s,\I}}^{\theta q} \(\int |{f}_{n,I}(x)|^{p} \,dx\)^{\frac{(1-\theta)q}{p}} \)^{r/q}\\
	&{} = C \norm{f}_{M^{p}_{s,\I}}^{\theta r} \sum_{I \in \mathcal{D}} \( \sum_{n\in\Z} 
	\( 2^{-\delta' |n-n_0|}  \int |{f}_{n,I}(x)|^{p} \,dx\)^{{q}/{r}} \)^{r/q} \\
	&{} \le C_{\delta'} \norm{f}_{M^{p}_{s,\I}}^{\theta r} \sum_{I \in \mathcal{D}}  \sum_{n\in\Z} 
	 2^{-\frac{\delta'}{2} |n-n_0|}  \int |{f}_{n,I}(x)|^{p} \,dx ,
\end{align*}
where we have used the H\"older inequality in $n$ to yield the last line.
Thus,
\[
	\norm{f}_{M^p_{q,r}}^r
	\le 
	C \norm{f}_{M^{p}_{s,\I}}^{\theta r} \sup_{n\in\Z}\sum_{I \in \mathcal{D}} \int |{f}_{n,I}(x)|^{p} \,dx
\]

Finally, for any fixed $n$, we have
\begin{align*}
	\sum_{I \in \mathcal{D}} \int |{f}_{n,I}(x)|^{p} \,dx
	&{}= \sum_{j\in \Z} \sum_{I \in \mathcal{D}_j}  \int |{f}_{n,I}(x)|^{p} \,dx \\
	&{}= \sum_{j\in \Z} \int_{\{ 2^{n}  \le 2^{-\frac{j}p} |f(x)| \le  2^{n+1}\}} |f(x)|^{p} \,dx,
\end{align*}
where we have used the fact that elements of $\mathcal{D}_j$ are mutually disjoint
and $\cup_{I\in{\mathcal{D}_j}} I = \R$.
Since $\{ 2^{n}  \le 2^{-\frac{j}p} |f(x)| \le  2^{n+1}\} $
and $\{ 2^{n}  \le 2^{-\frac{j'}p} |f(x)| \le  2^{n+1}\}$
are disjoint as long as $|j-j'| > p$, we have
\[
	\sum_{j\in \Z}  \int_{\{ 2^{n}  \le 2^{-\frac{j}p} |f(x)| \le  2^{n+1}\}} |f(x)|^{p} \,dx
	\le (p+1) \norm{f}_{L^p(\R)}^{p} =(p+1) \norm{f}_{L^p(\R)}^{(1-\theta)r},
\]
which completes the proof. 
\end{proof}%$\qed$

\section{On refined Stein-Tomas estimate}\label{sec:rST}

In this subsection, we prove the first inequality of 
the refined Stein-Tomas estimate \eqref{eq:ST3}.
\begin{theorem}\label{thm:b1}
Let $4/3\le p<\I$.
Then, there exist a constant $C=C(p)$ such that
\begin{equation}\label{eq:ST2}
\begin{aligned}
	\norm{ |\d_x|^{ 1/3p } e^{-t\d_x^3} f }_{ L^{3p}_{t,x} }
	\le{}& C
	\hMn{f}{p}{\frac{3p}{2}}{2(\frac{3p}2)'}.
\end{aligned}
\end{equation}
for any $f \in \hM{p}{\frac{3p}2}{2(\frac{3p}{2})'}$.
\end{theorem}

\begin{proof}
%\vskip2mm
%\noindent
%{\it Proof of Proposition \ref{thm:b1}.}
We argue as in Shao \cite{Shao}.
The square of the left hand side of \eqref{eq:ST2} is equal to
\[
	\norm{\iint e^{ix(\xi-\eta) + it(\xi^3-\eta^3)} |\xi\eta|^{1/3p} \hat{f}(\xi)\overline{\hat{f}(\eta)}\,d\xi d\eta }_{L^{3p/2}_{t,x}}.
\]
Changing variables by $a=\xi-\eta$ and $b=\xi^3-\eta^3$, we have
\[
\norm{\iint e^{ixa + itb} |\xi\eta|^{\frac{1}{3p}} \hat{f}(\xi)\overline{\hat{f}(\eta)}\frac1{3|\xi^2-\eta^2|} \,dadb }_{L^{3p/2}_{t,x}}.
\]
Since $3p/2\ge2$,
we use Hausdorff-Young inequality to deduce that this is bounded by
\begin{eqnarray*}
	\lefteqn{C \norm{|\xi\eta|^{1/3p} \hat{f}(\xi)\overline{\hat{f}(\eta)}|\xi^2-\eta^2|^{-1} }_{L^{(3p/2)'}_{a,b}}}\\
	&=& C \left\{ \iint_{\R^2} \frac{|\xi\eta|^{\frac1{3p-2}} |\hat{f}(\xi)|^{(\frac{3p}{2})'}|\hat{f}(\eta)|^{(\frac{3p}{2})'}}
	{|\xi-\eta|^{\frac{2}{3p-2}} |\xi + \eta|^{\frac{2}{3p-2}}}
	\,d\xi d\eta \right\}^{1-\frac{2}{3p}}.
\end{eqnarray*}
Thus, we have
\begin{equation*}%\label{eq:pf_ST2_1}
	\norm{ |\d_x|^{ 1/3p } e^{-t\d_x^3} f }_{ L^{3p}_{t,x} }^{2(\frac{3p}{2})'}
	\le C \iint_{\R^2}  \frac{ |\xi\eta|^{\frac{1}{3p-2}} |\hat{f}(\xi)|^{(\frac{3p}{2})'}|\hat{f}(\eta)|^{(\frac{3p}{2})'}}
	{|\xi+\eta|^{\frac{2}{3p-2}}|\xi-\eta|^{\frac{2}{3p-2}}}
	\,d\xi d\eta   .
\end{equation*}

% Now, we use a Whitney-type decomposition.
We now introduce a Whitney-type decomposition.
For an interval $I \in \mathcal{D}_j$, there exists a unique interval $J \in \mathcal{D}_{j-1}$
such that $I\subset J$. We call $J$ as a \emph{parent} of $I$.
For two intervals $I,I' \in \mathcal{D}$, we introduce a binary relation $\sim_{\mathcal{W}}$
so that
$I \sim_{\mathcal{W}} I'$ holds if the following three conditions are satisfied;
(i) $I$ and $I'$ belong to same $\mathcal{D}_{j}$, that is, $|I| = |I'|$;
(ii) $I$ is not neighboring neither $I'$ nor $-I'$; 
and (iii) a parent of $I$ is neighboring either a parent of $I'$ or a parent of $-I'$.
Set $\mathcal{W} := \{ (I,I') \in \mathcal{D}\times \mathcal{D} \ |\ I \sim_{\mathcal{W}} I' \}$.
Notice that if $I\sim_{\mathcal{W}} I'$ then $|I| \le \min (\dist (I,I'),\dist (I,-I')) \le 2 |I|$ and 
 that
for any $I\in \mathcal{D}$, $\# \{I' \in \mathcal{D} \ |\ I\sim_\mathcal{W} I'\} 
= 2,4$ or $6$.
Then, we have the following Whitney-type decomposition of $\R\times \R$;
\[
	\sum_{(I,I') \in{\mathcal{W}} } {\bf 1}_{I} (\xi) {\bf 1}_{I'} (\eta) = 1,\qquad
	(\xi,\eta) \in \R^2 \setminus \{(\xi,\pm\xi) \ |\ \xi \in \R\}.
\]

Let $\mathcal{W}$ be as above.
Since $ |\xi\eta| \le \max (|\xi+\eta|^2,|\xi-\eta|^2)$ for any $(\xi,\eta) \in \R^2$,
one sees that
\[
	\frac{|\xi \eta|}{|\xi+\eta|^2|\xi-\eta|^2} \le |I|^{-2}
\]
for any $(\xi,\eta) \in I \times I'$ with $(I,I') \in \mathcal{W}$.
We hence obtain
\begin{align*}
	&\iint_{\R^2}  \frac{ |\xi\eta|^{\frac{1}{3p-2}} |\hat{f}(\xi)|^{(\frac{3p}{2})'}|\hat{f}(\eta)|^{(\frac{3p}{2})'}}
	{|\xi+\eta|^{\frac{2}{3p-2}}|\xi-\eta|^{\frac{2}{3p-2}}}
	\,d\xi d\eta \\
	&{}= \iint_{\R^2} \sum_{(I,I') \in {\mathcal{W}} } \(\frac{|\xi \eta|  }
	{|\xi+\eta|^2 |\xi-\eta|^2}\)^{\frac1{3p-2}} |\hat{f}(\xi)|^{(\frac{3p}{2})'} |\hat{f}(\eta)|^{(\frac{3p}{2})'} {\bf 1}_I(\xi)  {\bf 1}_{I'} (\eta)
	\,d\xi d\eta \\
	&{}\le  \sum_{ I \in \mathcal{D} }\sum_{I';\, I\sim_{\mathcal{W}} I'}
	|I|^{-\frac{2}{3p-2}}
	\int_I |\hat{f}(\xi)|^{(\frac{3p}{2})'} \,d\xi \int_{I'} |\hat{f}(\eta)|^{(\frac{3p}{2})'} \,d\eta .
\end{align*}
We choose a slightly larger interval 
containing $I$ and either $I'$or $-I'$ but still of length
comparable to $I$\footnote{More specifically, it is enough to take a parent of a parent of a parent of $I$.}, still denote by $I$, we obtain
% \begin{equation}\label{eq:pf_ST2_2}
\begin{eqnarray*}
\lefteqn{\iint_{\R^2}  \frac{ |\xi\eta|^{\frac{1}{3p-2}} |\hat{f}(\xi)|^{(\frac{3p}{2})'}|\hat{f}(\eta)|^{(\frac{3p}{2})'}}
	{|\xi+\eta|^{\frac{2}{3p-2}}|\xi-\eta|^{\frac{2}{3p-2}}}
	\,d\xi d\eta }\\
	&\le& C
	\sum_{I \in \mathcal{D}} |I|^{-\frac{2}{3p-2}}
	\(\int_{I\cup (-I)} |\hat{f}(\xi)|^{(\frac{3p}{2})'} \,d\xi\)^2  \\
		&\le& C
	\sum_{I \in \mathcal{D}} 
	|I|^{-\frac{2}{3p-2}} 
	\|\hat{f}\|_{L^{(\frac{3p}2)'}(I)}^{\frac{6p}{3p-2}}\ \\
	&=& C \hMn{f}{p}{\frac{3p}2}{2(\frac{3p}2)'}^{2(\frac{3p}2)'},
% 	\(\int_{I} |\hat{f}(\xi)|^{(\frac{3p}{2})'} \,d\xi\)^2 ,
\end{eqnarray*}
% \end{equation}
which completes the proof. 
\end{proof}%$\qed$

\begin{remark}\label{rem:rSTNLS}
Proposition \ref{prop:rSTNLS} can be shown in the same way (see also \cite{BV}).
More precisely, we have
\[
	\norm{ e^{-it\d_x^2} f }_{ L^{3p}_{t,x} }^{2(\frac{3p}2)'}
	\le C\iint_{\R^2}  \frac{  |\hat{f}(\xi)|^{(\frac{3p}{2})'}|\hat{f}(\eta)|^{(\frac{3p}{2})'}}
	{|\xi-\eta|^{\frac{2}{3p-2}}}
	\,d\xi d\eta .
\]
The rest of proof is essentially the same. 
\end{remark}

\vskip3mm
\noindent {\bf Acknowledgments.}
The part of this work was done while the authors were 
visiting at Department of Mathematics at the University of
California, Santa Barbara whose hospitality they gratefully
acknowledge. 
S.M. is partially supported by JSPS, Grant-in-Aid for Young Scientists (B) 24740108.
J.S. is partially supported by JSPS, Strategic Young Researcher Overseas
Visits Program for Accelerating Brain Circulation and by MEXT,
Grant-in-Aid for Young Scientists (A) 25707004.

\end{document}